\documentclass[12pt,final]{article}

\usepackage{amsmath}
\usepackage{amsthm}
\usepackage{amssymb}
\usepackage{amsfonts}
\usepackage{latexsym}               
\usepackage{amsgen}
\usepackage[mathscr]{eucal}
\usepackage{mathrsfs}
\usepackage{enumerate}

\usepackage{mathptmx} 
\usepackage{fancybox}
\usepackage{color}
\usepackage{graphicx}

\usepackage{showkeys}
\usepackage{bm}

\newtheorem{theorem}{Theorem}
\newtheorem{proposition}[theorem]{Proposition}
\newtheorem{lemma}[theorem]{Lemma}
\newtheorem{corollary}[theorem]{Corollary}

\theoremstyle{definition}

\newtheorem{remark}[theorem]{Remark}

\numberwithin{equation}{section}
\numberwithin{theorem}{section}

\def\R{\mathbb{R}}
\def\ep{\varepsilon}
\def\vp{\varphi}

\def\cale{\mathcal{E}}
\def\calp{\mathcal{P}}
\def\calh{\mathcal{H}}
\def\dels{\delta}
\def\ub{u_{\text{\rm b}}}
\def\wc{w_{\text{\rm c}}}

\def\rt{\tilde{\rho}}
\def\ut{\tilde{u}}

\def\div{\mathop{\rm div}}
\def\pd{\partial}

\def\dt{\partial_t}

\def\dtn#1#2{T_{#1,#2}}

\def\lpeasp#1#2{L^{#1}_{\text{\rm e},#2}}

\def\lteasp#1{\lpeasp{2}{#1}}

\def\lp#1#2{\| #2 \|_{L^{#1}}}
\def\lpo#1#2#3{\| #2 \|_{L^{#1}(#3)}}
\def\lt#1{\|#1\|}
\def\li#1{\lp{\infty}{#1}}

\def\ltea#1#2{\| #2 \|_{L^2_{\text{\rm e},#1}}}
\def\hs#1#2{\| #2 \|_{H^{#1}}}

\def\ho#1{\hs{1}{#1}}

\begin{document}

\abovedisplayskip=8pt plus 2pt minus 4pt
\belowdisplayskip=\abovedisplayskip


\thispagestyle{plain}

\title{Stationary flows for compressible viscous fluid \\
in a perturbed half-space} 

\author{%
{\large\sc Masahiro Suzuki${}^1$}
{\normalsize and}
{\large\sc Katherine Zhiyuan Zhang${}^2$}
}

\date{%
\normalsize
${}^1$%
Department of Computer Science and Engineering, 
Nagoya Institute of Technology,
\\
Gokiso-cho, Showa-ku, Nagoya, 466-8555, Japan
\\ [7pt]
${}^2$%
Department of Mathematics, Brown University, 
\\
Providence, RI 02912, USA
}

\maketitle

\begin{abstract}
We consider the compressible Navier--Stokes equation 
in a perturbed half-space with an outflow boundary condition as well as the supersonic condition. 
For a half-space, it has been known that a certain planar stationary solution exists 
and it is time-asymptotically stable.
The planar stationary solution is independent of the tangential directions 
and its velocities of the tangential directions are zero.
In this paper, we show the unique existence of stationary solutions for the perturbed half-space.
The feature of our work is that our stationary solution depends on all directions 
and has multidirectional flow.
Furthermore, we also prove the asymptotic stability of this stationary solution.
\end{abstract}

\begin{description}
\item[{\it Keywords:}]
compressible Navier--Stokes equation, stationary solution, unique existence, asymptotic stability, 
multidirectional flow

\item[{\it 2010 Mathematics Subject Classification:}]
35B35; 
35B40; 
76N10; 
76N15 
\end{description}


\newpage

\section{Introduction}

We consider an asymptotic behavior of a solution to
the compressible Navier--Stokes equation
in a perturbed half-space 
$\Omega := \{ x \in \mathbb{R}^3 : x_1 > M(x_2, x_3), \ M \in H^9 (\mathbb{R}^2)\} $:
\begin{subequations}\label{nse}
\begin{gather}
\rho_t + \div (\rho u) = 0,
\label{nse1}
\\
\rho \{ u_t + (u \cdot \nabla) u \}
= 
\mu_1 \Delta u + (\mu_1 + \mu_2) \nabla (\div u) - \nabla p(\rho).
\label{nse2}
\end{gather}
In this equations, $t>0$ and
$x = (x_1, x_2, x_3) = (x_1, x') \in \Omega$ are the time and space variables, respectively.
The unknown functions $\rho = \rho(t,x)$ and 
$u=u(t,x) = (u_1,u_2,u_3)(t,x)$ stand for
fluid density and fluid velocity, respectively.
The function $p = p(\rho)$ means a pressure explicitly given by 
$p(\rho) := K \rho^\gamma$, where $K > 0$ and $\gamma \ge 1$
are constants.
The constants $\mu_1$ and $\mu_2$ are viscosity coefficients
satisfying $\mu_1 >0$ and $2 \mu_1 + 3 \mu_2 \ge 0$.
We put down an initial data
\begin{gather}
(\rho, u) (0,x)
=
(\rho_0, u_0)(x)
\label{ice}
\end{gather}
and an outflow boundary condition
\begin{equation}
u(t,M(x'),x') = \ub(x')=(u_{b1},u_{b2},u_{b3})(x'), \quad (\ub \cdot n)(x') \geq c >0,
\label{bce}
\end{equation}
\end{subequations}
where $c$ is a positive constant, and $n$ is the unit outer normal vector on $\partial \Omega$, which can be explicitly written as
\begin{equation}\label{nvector}
n(x')=(n_1,n_2,n_3)(x')
:=\left(\frac{-1}{\sqrt{1+|\nabla M|^2}},
\frac{\pd_{x_2}M}{\sqrt{1+|\nabla M|^2}},
\frac{\pd_{x_3}M}{\sqrt{1+|\nabla M|^2}} \right)(x').
\end{equation}

We assume that the end states of the initial data 
in a normal direction $x_1$ are:
\begin{equation}
\lim_{x_1 \to \infty}  \rho_0(x)
= \rho_+,
\quad
\lim_{x_1 \to \infty} u_0(x)
= (u_+,0,0),
\label{cd-ice}
\end{equation}
where $\rho_+$ and $u_+$ are constants.
It is also assumed that the initial density is uniformly positive:
\begin{equation}
\quad
\inf_{x \in \Omega} \rho_0 (x) > 0,
\quad
\rho_+ > 0.
\nonumber
\end{equation}
We will construct solutions of which the density is positive everywhere.
The outflow boundary condition $u_b\cdot n>0$ guarantees that no boundary condition is suitable for \eqref{nse1}.
The compatibility conditions are also necessary for the initial data $u_0$. 
We will mention clearly the conditions in Section \ref{sec2}.


Furthermore, we assume that the Mach number at the end states 
satisfies the supersonic condition:
\begin{equation}\label{super1}
\frac{|u_+|}{\sqrt{p'(\rho_+)}}>1.
\end{equation}


There have been many researches on the initial--boundary value problems of
the compressible Navier--Stokes equation.
We are interested with the long-time behavior of the solutions.
Matsumura and Nishida made the pioneering work \cite{m-n83},
where the initial--boundary value problems over an exterior domain 
and a half-space were studied.
They showed that the time-global solution exists and converges to 
the stationary solution as time tends to infinity by assuming that
the initial perturbation from the stationary solution belongs to $H^m$
and its $H^m$-norm is small enough.
Kagei and Kobayashi \cite{kagei05} gave a deeper analysis for the half-space
in the case that the stationary solution is a constant state.
They obtained an accurate convergence rate of the time global solutions toward the constant state
by assuming the initial perturbation belongs to $H^m\cap L^1$.
However, all these researches adopted the non-slip boundary condition and 
investigated only the case that the velocities of those stationary solutions are zero. 
It is of great interest to consider 
the case when the fluid is flowing in the stationary solutions. 
Valli \cite{Va83} investigated a situation that equation \eqref{nse2} 
have supplementarily an external forcing term, 
and then proved the unique existence and stability of stationary solution
over an bounded domain with the non-slip boundary condition.
The velocity of the stationary solution is nonzero but small,
since he assumed that the forcing term is sufficiently small.
Matsumura \cite{matu-01} gave the classification of 
the possible time-asymptotic states for a one-dimensional half-space problem 
with no external force  
and conjectured that one of time-asymptotic states for an {\it outflow
problem} is a stationary solution of which end state 
satisfying the {\it supersonic condition} \eqref{super1}. 

The outflow problem means an initial--boundary value problem
with an outflow boundary condition \eqref{bce}.
The asymptotic stability of the stationary solution (Matsumura's conjecture) was 
shown by Kawashima, Nishibata and Zhu in \cite{knz03}.
After this Nakamura, Nishibata and Yuge \cite{nny07} proved that 
the convergence rate toward the stationary solution is exponential 
by assuming that the initial perturbation belongs to some weighted Sobolev space.
For a three-dimensional half-space $\mathbb R^3_+$ i.e. the case $M\equiv 0$, 
Kagei and Kawashima \cite{kg06} showed that 
a planar stationary solution is time asymptotically stable,
where the planar solution is a special solution independent of tangential direction $x'$,
and its tangential velocities $(u_2,u_3)$ are zero.
It has been also known in \cite{nn09}
that the convergence rate is exponential if 
the initial perturbation decays exponentially fast at an infinite distance.

The main purpose of the present paper is to extend 
the results in \cite{kg06,nn09}, where the analysis were carried out on the half-space $\mathbb R^3_+$, 
to the case in which the domain $\Omega$ is a perturbed half-space with a curved boundary.
More precisely, we show the unique existence and asymptotic stability of 
stationary solution to \eqref{nse}.
The planar stationary solution studied in \cite{kg06,nn09} is
independent of tangential $x'$ and thus
satisfies a system of ODEs with respect to $x_1$.
The feature of our work is that our stationary solution
depends on all directions $x=(x_1,x')$ and has multidirectional flow.
Few mathematical results have been reported on nonlinear states
having multidirectional flows for compressible fluids.

\begin{flushleft}
\textbf{Acknowledgements. } M. S. was supported by JSPS KAKENHI Grant Number 18K03364. The authors would like to thank Professor Walter A. Strauss for all the support and helpful discussions.
\end{flushleft}

\subsection{Notation} \label{ss-Notation}
We introduce notation used often in this paper.
Let $\pd_i := \frac{\pd}{\pd x_i}$ and $\dt := \frac{\pd}{\pd t}$.
The operators $\nabla := (\pd_1 ,\pd_2,\pd_3)$
and $\Delta := \sum_{i=1}^n \pd_i^2$ denote
standard gradient and Laplacian with 
respect to $x = (x_1,x_2,x_3)$.
We also define a standard divergence by
 $\div u := \nabla \cdot u := \sum_{i=1}^3 \pd_i u_i$.
The operator $\nabla_{x'} := (\pd_2,\pd_3)$ denotes 
tangential gradient with respect to $x'=(x_2,x_3)$.

For a non-negative integer $k$,
we denote by $\nabla^k$ and $\nabla_{x'}^k$ the totality of 
all $k$-th order derivatives with respect to $x$ and $x'$, respectively.
For a domain $\Sigma \subset \R^n$ and $1 \le p \le \infty$,
the space $L^p(\Sigma)$ denotes the standard Lebesgue space
equipped with the norm $\lpo{p}{\cdot}{\Sigma}$.
For a non-negative integer $m$, $H^m = H^m(\Sigma)$ denotes
the $m$-th order Sobolev space over $\Sigma$ in the $L^2$ sense
with the norm $\|\cdot\|_{H^m(\Sigma)}$.
For any $\beta \ge 0$, the space $\lteasp{\beta}(\Sigma)$ denotes
the exponentially weighted $L^2$ space in the normal direction defined by
$\lteasp{\beta}(\Sigma) := \{ u \in L^2(\Sigma)
  \, ; \, \ltea{\beta}{u} < \infty \}$ equipped with the norm %
\[
\ltea{\beta}{u} :=
\Bigl(
\int_{\Omega} e^{\beta x_1} |u(x)|^2 \, dx
\Bigr)^{1/2}.
\]
In the case  $\Sigma = \Omega$, 
the spaces $L^p(\Omega)$, $H^m(\Omega)$, and $\lteasp{\beta}(\Omega)$ 
are sometimes abbreviated by $L^p$, $H^m$, and $\lteasp{\beta}$, respectively. 
Note that $L^2 = H^0 = \lteasp{0}$, and we denote $\lt{\cdot} := \lp{2}{\cdot}$.

We also define the following solution spaces
\begin{align*}
X_m(0,T)
& :=
\{
(\vp,\psi) \in C([0,T] ; H^m)
\; ; \,
\nabla \vp \in L^2([0,T] ; H^{m-1}), \
\nabla \psi \in L^2([0,T] ; H^m)
\},
\\
X^{\text{\rm e}}_{m,\beta} (0,T)
& :=
\{
(\vp,\psi) \in X_m(0,T)
\; ; \,
(\vp,\psi) \in C([0,T] ; \lteasp{\beta}), \
\nabla \psi \in L^2([0,T] ; \lteasp{\beta})
\},
\end{align*}
where $T>0$ and $\beta \ge 0$ are constants. Moreover, we define
\begin{equation*}
\mathcal{X}^{\text{\rm e}}_{m,\beta} (0,T)
=X^{\text{\rm e}}_{m-1,\beta} (0,T)
\cap L^\infty(0,T ; H^m(\Omega)).
\end{equation*}
%

We use $c$ and $C$ to denote generic positive constants 
depending on $\mu_1$, $\mu_2$, $K$, $\gamma$, $\rho_+$, $u_+$, $\|M\|_{H^9(\mathbb R^2)}$ and $\alpha$
but independent of $t$, $\beta$, $\delta$, $\zeta$ and any further information of $\Omega$.
We note that the positive constants $\alpha$, $\beta$, $\delta$ and $\zeta$ 
will be given in the next subsection.
Let us also denote a generic positive constant depending additionally
on other parameters $a$, $b$, $\ldots$ by $C(a,\, b,\, \ldots)$, and a generic positive constant depending additionally
on some further information of $\Omega$ (other than $\|M\|_{H^9(\mathbb R^2)}$) by $C(\Omega)$.  
Furthermore, $A \lesssim B$ means $A \leq C B$ for the generic constant $C$ given above, and $A \lesssim_\Omega B$ means $A \leq C(\Omega) B$ for the generic constant $C(\Omega)$ given above.

\subsection{Main results}
Before mentioning our main results, we introduce a result in \cite{knz03}
which showed the unique existence of planar stationary solutions
$(\rt,\ut)(x_1)=(\rt,\ut_1,0,0)(x_1)$
over the half-space 
$\mathbb R_+^3:=\{x \in \mathbb R^3 \,;\, x_1>0\}$.
The planar stationary solution $(\rt(x_1), \ut_1(x_1))$ solve 
ordinary differential equations
\begin{subequations}
\label{ste}
\begin{gather}
(\rt \ut_1)_{x_1} = 0,
\label{ste1}
\\
(\rt \ut_1^2 + p(\rt))_{x_1} = \mu \ut_{1 x_1 x_1}
\label{ste2}
\end{gather}
with conditions
\begin{equation}
\ut_1 (0) = \tilde{u}_{b},
\quad
\lim_{x_1 \to \infty} (\rt(x_1) ,\ut_1(x_1)) = (\rho_+,u_+),
\quad
\inf_{x_1 \in \R_+} \rt(x_1) > 0.
\label{stbc}
\end{equation}
\end{subequations}
where $\mu$ is a positive constant defined by $\mu:=2\mu_1+\mu_2$.
The following quantity $\tilde{\delta}$ plays an important role in stability analysis.
We call it a boundary strength.
\[
\tilde{\delta} := |\tilde{u}_{b}-u_+|.
\]

\begin{proposition}[\cite{knz03}] \label{ex-st}
Let \eqref{super1} hold. There exists a positive constant $\wc<1$ such that
problem \eqref{ste} has a unique solution $(\rt,\ut_1)$ if and only if 
the following two conditions hold:
\begin{equation}
u_+ < 0
\ \; \text{and} \ \;
\tilde{u}_{b} < \wc u_+.
\label{cd-st}
\end{equation}
Moreover, there exist a positive constant $\alpha$
such that the stationary solution $(\rt,\ut_1)$ satisfies
\begin{equation}
|\pd_{x_1}^k (\rt(x_1) - \rho_+, \ut_1(x_1) - u_+)|
\lesssim \tilde{\dels} e^{- \alpha x_1}
 \ \; \text{for} \ \;
k = 0,1,2,\dots.
\label{stdc1}
\end{equation}
\end{proposition}

From now on we discuss our main results.
We first show the unique existence of stationary solutions 
$(\rho^s,u^s)=(\rho^s,u^s_1,u^s_2,u^s_3)$ over the domain $\Omega$
by regarding $(\rho^s,u^s)(x)$ as a perturbation of $(\rt,\ut)(\tilde{M}(x))$, where
\begin{equation}\label{tM1}
\tilde{M}(x):=x_1-M(x'). 
\end{equation}
The stationary solutions satisfy the equations
\begin{subequations}
\label{snse}
\begin{gather}
\div (\rho^s u^s) = 0,
\label{snse1}
\\
\rho^s \{ (u^s \cdot \nabla) u^s \}
= \mu_1 \Delta u^s + (\mu_1 + \mu_2) \nabla (\div u^s) - \nabla p(\rho^s)
\label{snse2}
\end{gather}
with conditions
\begin{gather}
u^s(t,M(x'),x') = \ub(M(x'),x'),
\\
\lim_{x_1 \to \infty}  \rho^s(x)
= \rho_+,
\quad
\lim_{x_1 \to \infty} u^s(x)
=  (u_+,0,0),
\\
\inf_{x\in\Omega}\rho^s(x)>0.
\end{gather}
\end{subequations}
To state the existence theorem, we use the notation
\[
\dels := \|\ub - (\tilde{u}_b,0,0)\|_{H^{13/2}(\partial \Omega)}+\tilde{\delta}.
\]
and an extension $U(x)$ of $u_b(x')-(\tilde{u}_b,0,0)$, which satisfies 
\begin{subequations}\label{ExBdry0}
\begin{gather}
U(M(x'),x')=u_b(x')-(\tilde{u}_b,0,0),
\label{ExBdry1} \\
U(x_1,x')=0 \quad \text{if $x_1 >M(x')+1$},
\label{ExBdry2} \\
\|U\|_{H^7(\Omega)} \lesssim  \delta.
\label{ExBdry3}
\end{gather}
\end{subequations}
The existence result is summarized in the following theorem.
\begin{theorem}\label{th1}
Let \eqref{super1} and \eqref{cd-st} hold, and $m=3,4,5$.
There exist positive constants $\beta \leq \alpha/2$, 
where $\alpha$ is defined in Proposition \ref{ex-st},
and $\ep_0=\ep_0(\beta,\Omega)$ depending on $\beta$ and $\Omega$
such that if $\dels \le \ep_0$, then
the stationary problem \eqref{snse} has a unique solution
$(\rho^s,u^s)$ that satisfies
\begin{gather*}
(\rho^s-\rt\circ\tilde{M},u^s-\ut\circ\tilde{M}-U) 
\in \lteasp{\beta}(\Omega)\cap H^m(\Omega),
\\
\|(\rho^s-\rt\circ\tilde{M},u^s-\ut\circ\tilde{M}-U)\|_{\lteasp{\beta}}^2
+\|(\rho^s-\rt\circ\tilde{M},u^s-\ut\circ\tilde{M}-U)\|_{H^m}^2 \leq C_0\delta,
\end{gather*}
where $C_0=C_0(\beta,\Omega)$ is a positive constant depending on $\beta$ and $\Omega$.
\end{theorem}
We also state the stability theorem.
\begin{theorem}\label{th2}
Let \eqref{super1} and \eqref{cd-st} hold.
There exist positive constants $\beta \leq \alpha/2$, 
where $\alpha$ is defined in Proposition \ref{ex-st},
and $\ep_0=\ep_0(\beta,\Omega)$ depending on $\beta$ and $\Omega$ such that if 
$\|(\rho_0-\rho^s, u_0-u^s)\|_{\lteasp{\beta}}
+\hs{3}{(\rho_0-\rho^s, u_0-u^s)} + \dels \le \ep_0$
and $(\rho_0,u_0)$ satisfies the compatibility conditions of order 0 and 1,
then the initial-boundary value problem \eqref{nse} has 
a unique time-global solution $(\rho, u)$ such that
$(\rho-\rho^s,u-u^s) \in X^{\text{\rm e}}_{3,\beta} (0,T)$.
Moreover, it holds
\begin{equation*}
\|(\rho-\rho^s,u-u^s)(t)\|_{L^\infty}
\leq C_0 e^{-\zeta t},
\end{equation*}
where $C_0=C_0(\beta,\Omega)$ and $\zeta=\zeta(\beta,\Omega)$ are 
positive constants depending on $\beta$ and $\Omega$
but independent of $\delta$ and $t$.
\end{theorem}

Theorem \ref{th2} requires 
the condition $(\rho_0-\rho^s, u_0-u^s) \in \lteasp{\beta}(\Omega)$.
Without this condition, the following stability theorem holds.

\begin{theorem}\label{th3}
Let \eqref{super1} and \eqref{cd-st} hold.
There exists a positive constant $\ep_0=\ep_0(\beta,\Omega)$ 
depending on $\beta$ and $\Omega$ such that if 
$\hs{3}{(\rho_0-\rho^s, u_0-u^s-U)} + \dels  \le \ep_0$
and $(\rho_0,u_0)$ satisfies the compatibility conditions of order 0 and 1,
then the initial-boundary value problem \eqref{nse} has 
a unique time-global solution $(\rho, u)$ such that
$(\rho-\rho^s,u-u^s) \in X_{3} (0,T)$.
Moreover, 
\begin{equation*}
\|(\rho-\rho^s,u-u^s)(t)\|_{L^\infty}\to 0 \quad \text{as $t\to\infty$}.
\end{equation*}
\end{theorem}


If the domain $\Omega$ is sufficiently flat, in the above theorems,
we can take the constants $\ep_0$, $C_0$, and $\zeta$ independent of $\Omega$.
Namely, the following corollary holds.

\begin{corollary}\label{cor}
Suppose that $\|M\|_{H^s} \leq \kappa$ for $\kappa$ being in Lemma \ref{CattabrigaEst}.
Then Theorems \ref{th1}--\ref{th3} hold with constants $\ep_0$, $C_0$, and $\zeta$
independent of $\Omega$. %
\end{corollary}

\begin{remark}
What interests us most in Theorems \ref{th1}--\ref{th3} is that 
the existence and stability are shown as long as 
the boundary of domain is given by a graph.
In other words, the these theorems allow the boundary has a large curvature.
The works \cite{kg06,nn09} adopted the boundary condition as
$u_b(M(x'),x')=(\tilde{u}_b,0,0)$ for the half-space $\mathbb R^3_+$.
It is clear that our theorems cover this boundary condition as well.

It is also worth to point out that Corollary \ref{cor} can
cover the boundary condition $u_b(M(x'),x')=\tilde{u}_b n(x')$,
where $n(x')$ is the unit outer normal vector given in \eqref{nvector}.
Indeed, if $\varepsilon_0$ is independent of $\Omega$, 
there is no issue to take $n(x')$ depending on $\Omega$ so that
$\|\tilde{u}_b n(x')-(\tilde{u}_b,0,0)\|_{H^{13/2}(\partial\Omega)} 
+ |\tilde{u}_b-u_+|=\delta\leq \varepsilon_0/2$ holds.
This boundary condition seems more reasonable from physical point of view
since it means that the fluid is going out from 
only the normal direction of the boundary. 
\end{remark}

Note that it is hard to directly solve the stationary problem \eqref{snse}. 
This is different from the case when one has $\Omega = \mathbb{R}^3_+$ and 
looks for a planar stationary solution, 
where the planar stationary solution only depends on $x_1$ and 
therefore the system \eqref{snse} reduces to an ODE \eqref{ste}. 
It is also different from the stationary incompressible Navier-Stokes equation, 
in which the system is elliptic. 
Our stationary equations are not categorized as elliptic equations.
To get around this difficulty, 
we first prove the existence of a time-global solution to the problem \eqref{nse}, 
and then we construct a stationary solution making use of this time-global solution. 

Let us explain the idea to construct the time-global solution.
We use a continuous argument combining time-local solvability and an a priori estimate.
Then the derivation of a priori estimate is most important.
For example, one can have a priori estimates of solutions of some inhomogeneous parabolic equations 
over bounded domains even if the long-time behavior of solutions is not anticipated. 
The key of the proof is the dissipative structure 
which makes solutions of the corresponding homogeneous equations 
decay exponentially fast as time tends to infinity.
On the other hand, we expect from the stability theorem in \cite{nn09} introduced above
that the solution $(\rho,u)$ to problem \eqref{nse} with $u_b=(u_+,0,0)$ 
may converge the constant state $(\rho_+,u_b)$ exponentially fast as time tends to infinity.
For the case $u_b\neq(u_+,0,0)$, after suitable reformulation, 
all effects coming from $u_b\neq(u_+,0,0)$ are represented by inhomogeneous terms in the equations. 
Specifically, let us define a perturbation as
$$\Phi (t, x):=(\vp,\psi)(t,x) :=(\rho, u)(t,x)-(\rt, \ut)(\tilde{M}(x)) - (0, U) (x)$$ 
and reformulate \eqref{nse} into a problem for $\Phi$. 
The dissipative structure then enables us to obtain 
the a priori estimate of solutions $\Phi$ to the reformulated problem. 
For the construction of stationary solutions, we use a similar method as in \cite{Va83}.
More precisely, we define the translated time-global solutions $\Phi^k(t,x):=\Phi(t+kT^*,x)$ 
for any $T^* > 0$ and $k=1,2,3,\ldots$. 
Then we prove that the sequence $\{\Phi^k\}$ converges to 
a certain time-periodic solution with a period $T^*$.
After this we show by using the uniqueness of time-periodic solution and the arbitrariness of $T^*$ 
that the time-periodic solution is actually time-independent.
Therefore this gives a stationary solution to our problem. 



Before closing this section, we mention the outline of this paper.
In Section \ref{sec2}, we reformulate the initial-boundary value problem \eqref{nse} 
into an initial-boundary value problem for a perturbation 
from the stationary solution $(\rt\circ\tilde{M},\ut\circ\tilde{M})$ in the half-space, 
as stated in \eqref{eq-pv}. 
In Section \ref{sec3}, we show the unique existence of the time-global solution 
to the reformulated problem \eqref{eq-pv} (see Theorem \ref{global1}) 
by proving an a priori estimate in Proposition \ref{apriori1}.
The derivation of the a priori estimate is based 
on a combination of the energy form in \cite{kagei05,knz03}, 
the Matsumura--Nishida energy method in \cite{m-n83},
and the weighted energy method in \cite{nny07,nn09}.
In Section \ref{S5} we construct stationary solutions by the method mentioned just above.
Subsection \ref{S5.3} is devoted to the proof of the asymptotic stability of the stationary solution 
in the weighted space $L^2_{e, \beta} (\Omega)$. 
Here we can obtain the exponential convergence rate.
For the initial data which do not belong to $L^2_{e, \beta} (\Omega)$,
we also show the asymptotic stability of the stationary solution in Section \ref{S6}. 
In Appendix A, we give the proofs of some general inequalities. 
Furthermore, we construct an initial data satisfying the compatibility conditions in Appendix B.
The initial data is necessary to obtain the time-global solution in Section \ref{sec3}.

\section{Reformulation}\label{sec2}
For the proof of Theorems \ref{th1} and \ref{th2},
we begin by reformulating initial-boundary value problem \eqref{nse}. 
Let us introduce perturbations
\begin{gather*}
(\vp,\psi)(t,x)
:=
(\rho, u)(t,x)
-
(\rt, \ut)(\tilde{M}(x)) - (0, U) (x),
\ \; \text{where} \ \;
\psi = (\psi_1,\psi_2,\psi_3). 
\end{gather*}
Here $\tilde{M}(x)$ is defined in \eqref{tM1}.

Owing to equations in (\ref{nse}) and (\ref{ste}),
the perturbation $(\vp,\psi)$ satisfies the system of
equations
\begin{subequations}
\label{eq-pv}
\begin{gather}
\vp_t + u \cdot \nabla \vp + \rho \div \psi
= f+F,
\label{eq-pv1}
\\
\rho \{ \psi_t + (u \cdot \nabla) \psi \}
- L \psi + p'(\rho) \nabla \vp
= g + G.
\label{eq-pv2}
\end{gather}
The boundary and initial conditions for $(\vp,\psi)$ 
follow from (\ref{ice}), (\ref{bce}), and (\ref{stbc}) as
\begin{gather}
\psi(t,M(x'),x') = 0,
\label{pbc}
\\
(\vp,\psi)(0,x)
=(\vp_0, \psi_0)(x)
:= (\rho_0, u_0)(x) - (\rt, \ut)(\tilde{M}(x)) - (0, U) (x).
\label{pic}
\end{gather}
\end{subequations}
Here $L \psi$, $f$, $F$, $g$ and $G$ are defined by
\begin{align*}
L \psi &:= \mu_1 \Delta \psi + (\mu_1 + \mu_2) \nabla \div \psi,
\\
f &:= - \nabla\tilde{\rho} \cdot \psi  -  \tilde{u}_1'\varphi - \vp \div U ,
\\
F &:=  - \nabla \tilde{\rho} \cdot U - \tilde{\rho} \div U,
\\
g&:= - \rho (\psi \cdot \nabla) (\tilde{u}+U) - \varphi ((\tilde{u}+U) \cdot \nabla) (\tilde{u}+U)
-(p'(\rho)-p'(\rt)) \nabla \tilde{\rho},
\\
G&:=-\rt ((\tilde{u}+U) \cdot \nabla) U-\rt (U \cdot \nabla)\tilde{u} + LU + p'(\rt) \tilde{\rho}' \nabla M
\\
&\qquad +\begin{bmatrix}
\mu_1 \tilde{u}''_1 \sum^3_{j=2} (\pd_{j} M)^2 + \mu_1 \tilde{u}'_1 \sum^3_{j=2}  \pd^2_{j} M
\\
-(\mu_1+\mu_2) \tilde{u}''_1 \pd_{2} M
\\
-(\mu_1+\mu_2) \tilde{u}''_1 \pd_{3} M
\end{bmatrix}.
\end{align*}
Note that $L$ is a differential operator; $f$ and $g$ are homogeneous terms for $(\phi,\psi)$; 
$F$ and $G$ are inhomogeneous terms independent of $t$.
Furthermore, $F$ and $G$
can be estimated by using $M \in H^9(\mathbb R^2)$, \eqref{stdc1}, and \eqref{ExBdry0} as
\begin{equation}\label{h1}
\| (F,G) \|_{L^2_{e, 3\alpha/2}} \lesssim  \delta, \quad 
\|(F,G)\|_{H^5} \lesssim \delta.
\end{equation}


We often express the perturbation by
\[
\Phi := (\vp,\psi),
\quad
\Phi_0 := (\vp_0,\psi_0).
\]
In order to establish the local existence of the solution
in strong sense,
we assume compatibility conditions for the initial data.
It is necessary to assume the compatibility conditions of order 0, 1, and 2:
\begin{subequations}\label{cmpa0}
\begin{gather} 
\psi_0|_{x_1=M(x')} = 0,
\quad
\left\{
\rho_0 (u_0 \cdot \nabla) \psi_0
- L \psi_0
+ p'(\rho_0) \nabla \vp_0
- (g+G)|_{t=0}
\right\}|_{x_1 = M(x')}
= 0,
\label{cmpa3} \\
\left[\partial_t\left\{
\rho (u \cdot  \nabla) \psi
- L \psi
+ p'( \rho) \nabla \vp
- g
\right\}|_{t=0}\right]_{x_1 = M(x')}
= 0.
\label{cmpa2}
\end{gather}
\end{subequations}
Note that the equation \eqref{cmpa2} (which is of order 2) can be written into a form which only contains spatial-derivatives of the initial data by using \eqref{eq-pv} (for more details, see the proof of Lemma \ref{CompatibilityCond} in Appendix B).

It suffices to show Theorems \ref{th4}--\ref{th5} and Corollary \ref{cor2} below for the completion of the
proof of Theorems \ref{th1}--\ref{th2} and the claims corresponding to Theorems \ref{th1}--\ref{th2} in Corollary \ref{cor}, respectively. 

\begin{theorem}\label{th4}
Let \eqref{super1} and \eqref{cd-st} hold, and $m=3,4,5$.
There exist positive constants $\beta \leq \alpha/2$, 
where $\alpha$ is defined in Proposition \ref{ex-st},
and $\ep_0=\ep_0(\beta,\Omega)$ depending on $\beta$ and $\Omega$
such that if $\dels \le \ep_0$, 
the stationary problem corresponding to \eqref{eq-pv} has a unique solution
$\Phi^s \in \lteasp{\beta}(\Omega)\cap H^m(\Omega)$ with
\begin{gather*}
\|\Phi^s\|_{\lteasp{\beta}}^2+\|\Phi^s\|_{H^m}^2\leq C_0 \delta,
\end{gather*}
where $C_0=C_0(\beta,\Omega)$ is a positive constant depending on $\beta$ and $\Omega$.
\end{theorem}

\begin{theorem}\label{th5}
Let \eqref{super1} and \eqref{cd-st} hold.
There exist positive constants $\beta \leq \alpha/2$, 
where $\alpha$ is defined in Proposition \ref{ex-st},
and $\ep_0=\ep_0(\beta,\Omega)$ depending on $\beta$ and $\Omega$
such that if 
$\|\Phi_0-\Phi^s\|_{\lteasp{\beta}}+\hs{3}{\Phi_0-\Phi^s} + \dels \le \ep_0$
and $\Phi_0$ satisfies the compatibility condition \eqref{cmpa3} for $m=3,4$,
\eqref{cmpa0} for $m=5$,
then the initial-boundary value problem \eqref{eq-pv} has 
a unique time-global solution 
$\Phi \in X^{\text{\rm e}}_{m,\beta} (0,\infty)$.
Moreover, it satisfies
\begin{equation*}
\|(\Phi-\Phi^s)(t)\|_{L^\infty}\leq C_0 e^{-\zeta t},
\end{equation*}
where $C_0=C_0(\beta,\Omega)$ and $\zeta=\zeta(\beta,\Omega)$
are positive constant depending on $\beta$ and $\Omega$.
\end{theorem}

\begin{corollary}\label{cor2}
Suppose that $\|M\|_{H^s} \leq \kappa$ for $\kappa$ being in Lemma \ref{CattabrigaEst}.
Then Theorems \ref{th4} and \ref{th5} hold with constants $\ep_0$, $C_0$, and $\zeta$
independent of $\Omega$. 
\end{corollary}

\section{Time-global solvability}\label{sec3}

This section provides the time-global solvability
of initial--boundary value problem \eqref{eq-pv}.

\begin{theorem}\label{global1}
Let \eqref{super1} and \eqref{cd-st} hold, and $m=3,4,5$.
There exist positive constants $\beta \leq \alpha/2$, 
where $\alpha$ is defined in Proposition \ref{ex-st},
and $\ep_0=\ep_0(\beta,\Omega)$ depending on $\beta$ and $\Omega$ such that if 
$\|\Phi_0\|_{\lteasp{\beta}}+\hs{m}{\Phi_0} + \dels \le \ep_0$
and $\Phi_0$ satisfies the compatibility condition \eqref{cmpa3} for $m=3,4$,
\eqref{cmpa0} for $m=5$,
then the initial-boundary value problem \eqref{eq-pv} has 
a unique time-global solution 
$\Phi \in X^{\text{\rm e}}_{m,\beta} (0,T)$.
Moreover, it satisfies
\begin{equation}\label{bound1}
\|\Phi(t)\|_{\lteasp{\beta}}^2+\|\Phi(t)\|_{H^m}^2+\|\pd_t \Phi(t)\|_{H^{m-2}}^2
\\
\leq C_0(\|\Phi_0\|_{\lteasp{\beta}}^2+\hs{m}{\Phi_0}^2)e^{-\zeta t} + C_0\dels,
\quad t \in [0,\infty),
\end{equation}
where $C_0=C_0(\beta,\Omega)$ and $\zeta=\zeta(\beta,\Omega)$
are positive constant depending on $\beta$ and $\Omega$
but independent of $\delta$ and $t$.
\end{theorem}

The time-global solution $\Phi$ with \eqref{bound1} can be constructed 
by a standard continuation argument (see \cite{m-n83})
using the time-local solvability in Lemma \ref{local1} 
and the a priori estimate in Proposition \ref{apriori1} below.

\begin{lemma}\label{local1}
Let $m=3,4,5$. 
Suppose that the initial data $\Phi_0 \in H^m(\Omega)$ satisfies
the compatibility condition \eqref{cmpa3} for $m=3,4$,
\eqref{cmpa0} for $m=5$.
Then there exists a positive constant $T$ depending on 
$\hs{m}{\Phi_0}$ such that
initial-boundary value problem \eqref{eq-pv}
has a unique solution $\Phi \in X(0,T)$.
Moreover, if the initial data satisfies
$\Phi_0 \in \lteasp{\beta} (\Omega)$,
it holds $\Phi \in X_\beta^{\text{\rm e}}(0,T)$.
\end{lemma}

For the notational convenience, 
we define a norm $E_{m,\beta}(t)$ and a dissipative norm $D_{m,\beta}(t)$ by
\begin{align} 
E_{m,\beta}(t) &:=\|\Phi\|_{\lteasp{\beta}}^2+\|\Phi\|_{H^m}^2 \quad \text{for $m\geq 0$,}  \label{Ekbeta-def}
\\
D_{m,\beta}(t)
&:=\left\{\begin{array}{ll}
\beta\|\Phi\|_{\lteasp{\beta}}^2
+ \|\nabla\psi\|_{\lteasp{\beta}}^2
+ \|{\frac{d}{dt}} {\vp} \|^2 
+ \|\vp(\tau,M(\cdot),\cdot)\|_{L^2(\mathbb R^2)}^2 & \text{if $m=0$},
\\
\displaystyle
D_{0,\beta}(t)^2+\|(\nabla\Phi,\nabla^2\psi)\|_{H^{m-1}}^2
+\sum_{i=1}^{[(m+1)/2]} \lt{\pd_t^{i} \psi}^2_{H^{m+1-2i}}
+\left\|{\frac{d}{dt}} {\vp} \right\|_{H^{m}}^2 & \text{if $m\geq 1$}.
\end{array}\right.  \label{Dkbeta-def}
\end{align}
Furthermore, we also use
\[
N_{\beta}(T):= \sup_{t\in[0,T]}(\|\Phi(t)\|_{\lteasp{\beta}}+\|\Phi(t)\|_{H^3}).
\]

\begin{proposition}\label{apriori1}
Let \eqref{super1} and \eqref{cd-st} hold, and $m=3,4,5$.
Suppose that $\Phi \in X^{\text{\rm e}}_{m,\beta} (0,T)$
be a solution to initial-boundary value problem \eqref{eq-pv}
for some positive constant $T$.
Then there exist positive constants $\beta \leq \alpha/2$, 
where $\alpha$ is defined in Proposition \ref{ex-st},
and $\ep_0=\ep_0(\beta,\Omega)$ depending on $\beta$ and $\Omega$ such that if 
$\sup_{t\in[0,T]}(\|\Phi(t)\|_{\lteasp{\beta}}+\|\Phi(t)\|_{H^m}) + \dels \le \ep_0$, 
the following estimate holds:
\begin{gather}
e^{\zeta t} E_{m, \beta} (t) + \int^t_0 e^{\zeta \tau} D_{m,\beta} (\tau ) \, d \tau
\leq C_0(\|\Phi_0\|_{\lteasp{\beta}}^2 + \| \Phi_0\|^2_{H^3})
+ C_0 \delta e^{\zeta \tau},
\label{apes0}\\
\|\Phi(t)\|_{\lteasp{\beta}}^2+\|\Phi(t)\|_{H^m}^2+\|\pd_t \Phi(t)\|_{H^{m-2}}^2
\leq C_0(\|\Phi_0\|_{\lteasp{\beta}}^2+\hs{m}{\Phi_0}^2)e^{-\zeta t} + C_0\dels
\label{apes1}
\end{gather}
for $t \in [0,T]$, where $C_0=C_0(\beta,\Omega)$ and $\zeta=\zeta(\beta,\Omega)$
are positive constant depending on $\beta$ and $\Omega$
but independent of $\delta$ and $t$.
\end{proposition}


\begin{corollary}\label{cor1}
Suppose that $\|M\|_{H^s} \leq \kappa$ for $\kappa$ being in Lemma \ref{CattabrigaEst}.
Then Theorems \ref{global1} and Proposition \ref{apriori1} hold 
with constants $\ep_0$, $C_0$, and $\zeta$
independent of $\Omega$. 
\end{corollary}

Lemma \ref{local1} can be proved in much the same way as in \cite{kg06-loc}.
Therefore, we omit the proof. 
In the remainder of this section, we prove Proposition \ref{apriori1} 
only for the case $m=3$, since the case $m=4,5$ can be shown similarly.
We derive the $L^2$-norm of $\Phi$ by following the method in \cite{kg06,knz03,nn09}.
To estimate the derivatives of $\Phi$, 
we use essentially the Matsumura--Nishida energy method in \cite{m-n83}.

\subsection{$L^2$ estimate} \label{ss-L2}

This subsection is devoted to the derivation of the 
estimate of the perturbation $(\vp,\psi)$ in $\lteasp{\beta}(\Omega)$.
To do this, we introduce an energy form $\cale$,
similarly as in \cite{kg06,knz03,nn09}, by
\[
\cale := \int^\rho_{\rt} \frac{p(\eta) - p(\rt)}{\eta^2} d \eta  + \frac{1}{2} |\psi|^2
=K \rt^{\gamma-1} \omega \Bigl( \frac{\rt}{\rho} \Bigr) + \frac{1}{2} |\psi|^2,
\quad
\omega(r) := r - 1 - \int_1^r \eta^{-\gamma} \, d \eta.
\]
%
%
Under the smallness assumption on $N_\beta (T)$, 
we have $\li{\Phi(t)} \ll 1$ by Sobolev's inequality \eqref{sobolev2}.
Hence, the energy form $\cale$ is equivalent to the 
square of the perturbation $(\vp,\psi)$:
\begin{equation}
c (\vp^2 + |\psi|^2)
\le
\cale
\le
C (\vp^2 + |\psi|^2).
\label{sqr}
\end{equation}
Moreover we have the uniform bounds of solutions as follows:
\begin{equation}
0 < c \le \rho(t,x) \le C,
\quad
|u(t,x)| \le C,
\label{bdd}
\end{equation}
where we have used 
$N_\beta(T)+\dels \ll 1$.
Using the time and space weighted energy method,
we obtain the energy inequality in $L^2$ framework.

\begin{lemma}
\label{lm1}
Under the same conditions as in Proposition \ref{apriori1} with $m=3$, 
it holds that
\begin{multline}
e^{\zeta t} \|\Phi(t)\|_{\lteasp{\beta}}^2
+ \int_0^t e^{\zeta \tau} D_{0,\beta}^2(\tau) \, d \tau
\\
\lesssim \|\Phi_0\|_{\lteasp{\beta}}^2
+  \zeta \int_0^t e^{\zeta \tau}\|\Phi(\tau)\|_{\lteasp{\beta}}^2 \, d \tau
+  \dels \int_0^t e^{\zeta \tau} \lt{\nabla \vp(\tau)}^2 \, d \tau
+  \dels \int_0^t e^{\zeta \tau} d \tau
\label{ea0}
\end{multline}
for $t \in [0,T]$ and $\zeta>0$.
\end{lemma}

\begin{proof}
Following the computation in \cite{kg06,knz03}, we see that
the energy form $\cale$ satisfies
\begin{equation}
(\rho \cale)_t
- \div (G_1 + B_1)
+ \mu_1 |\nabla \psi|^2
+ (\mu_1 + \mu_2) (\div \psi)^2
=
R_{11},
\label{ea1}
\end{equation}
where 
\begin{align*}
G_1
&:=
- \rho u \cale
- (p(\rho) - p(\rt)) \psi,
\nonumber
\\
B_1
&:=
\mu_1 \nabla \psi \cdot \psi
+ (\mu_1 + \mu_2) \psi \div \psi,
\nonumber
\\
R_{11}
&:=
-\rho (\psi \cdot \nabla) (\tilde{u}+U) \cdot \psi - ( p(\rho) - p(\rt) -p'(\rt) \vp) \div \tilde{u} - \frac{\vp}{\rt} L\tilde{u} \cdot \psi + G \cdot \psi 
\\
& \quad 
- \vp(U\cdot\nabla)\ut\cdot\psi - \vp ((\ut+U) \cdot \nabla ) U \cdot \psi 
 -(p(\rho) - p(\rt) ) \div U  - p'(\rt) \vp \frac{\nabla \rt}{\rt} \cdot U . 
\nonumber 
\end{align*}
Multiplying (\ref{ea1}) by a weight function 
$w = w(x_1,t) := e^{\beta x_1} e^{\zeta t}$, we get
\begin{multline}
(w \rho \cale)_t
- \div
\bigl\{
w (G_1 + B_1)
\bigr\}
+ \nabla w \cdot G_1
+ \mu_1 w |\nabla \psi|^2
+ (\mu_1 + \mu_2) w (\div \psi)^2
\\
=
w_t \rho \cale
- \nabla w \cdot B_1
+ w R_{11}.
\label{ea2}
\end{multline}

We integrate this equality over $\Omega$. 
The second term on the left hand side is estimated from below by
using the divergence theorem with \eqref{bce} and \eqref{pbc} 
as well as (\ref{sqr}) and (\ref{bdd}):
\begin{equation}
- \int_{\Omega}  \div \bigl\{ w (G_1 + B_1)\bigr\} \, dx  
 =  \int_{\pd \Omega}
(w \rho  \cale)(u_b\cdot n)  \, d \sigma \,
\gtrsim
e^{\zeta t}   \|\vp|_{\partial \Omega} \|^2_{L^2_{x'}}. 
\label{ea3}
\end{equation}
Next we derive the lower estimate of the third term on the
left hand side of (\ref{ea2}).
Taking the fact that
 $\omega(s) = \frac{\gamma}{2} (s-1)^2 + O(|s-1|^3)$ for $|s-1| \ll 1$
into account,
we compute the term $\rho u_1 \cale$  in $G_1$ as
\begin{gather}
\rho u_1 \cale
=
\frac{K \gamma \rho_+^{\gamma-2} u_+}{2} \vp^2
+ \frac{\rho_+ u_+}{2} |\psi|^2
+ R_{12},
\label{ea6}
\\
\begin{aligned}
R_{12}
:=
(\rho u_1 - \rho_+ u_+) \cale
+ \rho_+ u_+
\left[
\frac{K \gamma }{2}
\Bigl(
\frac{\rt^{\gamma-1}}{\rho^2} - \rho_+^{\gamma-3}
\Bigr)
\vp^2
+ K \rt^{\gamma-1}
\Bigl\{
\omega \Bigl( \frac{\rt}{\rho} \Bigr)
- \frac{\gamma}{2} \Bigl( \frac{\rt}{\rho} - 1 \Bigr)^2
\Bigr\}
\right].
\end{aligned}
\nonumber
\end{gather}
Also, the second term appeared in $G_1$ is computed as
\begin{gather}
(p(\rho) - p(\rt)) \psi_1
=
p'(\rho_+) \vp \psi_1
+ R_{13},
\label{ea7}
\\
R_{13}
:=
(p'(\rt) - p'(\rho_+)) \vp \psi_1
+ (p(\rho) - p(\rt) - p'(\rt) \vp) \psi_1.
\nonumber
\end{gather}
Thus, using (\ref{ea6}) and (\ref{ea7}),
the third term in \eqref{ea2} 
is rewritten as
\begin{gather*}
\nabla w \cdot G_1
=
w_{x_1}
\Bigl(
F(\vp,\psi_1) + \frac{\rho_+ |u_+|}{2} |\psi'|^2
 + R_{12} + R_{13}
\Bigr),
\\
F(\vp,\psi_1)
:=
\frac{K \gamma \rho_+^{\gamma-2} |u_+|}{2} \vp^2
- p'(\rho_+) \vp \psi_1
+ \frac{\rho_+ |u_+|}{2} \psi_1^2,
\nonumber
\end{gather*}
where $\psi'$ is the second and third components of $\psi$ defined by 
$\psi' := (\psi_2,\psi_3)$.
Owing to the supersonic condition \eqref{super1}, 
the quadratic form $F(\vp,\psi_1)$ becomes
positive definite since the discriminant of $F(\vp,\psi_1)$ satisfies
\[
p'(\rho_+)^2 - K \gamma \rho_+^{\gamma-1} u_+^2
= p'(\rho_+)^2 \left(1 - \frac{u_+^2}{p'(\rho_+)}\right) < 0.
\]
On the other hand,
the remaining terms $R_{12}$ and $R_{13}$ satisfy
\begin{equation*}
|R_{12} + R_{13}| \lesssim
|(\rt - \rho_+,\ut - u_+)| |\Phi|^2 + |\Phi|^3
\lesssim  (N_\beta(t) + \dels) |\Phi|^2.
\end{equation*}
Therefore we obtain the following lower bound of the integration of
 $\nabla w \cdot G_1$ 
\begin{equation}
\int_{\Omega} \nabla w \cdot G_1 \, dx
\geq \beta e^{\zeta t} 
\bigl\{
 c - C (N_\beta(t) + \dels)
\bigr\}
\| \Phi \|^2_{L^2_{e, \beta} (\Omega)} .
\label{ea10}
\end{equation} 
The first and the second terms on the right hand side of (\ref{ea2})
are estimated by using (\ref{sqr}), (\ref{bdd}), and the
Schwarz inequality as
\begin{gather}
\int_{\Omega}
|w_t \rho \cale| \, dx
\lesssim  \zeta e^{\zeta t} \| \Phi \|^2_{L^2_{e, \beta}},  
\label{ea4}
\\
\int_{\Omega}
|\nabla w \cdot B_1| \, dx
\lesssim
\beta e^{\zeta t}
\big(\ep \| \psi \|^2_{L^2_{e, \beta}} + \ep^{-1}  \| \nabla \psi \|^2_{L^2_{e, \beta}} \big),
\label{ea5}
\end{gather}
where $\ep$ is an arbitrary positive constant.
For the term involving $R_{11}$, we observe 
\begin{equation*}
\begin{split}
|R_{11}| 
& \lesssim |(\nabla \tilde{u}, \nabla^2 \tilde{u},\nabla U)||\Phi|^2 +|\Phi||(G,U,\nabla U)|\\
&  \lesssim |(\nabla \tilde{u}, \nabla^2 \tilde{u},\nabla U)||\Phi|^2
+\delta e^{-\alpha x_1}|\Phi|^2
+\delta^{-1}e^{\alpha x_1}|(G,U,\nabla U)|^2. 
\end{split}
\end{equation*}
We apply Hardy's inequality \eqref{hardy} to the first two terms
with $\beta \leq \alpha/2$, \eqref{stdc1}, and \eqref{ExBdry0}, 
and estimate the last term by \eqref{h1}, it then holds that
\begin{equation}
\int_\Omega w |R_{11} | dx 
\lesssim e^{\zeta t} \delta ( \| \nabla \Phi \|^2 +  \|\vp|_{\partial \Omega} \|^2_{L^2_{x'}} +1).
\label{ea9}
\end{equation}

We integrate (\ref{ea2}) over $(0,t) \times \Omega$,
substitute the estimates (\ref{ea3}) and (\ref{ea10})--(\ref{ea9}) into the resultant equality
and then let $\ep$, \footnote{Hereafter we fix this $\beta$ in our whole proof.}{$\beta$}, and $N_\beta(T)+\dels$ be suitably small.
Furthermore, we use the fact that
\begin{align*}
\left\| \frac{d}{dt} \vp  \right\|^2 
= \| \rho \div \psi + f + F\|^2 
\lesssim \| \nabla \psi\|^2+\delta (\| \nabla \vp\|^2+ \|\vp|_{\partial \Omega} \|^2_{L^2_{x'}}+1),
\end{align*}
which follows from \eqref{eq-pv1}, \eqref{stdc1}, \eqref{ExBdry0}, and \eqref{hardy}.
These computations yield the desired inequality.
\end{proof}

\subsection{Time-derivative estimates}

In this section we derive time-derivative estimates.
To this end, by applying the differential operator $\partial_t^k$ for $k=0,1$ to 
(\ref{eq-pv1}) and (\ref{eq-pv2}), we have the following
two equations:
\begin{gather}
\partial_t^k \vp_t
+ u \cdot \nabla \partial_t^k \vp
+ \rho \div \partial_t^k \psi
= f_{0,k},
\label{ec1}
\\
\rho \{
\partial_t^k \psi_t
+ (u \cdot \nabla) \partial_t^k \psi
\}
- L(\partial_t^k \psi)
+ p'(\rho) \nabla \partial_t^k \vp
= g_{0,k},
\label{ec2}
\end{gather}
where
\begin{align*}
f_{0,k}
&:=
\partial_t^k (f + F)
- [\partial_t^k, u] \nabla \vp
- [\partial_t^k, \rho] \div \psi,
\nonumber
\\
g_{0,k}
&:= \partial_t^k (g + G)
- [\partial_t^k, \rho] \psi_t
- [\partial_t^k, \rho u] \nabla \psi
- [\partial_t^k, p'(\rho)] \nabla \vp,
\nonumber
\end{align*}
where $[T,u]v := T(uv) - u T v$ is a commutator.
We also often use the two inequalities:
\begin{gather}
\|\partial_t \Phi\|_{H^1} \lesssim  \sqrt{D_{2,\beta}} + \delta ,
\label{PhiH1}
\\
\|\partial_t\vp\|_{L^\infty}\lesssim  \sqrt{E_{3,\beta}} + \delta \lesssim  N_\beta(T) + \delta.
\label{PhiSup}
\end{gather}
Indeed we can derive these from \eqref{eq-pv1}
by using \eqref{stdc1}, \eqref{ExBdry0}, \eqref{h1}, and Hardy's inequality \eqref{hardy}.

We first estimate $\pd_t \Phi$ in the next lemma. 
\begin{lemma}
\label{lm2}
Under the same conditions as in Proposition \ref{apriori1} with $m=3$, 
it holds that
\begin{multline}
e^{\zeta t} \|\pd_t \Phi(t)\|^2
+ \int_0^t e^{\zeta \tau} 
\|\pd_t \nabla\psi(\tau)\|^2
\, d \tau
\\
\lesssim \|\Phi_0\|_{H^3}^2
+ (N_\beta(T)+\dels+\zeta) \int_0^t e^{\zeta \tau} D_{3,\beta}(\tau) \, d \tau
+ \dels \int_0^t e^{\zeta \tau} \tau
\label{ec0}
\end{multline}
for $t \in [0,T]$ and $\zeta>0$, 
where $C$ is a positive constant 
independent of $\delta$, $t$, and $\zeta$.
\end{lemma}
\begin{proof}
Multiplying \eqref{ec1} with $k=1$ by $P(\rho) \, \partial_t \vp$, where
$P(\rho) := p'(\rho) / \rho$, and using the facts that 
$\hat{\rho}_t = - \hat{\div} (\hat{\rho} \hat{u})$ 
and $ - P' (\hat{\rho}) \hat{\rho} + P(\hat{\rho}) = (3 - \gamma) P(\hat{\rho})$,
we get
\begin{multline}
\Bigl(
\frac{1}{2} P(\rho) | \partial_t \vp |^2
\Bigr)_t
+ \div \Bigl(
\frac{1}{2} P(\rho) u
|\partial_t \vp|^2
\Bigr)
+ p'(\rho) \div (\partial_t \psi) \partial_t \vp
\\
=
P(\rho) f_{0,k} \, \partial_t \vp
+
\frac{3-\gamma}{2} P(\rho) \div u \, |\partial_t \vp|^2.
\label{ec3}
\end{multline}
Multiply (\ref{ec2}) by $\partial_t \psi$ successively to get
\begin{gather}
\Bigl(
\frac{1}{2} \rho |\partial_t \psi|^2
\Bigr)_t
+ \div B_2
+ \mu_1 |\nabla (\partial_t \psi)|^2
+ (\mu_1 + \mu_2) |\div (\partial_t \psi)|^2
\mspace{240mu}
\notag
\\
\mspace{240mu}
= p'(\rho) \div (\partial_t \psi) \partial_t \vp
+ (g_{0,k} + p''(\rho) \, \partial_t \vp \nabla \rho)
   \cdot \partial_t \psi,
\label{ec4}
\\
B_2 :=
\frac{1}{2} \rho u |\partial_t \psi|^2
- \mu_1 \nabla (\partial_t \psi)  \cdot \partial_t \psi
- (\mu_1+\mu_2) \div (\partial_t \psi) \partial_t \psi 
+ p'(\rho) \partial_t \vp \, \partial_t \psi.
\nonumber
\end{gather}
Adding (\ref{ec3}) to (\ref{ec4}) yields
\begin{gather}
\Bigl(
\frac{1}{2} P(\rho) | \partial_t \vp |^2
+ \frac{1}{2} \rho |\partial_t \psi|^2
\Bigr)_t
+ \div \Bigl(
\frac{1}{2} P(\rho) u
|\partial_t \vp|^2
+ B_2
\Bigr)
\mspace{150mu}
\nonumber
\\
\mspace{200mu}
{}+ \mu_1 |\nabla (\partial_t \psi)|^2
+ (\mu_1 + \mu_2) |\div (\partial_t \psi)|^2
=
R_2,
\label{ec5}
\\
R_2
:=
P(\rho) f_{0,k} \, \partial_t \vp
+
\frac{3-\gamma}{2} P(\rho) \div u \, |\partial_t \vp|^2
+ (g_{0,k} + p''(\rho) \, \partial_t \vp \nabla \rho)
   \cdot \partial_t \psi.
\nonumber
\end{gather}
Owing to \eqref{bce} and (\ref{bdd}),
we have the nonnegativity of the second term on the
left hand side of (\ref{ec5}) as
\begin{gather}
\int_{\Omega}
\div \Bigl(
\frac{P(\rho) u}{2}
|\partial_t \vp|^2
+ B_2
\Bigr) \, dx
=\int_{\partial\Omega}
\frac{P(\rho) }{2}
|\partial_t \vp|^2 (u\cdot n)
\, d\sigma
\geq 0.
\label{ec6}
\end{gather}
Notice that here we used $\partial_t \psi =0$ on $\partial \Omega$, which holds because of \eqref{pbc}. By \eqref{PhiSup} and Sobolev inequality \eqref{sobolev2},
the nonlinear term $R_2$ is estimated as
\begin{equation}
|R_2|  \lesssim (N_{\beta}(T) + \dels)|(\Phi_t,\nabla\Phi)|^2.
\label{ec7}
\end{equation}

Now one can have the desired inequality (\ref{ec0}) as follows.
Multiply (\ref{ec5}) by a time weight function $e^{\zeta t}$, 
integrate the resultant equality over $(0,t) \times \Omega$,
and substitute in (\ref{ec6}) and (\ref{ec7}).
Then applying inequality \eqref{PhiH1}
yields the desired inequality (\ref{ec0}).
\end{proof}

Next we estimate $\pd_t^k\nabla\psi$ for $k=0,1$.

\begin{lemma}
\label{lm3}
Under the same conditions as in Proposition \ref{apriori1} with $m=3$, 
it holds that
\begin{align}
{}&
e^{\zeta t}\lt{\pd_t^k\nabla\psi(t)}^2
+ \int_0^t e^{\zeta \tau} \lt{\pd_t^{k+1}\psi(\tau)}^2 \, d \tau
\notag\\
&\lesssim \hs{3}{\Phi_0}^2
+ \lambda \calh_{k}^\zeta(t)
+ \lambda^{-1} \calp_{k}^\zeta(t)
+ (N_\beta(t)+\dels+\zeta) \int_0^t e^{\zeta \tau} D_{3,\beta}(\tau) \, d \tau
+  \dels \int_0^t e^{\zeta \tau} \tau
\label{ed0}
\end{align}
for $t \in [0,T]$, $\zeta>0$, $\lambda \in (0,1)$, and $k=0,1$, where $C$ is a positive constant 
independent of $\delta$, $t$, and $\zeta$.
Furthermore, $\calh_{k}^\zeta(t)$ and $\calp_{k}^\zeta(t)$ are defined by
\begin{align*}
\calh_{k}^\zeta(t)
& :=
e^{\zeta t} \lt{\pd_t^k \nabla \vp(t)}^2
+ \int_0^t e^{\zeta \tau} \lt{\pd_t^k \nabla \vp(\tau)}^2 \, d \tau,
\\
\calp_{k}^\zeta(t)
& :=
e^{\zeta t} \lt{\pd_t^k \psi(t)}^2
+ \int_0^t e^{\zeta \tau} \lt{\pd_t^k\nabla \psi(\tau)}^2 \, d \tau.
\end{align*}
\end{lemma}

\begin{proof}
Multiplying (\ref{ec2}) by $\partial_t^k \psi_t$, we get
\begin{equation}
\rho |\partial_t^k \psi_t|^2
+ \rho (u \cdot \nabla) \dtn{j}{k} \psi \cdot \partial_t^k \psi_t
- L (\partial_t^k \psi) \cdot \partial_t^k \psi_t
+ p'(\rho) \nabla \partial_t^k \vp \cdot \partial_t^k \psi_t
= g_{0,k} \cdot \partial_t^k \psi_t.
\label{ed3}
\end{equation}
The third and the fourth terms on the left hand side of (\ref{ed3}) 
are rewritten to
\begin{align}
&
- L (\partial_t^k \psi) \cdot \partial_t^k \psi_t
=
\Bigl(
\frac{\mu_1}{2} |\nabla \partial_t^k \psi|^2
+ \frac{\mu_1+\mu_2}{2} |\div \partial_t^k \psi|^2
\Bigr)_t
\nonumber
\\
&
\mspace{150mu}
- \div
\Bigl\{
\mu_1 \nabla \partial_t^k \psi \cdot \partial_t^k \psi_t
+ (\mu_1+\mu_2) (\div \partial_t^k \psi )\partial_t^k \psi_t
\Bigr\},
\label{ed4}
\\
&
p'(\rho) \nabla \partial_t^k \vp \cdot \partial_t^k \psi_t
=
\{ p'(\rho) \nabla \partial_t^k \vp \cdot \partial_t^k \psi \}_t
- \div (p'(\rho) \partial_t^k \vp_t \, \partial_t^k \psi)
\nonumber
\\
&
\mspace{160mu}
- p'(\rho) \div \partial_t^k \psi
   (u \cdot \nabla \partial_t^k \vp + \rho \div \partial_t^k \psi)
- p''(\rho) \vp_t \nabla \partial_t^k \vp \cdot \partial_t^k \psi
\nonumber
\\
&
\mspace{160mu}
- p''(\rho) \nabla \rho \cdot \partial_t^k \psi
   (u \cdot \nabla \partial_t^k \vp + \rho \div \partial_t^k \psi)
+ f_{0,k} \div (p'(\rho) \partial_t^k \psi).
\label{ed5}
\end{align}
Substituting (\ref{ed4}) and (\ref{ed5}) into (\ref{ed3}) yields
\begin{equation}
\dt E_{3}
- \div B_3
+ \rho |\partial_t^k \psi_t|^2
= G_3 + R_3,
\label{ed1}
\end{equation}
where $E_3$, $B_3$, $G_3$ and $R_3$ are defined by
\begin{align*}
E_3 := {}
&
\frac{\mu_1}{2} |\nabla \partial_t^k \psi|^2
+ \frac{\mu_1 + \mu_2}{2} |\div \partial_t^k \psi|^2
+ p'(\rho) \nabla \partial_t^k \vp \cdot \partial_t^k \psi,
\\
B_3 := {}
&
\mu_1 \nabla \partial_t^k \psi \cdot \partial_t^k \psi_t
+ (\mu_1 + \mu_2) \partial_t^k \psi_t \div \partial_t^k \psi
+ p'(\rho) \partial_t^k \vp_t \, \partial_t^k \psi,
\\
G_3 := {}
&
p'(\rho) \div \partial_t^k \psi \, (u \cdot \nabla \partial_t^k \vp 
  + \rho \div \partial_t^k \psi)
-\rho (u \cdot \nabla) \partial_t^k \psi \cdot \partial_t^k \psi_t,
\\
R_3 := {}
&
p''(\rho) \vp_t \nabla \partial_t^k \vp \cdot \partial_t^k \psi
+ p''(\rho) \nabla \rho \cdot \partial_t^k \psi
     (u \cdot \nabla \partial_t^k \vp  + \rho \div \partial_t^k \psi)
\\
&
- f_{0,k} \div(p'(\rho) \partial_t^k \psi)
+  g_{0,k}  \cdot \partial_t^k \psi_t.
\end{align*}
Owing to $\partial_t \psi =0$ on $\partial \Omega$, we have  
\begin{gather}
\int_{\Omega}
\div B_3
 \, dx= 0.
\label{ed10}
\end{gather}
For arbitrary positive constants $\ep$ and $\lambda$,
the integrations of $E_3$ and $G_3$ over $\Omega$ are estimated as
\begin{gather}
c \lt{\nabla \partial_t^k \psi}^2
- \lambda \lt{\nabla \partial_t^k \vp}^2
- C \lambda^{-1} \lt{\partial_t^k \psi}^2
\le
\int_{\Omega} E_3 \, d x
\lesssim
\ho{\partial_t^k \Phi}^2,
\label{ed6}
\\
\int_{\Omega} |G_3| \, d x
\lesssim
\lambda \lt{\nabla \partial_t^k \vp}^2
+ \ep \lt{\partial_t^k \psi_t}^2
+ (\lambda^{-1}+\ep^{-1}) \lt{\nabla \partial_t^k \psi}^2
\label{ed2}
\end{gather}
by using \eqref{bdd}.
It is straightforward by \eqref{PhiSup} to check that
\begin{equation}
|R_3| \lesssim \left\{
\begin{array}{ll}
(N_\beta(T)+\delta)|(\nabla\Phi,\pd_t \Phi)|^2+|(\rt',\ut', \nabla U)||\Phi|^2 +\delta^{-1}|(F, G)|^2 & \text{if $k=0$,}
\\
(N_\beta(T)+\delta)|(\nabla\Phi,\pd_t \Phi,\nabla^2\Phi,\pd_t \nabla \Phi,\pd_{tt} \psi)|^2 & \text{if $k=1$.}
\end{array} 
\right.
\label{ed9}
\end{equation}

Finally, we multiply (\ref{ed1}) by $e^{\zeta t}$,
integrate the resultant equality over $(0,t) \times \Omega$ and
substitute in the estimates (\ref{ed10})--(\ref{ed9}). Making use of \eqref{PhiH1}, \eqref{hardy} and letting $\ep$ be suitably small lead to the desired estimate (\ref{ed0}).
\end{proof}

\subsection{Spatial-derivative estimates}\label{Sptial-deriv}

In order to flatten the boundary and obtain tangential derivatives, we introduce the following change of variables:
\begin{equation}\label{CV1}
\Gamma : 
\begin{cases}
& x_1 = y_1 + M(y_2, y_3), \\
& x_2 = y_2, \\
& x_3 = y_3,  \\
\end{cases}
\end{equation} 
and its inverse 
\begin{equation} \label{CV11}
\hat{\Gamma} : 
\begin{cases}
& y_1 = x_1 - M(x_2, x_3), \\
& y_2 = x_2, \\
& y_3 = x_3.  \\
\end{cases}
\end{equation}
We notice that $\Gamma(\mathbb{R}_3^+)=\Omega$, where 
$\mathbb{R}_3^+ : = \{ (y_1, y_2, y_3) \in \mathbb{R}^3 : y_1 >0 \}$.
We set $y' =(y_2, y_3)$ and denote the matrix
\begin{equation}\label{CV2}
A(y') : = 
\left( \begin{array}{ccc} 
1 & 0 & 0   \\
-\partial_{y_2} M(y') & 1 & 0 \\
-\partial_{y_3} M(y') & 0 & 1 \\ 
 \end{array} \right).
\end{equation}
Let us define
\begin{gather*}
\hat{\vp} (t,y) := \vp(t,\Gamma(y)), \quad \hat{\psi} (t,y) : = \psi (t,\Gamma (y)), 
\\ \hat{\rho} (t,y) := \rho(t,\Gamma(y)), \quad \hat{u} (t,y) := u(t,\Gamma(y)), \quad \hat{U} (t,y) := U(t,\Gamma(y)).
\end{gather*}
Note that $(\hat{\vp}, \hat{\psi})(t,y)$ is a vector-valued function defined on $\{t \geq 0\} \times \overline{\mathbb{R}_3^+}$, 
and $(\rt,\ut)(\tilde{M}(\Gamma(y)))=(\rt,\ut)(y_1)$ holds. We now have
\begin{equation} \label{CV3}
\nabla_x \vp (t,\Gamma(y)) = \hat{\nabla} \hat{\vp}(t,y) := A \nabla_y \hat{\vp} (t,y) = \big(\sum_{j=1}^3 A_{kj} \partial_{y_j} \hat{\vp}  \big)_{k=1, 2, 3} (t,y), 
\end{equation}
\begin{equation} \label{CV4}
\div \!{}_x  \psi (t,\Gamma(y)) = \hat{\div} \  \hat{\psi}(t,y) := (A \nabla_y) \cdot \hat{\psi} (t,y) =  \sum_{k=1}^3 \sum_{j=1}^3 A_{kj} \partial_{y_j} \hat{\psi}_k (t,y), 
\end{equation}
\begin{equation} \label{CV5}
\Delta_x  \psi (t,\Gamma(y)) = \hat{\Delta} \hat{\psi}(t,y) := (A \nabla_y) \cdot (A \nabla_y \hat{\psi}) (t,y) =  \sum_{k=1}^3 \big( \sum_{j=1}^3 A_{kj} \partial_{y_j} \big)^2 \hat{\psi} (t,y),
\end{equation}
\begin{equation} \label{CV6}
\frac{d}{dt} \vp (t,\Gamma(y)) = \hat{\frac{d}{dt} } \hat{\vp} (t,y) := \partial_t \hat{\vp}(t,y) + \hat{u} \cdot \hat{\nabla} \hat{\vp} (t,y).
\end{equation}

From \eqref{eq-pv}, we obtain the equation for $(\hat{\vp}, \hat{\psi})$
\begin{subequations}
\label{eq-pv-hat}
\begin{gather}
 \hat{\vp}_t + \hat{u} \cdot \hat{\nabla} \hat{\vp} + \hat{\rho} \hat{\div} \hat{\psi}
= \hat{f} + \hat{F},
\label{eq-pv1-hat}
\\
\hat{\rho} \{ \hat{\psi}_t  + (\hat{u} \cdot \hat{\nabla}) \hat{\psi} \}
- \hat{L} \hat{\psi} + p'(\hat{\rho}) \hat{\nabla} \hat{\vp}
= \hat{g} + \hat{G}
\label{eq-pv2-hat}
\end{gather}
and the initial and boundary conditions
\begin{gather}
(\hat{\vp},\hat{\psi})(0,y)
=
 (\hat{\vp}_0, \hat{\psi}_0)(y)
= (\rho_0, u_0)(\Gamma(y)) - (\rt, \ut)(y_1) - (0, U(\Gamma (y))),
\label{pic-hat}
\\
\hat{\psi}(t,0,y') = 0.
\label{pbc-hat}
\end{gather}
\end{subequations}
Here $\hat{L} \hat{\psi}$, $\hat{f}$, $\hat{F}$, $\hat{g}$ and $\hat{G}$ are defined by
\begin{gather*}
\hat{L} \hat{\psi}(t,y) := \mu_1 \hat{\Delta} \hat{\psi} (t,y) + (\mu_1 + \mu_2) \hat{\nabla} \hat{\div} \hat{\psi} (t,y),
\\
\hat{f} (t,y) :=f(t,\Gamma(y)),
\quad
\hat{F} (y) :=F(\Gamma(y)),
\quad
\hat{g} (t,y) := g(t,\Gamma(y)), 
\quad
\hat{G} (y) := G (\Gamma(y)).
\end{gather*}

We now derive the estimate on the spatial-derivatives for the tangential directions. 
\begin{lemma}
\label{lm2-hat}
Under the same conditions as in Proposition \ref{apriori1} with $m=3$, 
it holds that 
\begin{align}
{}&e^{\zeta t} \|\nabla^l_{y'} \hat{\Phi}(t)\|_{L^2(\mathbb R^3_+)}^2
+ \int_0^t e^{\zeta \tau} 
\left(\|\nabla\nabla^l_{y'}  \hat{\psi}(\tau)\|_{L^2(\mathbb R^3_+)}^2
+ \left\| \nabla^l_{y'} \hat{\frac{d}{dt}} \hat{\vp} (\tau) \right\|_{L^2(\mathbb R^3_+)}^2 
\right) \, d \tau 
\notag \\
&\lesssim\|\Phi_0\|_{H^3}^2
+ \epsilon \int^t_0 e^{\zeta \tau}\left( \| \nabla \vp(\tau) \|^2_{H^{l-1}} + \| \nabla \psi(\tau) \|^2_{H^{l}} \right) d \tau 
+ \epsilon^{-1} \int^t_0 e^{\zeta \tau}\| \nabla \psi(\tau) \|^2_{H^{l-1}} d \tau 
\notag \\
&\quad + (N_\beta(T)+\dels+\zeta) \int_0^t e^{\zeta \tau} D_{3,\beta}(\tau) \, d \tau
+  \dels   \int_0^t e^{\zeta \tau} d \tau
\label{ec0-hat}
\end{align}
for $t \in [0,T]$, $\zeta>0$, $\epsilon \in (0,1)$, and $l = 1, 2, 3$.
\end{lemma}
\begin{proof}
Applying the differential operator $\nabla^l_{y'}$ to 
(\ref{eq-pv1-hat}) and (\ref{eq-pv2-hat}), we have the following
two equations:
\begin{gather}
\nabla^l_{y'} \hat{\vp}_t
+ \hat{u} \cdot \hat{\nabla} \nabla^l_{y'} \hat{\vp}
+ \hat{\rho} \hat{\div} \nabla^l_{y'} \hat{\psi}
= \hat{f}_{l,0},
\label{ec1-hat}
\\
\hat{\rho} \{
\nabla^l_{y'} \hat{\psi}_t
+ (\hat{u} \cdot \hat{\nabla}) \nabla^l_{y'} \hat{\psi}
\}
- \hat{L}(\nabla^l_{y'} \hat{\psi})
+ p'(\hat{\rho}) \hat{\nabla} \nabla^l_{y'} \hat{\vp}
= \hat{g}_{l,0},
\label{ec2-hat}
\end{gather}
where
\begin{align*}
\hat{f}_{l,0}
&:=
\nabla^l_{y'} (\hat{f} + \hat{F})
- \{\nabla^l_{y'}(( \hat{u} \cdot \hat{\nabla} ) \hat{\vp} ) - \hat{u} \cdot \hat{\nabla} \nabla^l_{y'} \hat{\vp}\}
- \{\nabla^l_{y'} ( \hat{\rho} \hat{\div} \hat{\psi}) - \hat{\rho} \hat{\div} \nabla^l_{y'} \hat{\psi} \},
\\
\hat{g}_{l,0}
&:= \nabla^l_{y'} (\hat{g} + \hat{G})
- [ \nabla^l_{y'}, \hat{\rho}] \hat{\psi}_t  
- \{ \nabla^l_{y'} (\hat{\rho} ( \hat{u} \cdot \hat{\nabla}) \hat{\psi}) - \hat{\rho} (\hat{u} \cdot \hat{\nabla}) \nabla^l_{y'} \hat{\psi}   \}
- [\nabla^l_{y'},\hat{L}] \hat{\psi}\\
& \qquad - \{ \nabla^l_{y'} (p'(\hat{\rho}) \hat{\nabla} \hat{\psi} ) - p'(\hat{\rho}) \hat{\nabla} \nabla^l_{y'} \hat{\psi}  \}.
\end{align*}
Recall that $P(\rho) := p'(\rho) / \rho$. Multiplying \eqref{ec1-hat} by $P(\hat{\rho}) \, \nabla^l_{y'} \hat{\vp}$ and using the facts that $\hat{\rho}_t = - \hat{\div} (\hat{\rho} \hat{u})$ and $ - P' (\hat{\rho}) \hat{\rho} + P(\hat{\rho}) = (3 - \gamma) P(\hat{\rho}) $, 
we get
\begin{multline}
\Bigl(
\frac{1}{2} P(\hat{\rho}) | \nabla^l_{y'} \hat{\vp} |^2
\Bigr)_t
+ \hat{\div} \Bigl(
\frac{1}{2} P(\hat{\rho}) \hat{u}
|\nabla^l_{y'} \hat{\vp}|^2
\Bigr)
+ p'(\hat{\rho}) \hat{\div} (\nabla^l_{y'} \hat{\psi}) \nabla^l_{y'} \hat{\vp}
\\
=
P(\hat{\rho}) \hat{f}_{l,0} \, \nabla^l_{y'} \hat{\vp}
+
\frac{3-\gamma}{2} P(\hat{\rho}) \hat{\div} \hat{u} \, |\nabla^l_{y'} \hat{\vp}|^2.
\label{ec3-hat}
\end{multline}
Multiply (\ref{ec2}) by $\nabla^l_{y'} \hat{\psi}$ successively and make use of $\hat{\rho}_t = - \hat{\div} (\hat{\rho} \hat{u})$ to get
\begin{gather}
\Bigl(
\frac{1}{2} \hat{\rho} |\nabla^l_{y'} \hat{\psi}|^2
\Bigr)_t
+\hat{\div} \hat{B}_2
+ \mu_1 |\hat{\nabla} (\nabla^l_{y'} \hat{\psi})|^2
+ (\mu_1 + \mu_2) |\hat{\div} (\nabla^l_{y'} \hat{\psi})|^2
\mspace{200mu}
\nonumber
\\
\mspace{150mu}
= p'(\hat{\rho}) \hat{\div} (\nabla^l_{y'} \hat{\psi}) \nabla^l_{y'} \hat{\vp}
+ (\hat{g}_{l,0} + p''(\hat{\rho}) \, \nabla^l_{y'} \hat{\vp} \hat{\nabla} \hat{\rho})
   \cdot \nabla^l_{y'} \hat{\psi},
\label{ec4-hat}
\\
\hat{B}_2
:= {} 
\frac{1}{2} \hat{\rho} \hat{u} |\nabla^l_{y'} \hat{\psi}|^2
- \mu_1 \hat{\nabla} (\nabla^l_{y'} \hat{\psi})  \cdot \nabla^l_{y'} \hat{\psi}
- (\mu_1+\mu_2) \hat{\div} (\nabla^l_{y'} \hat{\psi}) \nabla^l_{y'} \hat{\psi} 
+ p'(\hat{\rho}) \nabla^l_{y'} \hat{\vp} \, \nabla^l_{y'} \hat{\psi}.
\nonumber
\end{gather}
Adding (\ref{ec3-hat}) to (\ref{ec4-hat}) yields
\begin{gather}
\Bigl(
\frac{1}{2} P(\hat{\rho}) | \nabla^l_{y'} \hat{\vp} |^2
+ \frac{1}{2} \hat{\rho} |\nabla^l_{y'} \hat{\psi}|^2
\Bigr)_t
+ \hat{\div} \Bigl(
\frac{1}{2} P(\hat{\rho}) \hat{u}
|\nabla^l_{y'} \hat{\vp}|^2
+ \hat{B}_2
\Bigr)
\mspace{150mu}
\nonumber
\\
\mspace{200mu}
{}+ \mu_1 |\hat{\nabla} (\nabla^l_{y'} \hat{\psi})|^2
+ (\mu_1 + \mu_2) |\hat{\div} (\nabla^l_{y'} \hat{\psi})|^2
=
\hat{R}_2,
\label{ec5-hat}
\\
\hat{R}_2
:=
P(\hat{\rho}) \hat{f}_{l,0} \, \nabla^l_{y'} \hat{\vp}
+
\frac{3-\gamma}{2} P(\hat{\rho}) \hat{\div} \hat{u} \, |\nabla^l_{y'} \hat{\vp}|^2
+ (\hat{g}_{l,0} + p''(\hat{\rho}) \, \nabla^l_{y'} \hat{\vp} \hat{\nabla} \hat{\rho})
   \cdot \nabla^l_{y'} \hat{\psi}.
\nonumber
\end{gather}

Let us look at the left hand side of \eqref{ec5-hat}.
Owing to the divergence theorem with \eqref{bce} and \eqref{pbc},
we have the nonnegativity of the second terms on the
left hand side of (\ref{ec5-hat}) as
\begin{align}
\int_{\mathbb{R}^3_+}
\hat{\div} \Bigl(
\frac{P(\hat{\rho}) \hat{u}}{2}
|\nabla^l_{y'} \hat{\vp}|^2 + \hat{B}_2
\Bigr) \, dy
=\int_{\mathbb{R}^2}
\frac{P(\hat{\rho}) }{2}
|\nabla^l_{y'} \hat{\vp}|^2
 (\ub \cdot n) \sqrt{1+|\nabla M|^2}  \, dy' \geq 0.
\label{ec6-hat}
\end{align}
Notice that here we used $\nabla^l_{y'} \psi =0$ on $\partial \Omega$, which holds because of \eqref{pbc}. Using the fact that 
$|\nabla \nabla^l_{y'} \hat{\psi}| \lesssim |\hat{\nabla} \nabla^l_{y'} \hat{\psi}|$,
we also have the good contribution from the third and fourth terms in \eqref{ec5-hat} as
\begin{equation}\label{ec10-hat}
\int_{\mathbb{R}^3_+}  \mu_1 |\hat{\nabla} \nabla^l_{y'} \hat{\psi}|^2
+ (\mu_1 + \mu_2) |\hat{\div} \nabla^l_{y'} \hat{\psi}|^2 \,dy
\gtrsim \int_{\mathbb{R}^3_+} \mu_1 |\nabla \nabla^l_{y'} \hat{\psi}|^2 \,dy.
\end{equation}

We are going to show that $\hat{R}_2$ satisfies
\begin{multline}
\int_{\mathbb R^3_+}|\hat{R}_2| \,dy
\lesssim 
(N_{\beta}(T) + \dels)D_{3,\beta} 
+ \epsilon \left(\|\nabla_y \hat{\vp}\|_{H^{l-1}(\mathbb R^3_+)}^2
+ \|\nabla_y \hat{\psi}\|_{H^l(\mathbb R^3_+)}^2\right)
\\
+ \epsilon^{-1} \|\nabla_y \hat{\psi}\|_{H^{l-1}(\mathbb R^3_+)}^2 
+ \delta^{-1}\|(\hat{F},\hat{G})\|_{H^3(\mathbb R^3_+)}^2, 
\label{ec7-hat}
\end{multline} 
where some $\epsilon \in (0,1)$. 
Let us first estimate the integrations of 
$P(\hat{\rho}) \hat{f}_{l,0} \, \nabla^l_{y'} \hat{\vp}$ and 
$\hat{g}_{l,0} \cdot \nabla^l_{y'} \hat{\psi}$ in $\hat{R}_2$, 
where $\hat{f}_{l,0}$ and $\hat{g}_{l,0}$ are defined in \eqref{ec1-hat} and \eqref{ec2-hat}.
Noting that $(\rt',\ut')(\tilde{M}(\Gamma(y)))=(\rt',\ut')(y_1)$ and applying
Sobolev's inequalities  \eqref{sobolev0}--\eqref{sobolev2} with \eqref{stdc1} and \eqref{ExBdry3},
we have
\begin{align}
\|(\nabla_{y'}^l \hat{f},\nabla_{y'}^l \hat{g})\|_{L^2(\mathbb R^3_+)}
& \lesssim \delta  \left( \|\nabla_y \hat{\vp}\|_{H^{2}(\mathbb R^3_+)} + \|\nabla_y \hat{\psi}\|_{H^3(\mathbb R^3_+)} \right) 
\notag \\
& \quad + \left\||(\hat{\vp}, \hat{\psi})| (|(\rt',\ut')||\nabla^l_{y'} \nabla M| + |\nabla^l_{y'} \nabla U|)\right\|_{L^2(\mathbb R^3_+)}
\notag\\
& \lesssim \delta  \left( \|\nabla_y \hat{\vp}\|_{H^{2}(\mathbb R^3_+)} + \|\nabla_y \hat{\psi}\|_{H^3(\mathbb R^3_+)} \right) 
\notag \\
& \quad + \| (\hat{\vp}, \hat{\psi})\|_{L^6(\mathbb R^3_+)} \left( \left\||(\rt',\ut')| |\nabla^l_{y'} \nabla M| \right\|_{L^3(\mathbb R^3_+)} + \|\nabla^l_{y'} \nabla U \|_{L^3(\mathbb R^3_+)} \right)
\notag\\
&  \lesssim \delta  D_{3,\beta}. 
\label{ec8-hat}
\end{align}
Using \eqref{ec8-hat} and Schwarz's inequality, one can see that the integrations of
$|P(\hat{\rho}) (\nabla_{y'}^l \hat{f} + \nabla_{y'}^l \hat{F}) \nabla^l_{y'} \hat{\vp}|$ and 
$|(\nabla_{y'}^l \hat{g} + \nabla_{y'}^l \hat{G})  \cdot \nabla^l_{y'} \hat{\psi}|$ are bounded 
from above by the right hand side of \eqref{ec7-hat}.
Notice that the other terms in $P(\hat{\rho}) \hat{f}_{l,0} \, \nabla^l_{y'} \hat{\vp}$ and 
$\hat{g}_{l,0} \cdot \nabla^l_{y'} \hat{\psi}$ are just commutator terms.
Using suitably Lemma \ref{CommEst} with the facts that 
\begin{gather*}
\hat{\rho}(y)= \tilde{\rho}(y_1) + \hat{\vp}(y), \quad 
\nabla^l_{y'} (\rt(y_1) \, \cdot \, ) =\rt(y_1)(\nabla^l_{y'} \, \cdot\, ),
\\
\hat{u}(y)= \tilde{u}(y_1) + \hat{\psi}(y) + \hat{U}(y), \quad
\nabla^l_{y'} (\ut(y_1) \, \cdot \, ) =\ut(y_1)(\nabla^l_{y'} \, \cdot \, ),
\end{gather*}
we can see that the commutator terms are bounded by the right hand side of \eqref{ec7-hat}.
Now we have completed the estimation of all terms 
in $P(\hat{\rho}) \hat{f}_{l,0} \, \nabla^l_{y'} \hat{\vp}$ and 
$\hat{g}_{l,0} \cdot \nabla^l_{y'} \hat{\psi}$. 
It is quite straightforward to handle the other terms in $\hat{R}_2$ 
with aid of \eqref{stdc1} and \eqref{ExBdry3}.
Therefore we conclude \eqref{ec7-hat}.

Applying $\nabla^l_{y'}$ to \eqref{eq-pv1-hat}, we arrive at
\begin{equation*}
\nabla^l_{y'} \hat{\frac{d}{dt}} \hat{\vp}=
-\nabla^l_{y'} \big( \hat{\rho} \hat{\div} \hat{\psi} \big)
+ \nabla^l_{y'} (\hat{f} + \hat{F}). 
\end{equation*}
We take the $L^2$-norm and estimate the terms on the right hand side 
with aid of \eqref{ec8-hat} as

\begin{equation}
\left\| \nabla^l_{y'} \hat{\frac{d}{dt}} \hat{\vp}  \right\|_{L^2(\mathbb R^3_+)}^2 
\lesssim \| \hat{\nabla} \nabla_{y'}^l \hat{\psi}\|_{L^2(\mathbb R^3_+)}^2  
+\delta D_{3,\beta}+\|\nabla_y \hat{\psi}\|_{H^{l-1}(\mathbb R^3_+)}^2
+\|\hat{F}\|_{L^2(\mathbb R^3_+)}^2.
\label{ec9-hat}
\end{equation}

We multiply (\ref{ec5-hat}) by the time weight function $e^{\zeta t}$, 
integrate the resultant equality over $(0,t) \times \mathbb{R}^3_+$,
and substitute (\ref{ec6-hat})--(\ref{ec7-hat}) into the result. 
Using \eqref{h1} and \eqref{ec9-hat} and
performing change of variables $y \rightarrow x$ on the right hand side, 
we arrive at the desired inequality (\ref{ec0-hat}). 
This completes the proof of the lemma. 
\end{proof}

Next we estimate the spatial-derivatives for the normal direction.
To simplify the notations, we denote $\partial_j := \partial_{y_j}$ for $j =1,2,3$. 
Applying $\partial_1$ to \eqref{eq-pv1-hat} and multiplying the result 
by $\mu := 2 \mu_1 + \mu_2$ yields
\begin{equation}  
\mu \partial_1  \hat{\frac{d}{dt}} \hat{\vp}  + \mu \partial_1 \hat{\rho} \hat{\div} \hat{\psi} + \mu \hat{\rho} \partial_1 \hat{\div} \hat{\psi} = \mu \partial_1 (\hat{f} + \hat{F}) . 
\label{ee1-hat}
\end{equation}
We need to make some cancellation on the term $\mu \hat{\rho} \partial_1 \hat{\div} \hat{\psi}$ so as to avoid the highest order derivative in the normal direction $y_1$. Denote
\begin{align} 
& \mathcal{A}_1 := \frac{\mu}{\mu_1 (1 + |\nabla M|^2) + \mu_1 + \mu_2}, \quad
 \mathcal{A}_j := -\frac{\partial_j M}{\mu_1 (1 + |\nabla M|^2)}\{\mu-(\mu_1 + \mu_2)\mathcal{A}_1\},\quad  j=2,3,
\notag \\
& \tilde{\mathcal{A}}_1:=\mathcal{A}_1-\mathcal{A}_2\partial_2 M-\mathcal{A}_3\partial_3 M>0, \quad
\mathcal{D} := \tilde{\mathcal{A}}_1\partial_1 + {\mathcal{A}}_2\partial_2 + {\mathcal{A}}_3\partial_3.  
\label{Aj-def}
\end{align}
Take an inner product of \eqref{eq-pv2-hat} with 
$(\hat{\rho} \mathcal{A}_1, \hat{\rho}\mathcal{A}_2,\hat{\rho} \mathcal{A}_3 )^{\top}$, we obtain
\begin{multline}
\hat{\rho}^2 \left( \sum_{j=1}^3 \mathcal{A}_j \hat{\psi}_{jt} + \sum_{j=1}^3 \mathcal{A}_j (\hat{u} \cdot \hat{\nabla}) \hat{\psi}_j  \right) - \mu_1 \hat{\rho} \sum_{j=1}^3 \mathcal{A}_j \hat{\Delta} \hat{\psi}_j 
\\
- (\mu_1 + \mu_2) \hat{\rho} \sum_{j=1}^3 \mathcal{A}_j \partial_j \hat{\div} \hat{\psi} + p'(\hat{\rho}) \hat{\rho} \mathcal{D}\hat{\vp}  = \hat{\rho} \sum_{j=1}^3 \mathcal{A}_j (\hat{g}_j + \hat{G}_j), 
\label{ee2-hat}
\end{multline}
where $\hat{g}_j$ and $ \hat{G}_j$ are the $j$-th components of $\hat{g}$ and $\hat{G}$, respectively.
Adding \eqref{ee1-hat} and \eqref{ee2-hat} together gives
\begin{multline}
\mu \partial_1  \hat{\frac{d}{dt}} \hat{\vp}  + \mu \partial_1 \hat{\rho} \hat{\div} \hat{\psi}  + \hat{\rho}^2 \left( \sum_{j=1}^3 \mathcal{A}_j \hat{\psi}_{jt} + \sum_{j=1}^3 \mathcal{A}_j (\hat{u} \cdot \hat{\nabla}) \hat{\psi}_y  \right) \\
 + p'(\hat{\rho}) \hat{\rho} \mathcal{D}\hat{\vp} + I + II + III 
 = \mu \partial_1 (\hat{f} + \hat{F} ) + \hat{\rho} \sum_{j=1}^3 \mathcal{A}_j (\hat{g}_j + \hat{G}_j), 
\label{ee3-hat}
\end{multline}
where
\begin{align*}
I &: = \mu \hat{\rho} \partial_1 \hat{\div} \hat{\psi} = \mu \hat{\rho} (\partial_1^2 \hat{\psi}_1
-\partial_2 M \partial_1^2 \hat{\psi}_2 - \partial_3 M \partial_1^2 \hat{\psi}_3)+I', 
\\
II &:= - \mu_1 \hat{\rho} \sum_{j=1}^3 \mathcal{A}_j \hat{\Delta} \hat{\psi}_j = - \mu_1 \hat{\rho} \sum_{j=1}^3 \mathcal{A}_j (1 + |\nabla M|^2)\partial_1^2 \hat{\psi}_j+II',
\\
III &:= - (\mu_1 + \mu_2) \hat{\rho} \sum_{j=1}^3 \mathcal{A}_j \partial_j \hat{\div} \hat{\psi} = - (\mu_1 + \mu_2) \hat{\rho} \mathcal{A}_1 (\partial_1^2 \hat{\psi}_1
-\partial_2 M \partial_1^2 \hat{\psi}_2 - \partial_3 M \partial_1^2 \hat{\psi}_3) + III',  
\end{align*}
where $I'$, $II'$, and $III'$ do not have terms 
with second-order normal derivative $\partial_{1}^2$.
Due to the choice of $\mathcal{A}_j$, it is straightforward to check that
\begin{equation}
I+II+III= I' + II' + III' = \hat{\rho}\sum_{1\leq |\bm{b}|\leq 2, \ b_1\neq2, \ j=1,2,3} a_{\bm{b}}\partial^{\bm{b}}_y\hat{\psi}_j,
\label{ee4-hat}
\end{equation}
where $a_{\bm{b}}$ denotes scalar-valued functions 
$a_{\bm{b}} = a_{\bm{b}} (\mu_1, \mu_2, \nabla M, \nabla^2 M)$.
Substituting \eqref{ee4-hat} into \eqref{ee3-hat}, multiplying the result by $\tilde{\mathcal{A}}_1$,
and using $\tilde{\mathcal{A}}_1\partial_1=\mathcal{D}-\mathcal{A}_2\partial_2-\mathcal{A}_3\partial_3$,
we arrive at
\begin{align}
& \mu \mathcal{D}  \hat{\frac{d}{dt}} \hat{\vp}  - \mu(\mathcal{A}_2\partial_2+\mathcal{A}_3\partial_3)  \hat{\frac{d}{dt}} \hat{\vp}  + \mu \tilde{\mathcal{A}}_1 \partial_1 \hat{\rho} \hat{\div} \hat{\psi}\notag \\
& \quad   + \tilde{\mathcal{A}}_1 \hat{\rho}^2 \left( \sum_{j=1}^3 \mathcal{A}_j \hat{\psi}_{jt} + \sum_{j=1}^3 \mathcal{A}_j (\hat{u} \cdot \hat{\nabla}) \hat{\psi}_j  \right)
+ \tilde{\mathcal{A}}_1 p'(\hat{\rho}) \hat{\rho} \mathcal{D}\hat{\vp}  
 + \tilde{\mathcal{A}}_1 \hat{\rho}\sum_{1\leq |\bm{b}|\leq 2, \ b_1\neq2, \ j=1,2,3} a_{\bm{b}} \partial^{\bm{b}}_y\hat{\psi}_j
\notag \\
& = \mu \tilde{\mathcal{A}}_1 \partial_1 (\hat{f} + \hat{F} ) + \tilde{\mathcal{A}}_1 \hat{\rho} \sum_{j=1}^3 \mathcal{A}_j (\hat{g}_j + \hat{G}_j). 
\label{ee5-hat}
\end{align}

\begin{lemma}
\label{lm4-hat}
Suppose that the same conditions as in Proposition \ref{apriori1} with $m=3$ hold. Define the index $\bm{a} = (a_1, a_2, a_3)$ with $a_1, a_2, a_3 \geq 0$ and $|\bm{a}| := a_1 + a_2 + a_3$. Let $\partial^{\bm{a}} := \partial^{a_1}_{y_1} \partial^{a_2}_{y_2} \partial^{a_3}_{y_3}$. Then it holds that 
\begin{align}
& e^{\zeta t} \| \partial^{\bm{a}} \partial_1 \hat{\vp} (t) \|_{L^2(\mathbb R^3_+)}^2 
+ \int^t_0 e^{\zeta \tau}\left(  \|\partial^{\bm{a}} {\cal D} \hat{\vp} (\tau) \|_{L^2(\mathbb R^3_+)}^2 
+ \left\| \partial^{\bm{a}} \partial_1 \hat{\frac{d}{dt}} \hat{\vp}(\tau) \right\|_{L^2(\mathbb R^3_+)}^2  \right) \, d\tau 
\notag\\
&  \lesssim  \|\vp_0\|_{H^3}^2 
+ \int_0^t e^{\zeta \tau} \left(
\left\|\partial^{\bm{a}} \nabla_{y'} \hat{\frac{d}{dt}} \hat{\vp}(\tau) \right\|_{L^2(\mathbb R^3_+)}^2 
+ \|\partial^{\bm{a}} \nabla_y \nabla_{y'} \hat{\psi}(\tau)\|_{L^2(\mathbb R^3_+)}^2 \right)\, d \tau
\notag\\
&\quad + \int_0^t e^{\zeta \tau} \left(|\bm{a}|\|\nabla_y {\vp}(\tau) \|^2_{H^{|\bm{a}| -1}}
+|\bm{a}|\left\|\nabla {\frac{d}{dt}} {\vp}(\tau) \right\|_{H^{|\bm{a}|-1}}^2
+\|{\psi}_t(\tau) \|^2_{H^{|\bm{a}|}} 
+\|\nabla_y {\psi}(\tau)\|^2_{H^{|\bm{a}|}} \right)\, d \tau
\notag\\
& \quad 
+ (N_\beta (T)+\dels+\zeta) \int_0^t e^{\zeta \tau} D_{3, \beta}(\tau) \, d \tau 
+ \dels\int_0^t e^{\zeta \tau} d \tau
\label{ee0-hat}
\end{align}
for $t \in [0, T]$, $\zeta > 0$, and $0\leq |\bm{a}| \leq 2$.
\end{lemma}

\begin{proof}
Applying $\partial^{\bm{a}}$ to \eqref{ee5-hat} yields
\begin{equation}
\mu \partial^{\bm{a}} \mathcal{D} \hat{\frac{d}{dt}} \hat{\vp} + \tilde{\mathcal{A}}_1 p'(\hat{\rho}) \hat{\rho} \partial^{\bm{a}} \mathcal{D} \hat{\vp}  
= \hat{I}_1 ,
\label{ee6-hat}
\end{equation}
where 
\begin{equation*}
\begin{split}
 \hat{I}_1 := 
& \Big\{ -\partial^{\bm{a}} [\tilde{\mathcal{A}}_1 p'(\hat{\rho}) \hat{\rho}  \mathcal{D} \hat{\vp}]
+ \tilde{\mathcal{A}}_1 p'(\hat{\rho}) \hat{\rho} \partial^{\bm{a}} \mathcal{D} \hat{\vp} 
\Big\} 
\\
& + \partial^{\bm{a}} \Big\{\mu(\mathcal{A}_2\partial_2+\mathcal{A}_3\partial_3) \hat{\frac{d}{dt}} \hat{\vp} - \mu \tilde{\mathcal{A}}_1 \partial_1 \hat{\rho} \hat{\div} \hat{\psi} - \tilde{\mathcal{A}}_1 \hat{\rho}^2 ( \sum_{j=1}^3 \mathcal{A}_j \hat{\psi}_{jt} + \sum_{j=1}^3 \mathcal{A}_j \hat{u} \cdot \hat{\nabla} \hat{\psi}_j ) 
\\
& \qquad \qquad  + \tilde{\mathcal{A}}_1 \hat{\rho}\sum_{1 \leq |\bm{b}|\leq 2, \ b_1\neq2, \ j=1,2,3} a_{\bm{b}} \partial^{\bm{b}}_y\hat{\psi}_j + \mu \tilde{\mathcal{A}}_1 \partial_1 (\hat{f} + \hat{F} ) + \tilde{\mathcal{A}}_1 \hat{\rho} \sum_{j=1}^3 \mathcal{A}_j (\hat{g}_j + \hat{G}_j) \Big\} . \\
\end{split}
\end{equation*}
Multiplying \eqref{ee6-hat} by $\partial^{\bm{a}} \mathcal{D} \hat{\vp}$ and $\partial^{\bm{a}} \mathcal{D} \hat{\frac{d}{dt}} \hat{\vp}$, respectively, and adding the two resultant equalities together, we obtain
\begin{multline}
 \left(\frac{1}{2} \mu |\partial^{\bm{a}} \mathcal{D} \hat{\vp}|^2 + \frac{1}{2} \tilde{\mathcal{A}}_1 p'(\hat{\rho}) \hat{\rho} |\partial^{\bm{a}} \mathcal{D} \hat{\vp}|^2  \right)_t + \mu \left|\partial^{\bm{a}} \mathcal{D} \hat{\frac{d}{dt}} \hat{\vp}\right|^2 + \tilde{\mathcal{A}}_1 p'(\hat{\rho}) \hat{\rho} |\partial^{\bm{a}} \mathcal{D} \hat{\vp}|^2 \\
+ \hat{\div}\left(\frac{1}{2} \mu |\partial^{\bm{a}} \mathcal{D} \hat{\vp}|^2 \hat{u} + \frac{1}{2} \tilde{\mathcal{A}}_1 p'(\hat{\rho}) \hat{\rho} |\partial^{\bm{a}} \mathcal{D} \hat{\vp}|^2\hat{u}  \right)
= \hat{R}_3, 
\label{ee7-hat}
\end{multline}
where
\begin{equation*}
\begin{split}
\hat{R}_3&:=\big\{[\hat{u} \cdot \hat{\nabla} \partial^{\bm{a}} \mathcal{D} \hat{\vp} ] - \partial^{\bm{a}} \mathcal{D}   [\hat{u} \cdot \hat{\nabla} \hat{\vp}] \big\} (\mu \partial^{\bm{a}} \mathcal{D} \hat{\vp} + \tilde{\mathcal{A}}_1 p'(\hat{\rho}) \hat{\rho}\partial^{\bm{a}} \mathcal{D} \hat{\vp})
\\
& \quad + \frac{1}{2}\mu |\partial^{\bm{a}} \mathcal{D} \hat{\vp}|^2 \hat{\div} \hat{u}  
+ \frac{1}{2} \tilde{\mathcal{A}}_1 (p'(\hat{\rho}) \hat{\rho} )_t|\partial^{\bm{a}} \mathcal{D} \hat{\vp}|^2 
+ \frac{1}{2} |\partial^{\bm{a}} \mathcal{D} \hat{\vp}|^2 \hat{\div} \left(\tilde{\mathcal{A}}_1 p'(\hat{\rho}) \hat{\rho} \hat{u}\right) \\
& \quad + \hat{I}_1 \left(\partial^{\bm{a}} \mathcal{D} \hat{\vp} + \partial^{\bm{a}} \mathcal{D}  \hat{\frac{d}{dt}} \hat{\vp} \right) . 
\end{split}
\end{equation*}

Let us estimate the integrations of some terms in \eqref{ee7-hat}. 
We first find the good contribution for $\partial^{\bm{a}} \partial_1 \hat{\frac{d}{dt}} \hat{\vp}$
from the second and third terms on the left hand side. 
Indeed, using the fact $\partial_1=\tilde{\mathcal{A}}_1^{-1}\mathcal{D}-\tilde{\mathcal{A}}_1^{-1}\mathcal{A}_2\partial_2-\tilde{\mathcal{A}}_1^{-1}\mathcal{A}_3\partial_3$, 
we see that
\begin{gather} 
\int_{\mathbb R^3_+} \left|\partial^{\bm{a}} \partial_1 \hat{\frac{d}{dt}} \hat{\vp}\right|^2 \!dy 
\lesssim \int_{\mathbb R^3_+}
 \mu\left|\partial^{\bm{a}} \mathcal{D} \hat{\frac{d}{dt}} \hat{\vp}\right|^2 \!dy  
+\left\|\partial^{\bm{a}} \nabla_{y'} \hat{\frac{d}{dt}} \hat{\vp}\right\|_{L^2(\mathbb R^3_+)}^2 
\!+|\bm{a}| \left\|\hat{\frac{d}{dt}} \hat{\vp} \right\|_{H^{|\bm{a}|}(\mathbb R^3_+)}^2.
\label{ee10-hat}
\end{gather}
Owing to the {Divergence} Theorem with \eqref{bce} and \eqref{pbc},
we have the nonnegativity of the fourth term on the left hand side as
\begin{align}
&{}
\int_{\mathbb{R}^3_+}
\hat{\div}\left(\frac{1}{2} \mu |\partial^{\bm{a}} \mathcal{D} \hat{\vp}|^2 \hat{u} + \frac{1}{2} \tilde{\mathcal{A}}_1 p'(\hat{\rho}) \hat{\rho} |\partial^{\bm{a}} \mathcal{D} \hat{\vp}|^2\hat{u}  \right) dy
\notag \\
&=\int_{\mathbb{R}^2}
\frac{\ub \cdot n}{2}
\left(\mu |\partial^{\bm{a}} \mathcal{D} \hat{\vp}|^2+\tilde{\mathcal{A}}_1 p'(\hat{\rho}) \hat{\rho} |\partial^{\bm{a}} \mathcal{D} \hat{\vp}|^2\right)
 \sqrt{1+|\nabla M|^2} \, dy' \geq 0.
\label{ee8-hat}
\end{align}
Furthermore, we claim that the integration of $\hat{R}_3$ in \eqref{ee7-hat} is estimated as
\begin{align}
\int_{\mathbb R^3_+} |\hat{R}_3 | \,dy 
& \lesssim (\epsilon + N_\beta(T) + \delta )  \| \partial^{\bm{a}} \partial_1 \hat{\vp} \|_{L^2(\mathbb R^3_+)}^2 
+( \epsilon + N_\beta(T) + \delta ) \left\|\partial^{\bm{a}} \partial_1 \hat{\frac{d}{dt}} \hat{\vp}  \right\|_{L^2(\mathbb R^3_+)}^2
\notag\\
&\quad +( N_\beta(T) + \delta ) D_{3,\beta} 
+\epsilon^{-1}\left\|\partial^{\bm{a}} \nabla_{y'} \hat{\frac{d}{dt}} \hat{\vp} \right\|_{L^2(\mathbb R^3_+)}^2
+\epsilon^{-1}\|\partial^{\bm{a}} \nabla_y \nabla_{y'} \hat{\psi}\|_{L^2(\mathbb R^3_+)}^2
\notag\\
&\quad +\epsilon^{-1}|\bm{a}|\|\nabla_y \hat{\vp} \|_{H^{|\bm{a}| -1}(\mathbb R^3_+)}
+\epsilon^{-1}|\bm{a}|\left\|\hat{\frac{d}{dt}} \hat{\vp}\right\|_{H^{|\bm{a}|}(\mathbb R^3_+)}
+\epsilon^{-1}\| \hat{\psi}_t \|_{H^{|\bm{a}|}(\mathbb R^3_+)} 
\notag\\
&\quad +\epsilon^{-1}\|\nabla_y \hat{\psi}\|_{H^{|\bm{a}|}(\mathbb R^3_+)}
+\epsilon^{-1}\delta,
\label{ee11-hat}
\end{align}
where $\epsilon$ is a positive constant to be determined later.
To show this, we start from the estimation of $\hat{I}_1$:
\begin{align}
\| \hat{I}_1\|_{L^2(\mathbb R^3_+)} & \lesssim
( N_\beta(T) + \delta )  D_{3,\beta}
+\left\|\partial^{\bm{a}} \nabla_{y'} \hat{\frac{d}{dt}} \hat{\vp} \right\|_{L^2(\mathbb R^3_+)}\!
+\|\partial^{\bm{a}} \nabla_y \nabla_{y'} \hat{\psi}\|_{L^2(\mathbb R^3_+)}
+|\bm{a}|\|\nabla_y \hat{\vp} \|_{H^{|\bm{a}| -1}(\mathbb R^3_+)}
\notag \\
&\quad +|\bm{a}|\left\|\hat{\frac{d}{dt}} \hat{\vp}\right\|_{H^{|\bm{a}|}(\mathbb R^3_+)} \!
+\| \hat{\psi}_t \|_{H^{|\bm{a}|}(\mathbb R^3_+)} 
+\|\nabla_y \hat{\psi}\|_{H^{|\bm{a}|}(\mathbb R^3_+)}
+\|(\hat{F},\hat{G})\|_{L^2(\mathbb R^3_+)}.
\label{ee12-hat}
\end{align}
It is straightforward to check that all terms except those having $\hat{f}$ and $\hat{g}$
can be estimated by the right hand side of \eqref{ee12-hat}.
Let us handle the terms having $\hat{f}$ and $\hat{g}$.
By the applications of Hardy's inequality \eqref{hardy} and
Sobolev's inequalities \eqref{sobolev1} and \eqref{sobolev2}
together with \eqref{stdc1} and \eqref{ExBdry3}, it holds that  
\begin{align}
{}& \| \mu \partial^{\bm{a}}  \tilde{\mathcal{A}}_1 \partial_1 \hat{f}\|_{L^2(\mathbb R^3_+)}
+ \left\| \partial^{\bm{a}} \tilde{\mathcal{A}}_1  \hat{\rho} \sum_{j=1}^3 \mathcal{A}_j \hat{g}_j \right\|_{L^2(\mathbb R^3_+)} 
\notag \\
&\lesssim \delta \left( \|\nabla\hat{\vp} \|_{H^2(\mathbb R^3_+)} 
+ \|\nabla\hat{\psi} \|_{H^3(\mathbb R^3_+)}\right)
+ \sum_{i=1}^{|\bm{a}|+1}\left\||(\rt^{(i)},\ut^{(i)})||\hat{\Phi}|\right\|_{L^2(\mathbb R^3_+)}\!
+ \|\hat{\Phi}\|_{L^6(\mathbb R^3_+)} \sum_{i=0}^{|\alpha|+1} \|\nabla^i U\|_{L^3(\mathbb R^3_+)}
\notag \\
&\lesssim \delta D_{3,\beta}. 
\label{ee13-hat}
\end{align}
Therefore we conclude that \eqref{ee12-hat} holds.
From now on we estimate the integration of $\hat{R}_3$.
It is easy to show by using \eqref{ee12-hat}, Schwarz's inequality and Sobolev's inequality \eqref{sobolev2} that the last four terms in $\hat{R}_3$ are bounded by the right hand side of \eqref{ee11-hat}. It remains to handle only the first term, that is, the commutator term.
The $L^2$-norm of the commutator $\hat{u} \cdot \hat{\nabla} \partial^{\bm{a}} \mathcal{D} \hat{\vp} - \partial^{\bm{a}} \mathcal{D} (\hat{u} \cdot \hat{\nabla} \hat{\vp} )$ can be estimated as
\begin{align*}
&\|\hat{u} \cdot \hat{\nabla} \partial^{\bm{a}} \mathcal{D} \hat{\vp} - \partial^{\bm{a}} \mathcal{D} (\hat{u} \cdot \hat{\nabla} \hat{\vp} )\| 
\notag\\
&\leq \|\ut_1 \partial_1 \partial^{\bm{a}} \mathcal{D} \hat{\vp} - \partial^{\bm{a}} \mathcal{D} (\ut_1\partial_1 \hat{\vp} )\|
+\|(\hat{\psi}+\hat{U}) \cdot \hat{\nabla} \partial^{\bm{a}} \mathcal{D} \hat{\vp} - \partial^{\bm{a}} \mathcal{D} ((\hat{\psi}+\hat{U}) \cdot \hat{\nabla} \hat{\vp} )\| 
\notag\\
&= \|(\ut_1-u_+) \partial^{\bm{a}} \mathcal{D} \partial_1 \hat{\vp} - \partial^{\bm{a}} \mathcal{D} ((\ut_1-u_+)\partial_1 \hat{\vp} )\|
+\|(\hat{\psi}+\hat{U}) \cdot \hat{\nabla} \partial^{\bm{a}} \mathcal{D} \hat{\vp} - \partial^{\bm{a}} \mathcal{D} ((\hat{\psi}+\hat{U}) \cdot \hat{\nabla} \hat{\vp} )\| 
\notag\\
&\lesssim (N_\beta(T)+\delta)D_{3,\beta},
\end{align*}
where we have written explicitly $\hat{u}$ and used the triangular inequality in deriving the first inequality; we have expanded the derivative operators and applied Sobolev's inequalities \eqref{sobolev0} and \eqref{sobolev2} deriving the last inequality.
Using this, one can check that the integrations of the first two terms in $\hat{R}_3$ 
is also bound by the right hand side of \eqref{ee11-hat}.
Therefore we conclude that \eqref{ee11-hat} holds.

We multiply (\ref{ee7-hat}) by the time weight function $e^{\zeta t}$, 
integrate the resultant equality over $(0,t) \times \mathbb{R}^3_+$,
substitute (\ref{ee8-hat}) and (\ref{ee11-hat}) into the result,
let $\epsilon + N_\beta(T) + \delta$ be small enough,
and use (\ref{ee10-hat}).
Performing change of variables $y \rightarrow x$ for some terms on right hand side, 
we arrive at the desired inequality (\ref{ee0-hat}). 
This completes the proof of the lemma.
\end{proof}

\subsection{Cattabriga estimates} \label{ss-CattabrigaEst}

We complete the good contribution of spatial-derivatives 
using the Cattabriga estimate in Lemma \ref{CattabrigaEst}.
We remark that the Cattabriga estimate has crucial dependence on $\Omega$. The other estimates rely on Hardy's inequality, Sobolev's inequalities, Gagliardo-Nirenberg inequality and the commutator estimates, which depend on Sobolev's norms of $M$. Recall that in Subsection \ref{ss-Notation} all the constants depend on the Sobolev's norms of $M$. 

\begin{lemma} \label{lm1-C}
Under the same assumption as in Proposition \ref{apriori1} with $m=3$,
it holds, for $k=0, 1, 2$, 
\begin{align}
\| \nabla^{k+2} \psi \|^2 + \| \nabla^{k+1} \vp \|^2 
 \lesssim_{\Omega} \| \psi_t \|^2_{H^{k}} 
+ \| \nabla \psi \|^2_{H^{k}} 
+ \left\| \frac{d}{dt} \vp \right\|^2_{H^{k+1}} \!
+ (N_\beta (T) + \delta ) D_{3,\beta}+\delta.
\label{ef0-C}
\end{align}
\end{lemma}
\begin{proof}
From \eqref{eq-pv}, and recalling $\frac{d}{dt} = \partial_t + u \cdot \nabla $, we obtain
a boundary value problem of the Stokes equation:
\begin{equation} \label{ef7-C}
 \rho_+\div\psi = V , \quad
 - \mu_1 \Delta \psi + p'(\rho_+)\nabla \vp = W, 
\quad
 \psi|_{\partial \Omega} =0 , \quad \lim_{|x| \rightarrow \infty} |\psi| =0,
\end{equation}
where
\begin{align*}
V : =& f + F - \frac{d}{dt}\vp - (\rho-\rho_+)\div \psi , \\
W : =& - \rho \{ \psi_t + ( u \cdot \nabla ) \psi \} + (\mu_1 + \mu_2) \nabla \div \psi + g + G - (p'(\rho)-p'(\rho_+)) \nabla \vp
\\
=& - \rho \{ \psi_t + ( u \cdot \nabla ) \psi \} + (\mu_1 + \mu_2) \rho_+^{-1} \nabla V + g + G - (p'(\rho)-p'(\rho_+)) \nabla \vp.
\end{align*}
Applying the Cattabriga estimate \eqref{Cattabriga} to problem \eqref{ef7-C}, we have 
\begin{equation*}
\| \nabla^{k+2} \psi \|^2 + \| \nabla^{k+1} \vp \|^2 \lesssim_{\Omega} \| V \|^2_{H^{k+1}} + \| W \|^2_{H^k} + \| \nabla \psi \|^2.
\end{equation*}
It is straightforward to show that $\| V \|^2_{H^{k+1}}$ and $\| W \|^2_{H^k}$ are bound by the right hand side of \eqref{ef0-C}. Indeed, we can use the same method as in the derivation of \eqref{ee13-hat} to estimates the terms $f$ and $g$. The other terms can be estimated by using \eqref{h1} and Sobolev's inequalities \eqref{sobolev0} and \eqref{sobolev2}. Therefore we conclude \eqref{ef0-C}.
\end{proof}

We also show similar estimates for $(\hat{\vp},\hat{\psi})(t,y)$, where $y \in \mathbb R^3_+$.
For the notational convenience, we denote
\begin{equation}\label{checkOp}
\check{\partial}_{y_j}: = \sum_{i=1}^3 (A^{-1}(x'))_{ij} \partial_{x_i}. 
\end{equation}
where $A$ is defined in \eqref{CV2}; $(A^{-1}(x'))_{ij}$ means the $(i,j)$-component of $A^{-1}(x')$;
$\check{\partial}_{y_j} \vp (t,\Gamma(y)) = \partial_{y_j} \hat{\vp}(t,y)$ holds.
Furthermore, $\check{\nabla}^l_{y'}$ means the totality of all $l$-times tangential derivatives 
$\check{\partial}_{y_j}$ only for $j=2,3$.
Then applying $\check{\nabla}^l_{y'}$ to \eqref{ef7-C}, we obtain a boundary value problem of the Stokes equation: 
\begin{equation}\label{checkEq}
\rho_+\div \check{\nabla}^l_{y'} \psi = \check{V}, \ \
- \mu_1 \Delta \check{\nabla}^l_{y'} \psi + p'(\rho_+)\nabla \check{\nabla}^l_{y'} \vp = \check{W}, \ \  
\check{\nabla}^l_{y'} \psi|_{\partial \Omega} =0 , \ \ \lim_{|x| \rightarrow \infty} | \check{\nabla}^l_{y'} \psi| =0,
\end{equation}
where
\begin{align*}
\check{V} : = 
& \check{\nabla}^l_{y'} (f+F) - \check{\nabla}^l_{y'} \frac{d}{dt}\vp - \check{\nabla}^l_{y'} [(\rho-\rho_+)\div \psi ] + [\rho_+ \div \check{\nabla}^l_{y'} \psi -  \rho_+ \check{\nabla}^l_{y'} \div \psi] , \\
\check{W} : = 
&- \check{\nabla}^l_{y'} [ \rho \{ \psi_t + ( u \cdot \nabla ) \psi \} ] + (\mu_1 + \mu_2) \check{\nabla}^l_{y'}  \nabla \div \psi + \check{\nabla}^l_{y'} g + \check{\nabla}^l_{y'} G \\
& - \check{\nabla}^l_{y'} [ (p'(\rho)-p'(\rho_+)) \nabla \vp ]  + \mu_1 (\check{\nabla}^l_{y'} \Delta \psi - \Delta \check{\nabla}^l_{y'} \psi) 
+ p'(\rho_+) ( \nabla \check{\nabla}^l_{y'} \vp - \check{\nabla}^l_{y'} \nabla \vp ).  
\end{align*}

\begin{lemma} \label{lm2-C}
Under the same assumption as in Proposition \ref{apriori1} with $m=3$,
it holds, for $k=0,  1$, $k+l = 1, 2$, 
\begin{align}
&\| \nabla^{k+2}  \nabla^l_{y'} \hat{\psi} \|_{L^2(\mathbb R^3_+)}^2 
+ \|\nabla^{k+1} \nabla^l_{y'} \hat{\vp} \|_{L^2(\mathbb R^3_+)}^2 
\notag \\
& \lesssim_{\Omega} 
\left\|{\nabla}^l_{y'} \hat{\frac{d}{dt}} \hat{\vp} \right\|^2_{H^{k+1}(\mathbb R^3_+)} \!
+ \| \psi_t \|^2_{H^{k+l}} 
+ \| \nabla \psi \|^2_{H^{k+l}} 
+ \| \nabla \vp \|^2_{H^{k+l-1}}
+ (N_\beta (T) + \delta ) D_{3,\beta}+ \delta.
\label{eg0-C}
\end{align}
\end{lemma}
\begin{proof}
We apply the Cattabriga estimate \eqref{Cattabriga} to \eqref{checkEq} and obtain 
\begin{equation}
\| \nabla^{k+2} \check{\nabla}^l_{y'} \psi \|^2 + \|\nabla^{k+1} \check{\nabla}^l_{y'} \vp \|^2 \lesssim_{\Omega} \|\check{V} \|^2_{H^{k+1}} + \|\check{W}\|^2_{H^k} + \|\check{\nabla}^l_{y'} \psi \|^2.
\label{eg1-C}
\end{equation}
The terms $\|\check{V} \|^2_{H^{k+1}}$ and $\|\check{W}\|^2_{H^k}$ can be estimated 
in the same way as in the proof of Lemma \ref{lm1-C}. Now we conclude from \eqref{eg1-C} that  
\begin{align*}
&\| \nabla^{k+2} \check{\nabla}^l_{y'} \psi \|^2 
+ \|\nabla^{k+1} \check{\nabla}^l_{y'} \vp \|^2  
\\
&\lesssim_{\Omega} 
\left\|\check{\nabla}^l_{y'} \frac{d}{dt} \vp \right\|^2_{H^{k+1}} \!
+ \| \psi_t \|^2_{H^{k+l}} 
+ \| \nabla \psi \|^2_{H^{k+l}} 
+ \| \nabla \vp \|^2_{H^{k+l-1}}
+ (N_\beta (T) + \delta ) D_{3,\beta}+ \delta.
\end{align*}
Then changing the coordinate $x \in \Omega$ to the coordinate $y \in \mathbb R_+^3$ 
in the left hand side of this inequality 
and also using $\check{\partial}_{y_j} \vp (t,\Gamma(y)) = \partial_{y_j} \hat{\vp}(t,y)$, 
we arrive at \eqref{eg0-C}.
\end{proof}

Let us complete the derivation of the dissipative terms for the spatial derivatives.

\begin{lemma} \label{lm1-comp}
Under the same assumption as in Proposition \ref{apriori1} with $m=3$,
it holds that
\begin{align}
&   e^{\zeta t} \| \nabla^{p+1}  \vp (t) \|^2 + \int^t_0 e^{\zeta \tau} \left\| \nabla^{p+1}   \frac{d}{dt} \vp (\tau) \right\|^2  d \tau   
\notag \\
&\lesssim_{\Omega} \| \Phi_0\|^2_{H^3} + e^{\zeta t} \| \vp (t) \|_{H^p}^2 + \int^t_0 e^{\zeta \tau} \|\psi_t(\tau) \|^2_{H^p} \, d \tau  
+ \epsilon^{-1}\int^t_0 e^{\zeta \tau} D_{p, \beta} (\tau) \, d \tau 
\notag \\
& \qquad + \epsilon \int^t_0  e^{\zeta \tau}(\| \nabla \vp (\tau) \|^2_{H^{p}} +\| \nabla \psi (\tau) \|^2_{H^{p+1}} ) d \tau 
\notag \\
& \qquad + (N_\beta (T) + \delta + \zeta) \int^t_0 e^{\zeta \tau} D_{3, \beta} (\tau) \, d \tau
+ \delta  \int^t_0 e^{\zeta \tau}  \, d \tau
\label{el0-comp}
\end{align}
for $\epsilon \in (0,1)$ and $p=0,1,2$.
\end{lemma}

\begin{proof}
As in Lemma \ref{lm4-hat}, we define the index $\bm{a} = (a_1, a_2, a_3)$ with $a_1, \ a_2, \ a_3 \geq 0$ and $|\bm{a}| := a_1 + a_2 + a_3$. Let $\partial^{\bm{a}} := \partial^{a_1}_{y_1} \partial^{a_2}_{y_2} \partial^{a_3}_{y_3}$. It suffices to prove that for $j = 1, 2, \ldots , p+1$ and $\bm{a} = (a_1, a_2, a_3)$,
\begin{equation}
\sum_{|\bm{a}| = p+1, a_1 \leq j} \left( e^{\zeta t} \| \partial^{\bm{a}}_y \hat{\vp} (t) \|_{L^2(\mathbb R^3_+)}^2 
+ \int^t_0 e^{\zeta \tau} \left\| \partial^{\bm{a}}_y \hat{\frac{d}{dt}}\hat{\vp} (\tau) \right\|_{L^2(\mathbb R^3_+)}^2 d \tau \right)
\lesssim_{\Omega} \mathcal{R}_{p,\epsilon},
\label{ek0}
\end{equation}
where 
\begin{equation*}
\begin{split}
\mathcal{R}_{p,\epsilon} :=& \| \Phi_0\|^2_{H^3} + \int^t_0 e^{\zeta \tau} \|\psi_t (\tau) \|^2_{H^p} \, d \tau  
+ \epsilon^{-1}\int^t_0 e^{\zeta \tau} D_{p, \beta} (\tau) \, d \tau \\
& + \epsilon \int^t_0  e^{\zeta \tau}(\| \nabla \vp (\tau) \|^2_{H^{p}} +\| \nabla \psi (\tau) \|^2_{H^{p+1}} ) d \tau \\
& + (N_\beta (T) + \delta + \zeta) \int^t_0 e^{\zeta \tau} D_{3, \beta} (\tau) \, d \tau
+ \delta  \int^t_0 e^{\zeta \tau}  \, d \tau. 
\end{split}
\end{equation*}
Indeed the desired estimate \eqref{el0-comp} follows from changing the coordinate $y \in \mathbb R_+^3$ to the coordinate $x \in \Omega$ in the left hand side of \eqref{ek0} with $j=p+1$.

To obtain \eqref{ek0} with $j=1$, we add up \eqref{ec0-hat} with $l = p+1$ 
and \eqref{ee0-hat} with $a_1=0$, $a_2 + a_3 = p$, 
and estimate 
$\int^t_0 e^{\zeta \tau} \| \partial^{\bm{a}} \nabla_{y'} \hat{\frac{d}{dt}} \hat{\vp} \|_{L^2}^2  \, d \tau$ 
and 
$\int^t_0 e^{\zeta \tau} \| \partial^{\bm{a}} \nabla_y \nabla_{y'} \hat{\psi} \|_{L^2}^2  \, d \tau$ 
by using \eqref{ec0-hat}.
Now, assuming \eqref{ek0} holds for $j=q$, 
we show that it holds for $j=q+1$. 
We take the weighted-in-time integral of \eqref{eg0-C} with $l = p+1-q$ and $k=q-1$, 
and use \eqref{ek0} with $j = q$ 
to estimate the highest-order term in $\|{\nabla}^{p+1-q}_{y'} \hat{\frac{d}{dt}} \hat{\vp} \|^2_{H^q} $. 
Then we arrive at
\begin{equation}
\int^t_0 e^{\zeta \tau}  \left( \| \nabla^{q+1}  \nabla^{p+1-q}_{y'} \hat{\psi}(\tau) \|_{L^2(\mathbb R^3_+)}^2 + \| \nabla^q  \nabla^{p+1-q}_{y'} \hat{\vp} (\tau) \|_{L^2(\mathbb R^3_+)}^2 \right) \, d \tau 
\lesssim_{\Omega} \mathcal{R}_{p,\epsilon}.
\label{ek2}
\end{equation}
Taking $a_1 =q$ and $ a_2 + a_3 = p-q$ in \eqref{ee0-hat}, and using \eqref{ek0} with $j = q$ and \eqref{ek2} to estimate the second term on the right hand side of \eqref{ee0-hat},
we obtain
\begin{equation}
e^{\zeta t} \| \partial^{q+1}_1 \nabla^{p-q}_{y'} \hat{\vp} (t) \|_{L^2(\mathbb R^3_+)}^2 + \int^t_0 e^{\zeta \tau}  \left\|\partial^{q+1}_1 \nabla^{p-q}_{y'} \hat{\frac{d}{dt}} \hat{\vp}(\tau)  \right\|_{L^2(\mathbb R^3_+)}^2  \, d \tau 
\lesssim_{\Omega} \mathcal{R}_{p,\epsilon}.
\label{ek3}
\end{equation} 
Combining \eqref{ek3} and \eqref{ek0} with $j=q$, we obtain \eqref{ek0} with $j=q+1$. The proof is completed by induction. 
\end{proof}


Moreover, we have the following lemma.
\begin{lemma} \label{lm3-comp}
Under the same assumption as in Proposition \ref{apriori1} with $m=3$,
it holds, for $p=0,1,2$, 
\begin{align}
&   e^{\zeta t} \| \nabla^{p+1}  \vp (t) \|^2 + \int^t_0 e^{\zeta \tau}  \|( \nabla^{p+1} \vp, \nabla^{p+2} \psi )(\tau)\|^2   \, d \tau  + \int^t_0 e^{\zeta \tau} \left\| \nabla^{p+1}   \frac{d}{dt} \vp (\tau) \right\|^2 d \tau   
\notag \\
&\lesssim_{\Omega} \| \Phi_0\|^2_{H^3} + e^{\zeta t} \| \vp (t) \|_{H^p}^2 + \int^t_0 e^{\zeta \tau} \|\psi_t(\tau) \|^2_{H^p} \, d \tau  + \int^t_0 e^{\zeta \tau} D_{p, \beta} (\tau) \, d \tau 
\notag\\
& \qquad + (N_\beta (T) + \delta + \zeta) \int^t_0 e^{\zeta \tau} D_{3, \beta} (\tau) \, d \tau
+ \delta  \int^t_0 e^{\zeta \tau}  \, d \tau. 
\label{em0-comp}
\end{align}
\end{lemma}
\begin{proof}
Take the weighted-in-time integral of \eqref{ef0-C} with $k =p$ and
combine it together with \eqref{el0-comp}. 
Then, letting $\epsilon$ be suitably small, we arrive at \eqref{em0-comp}.
\end{proof}

\subsection{Elliptic estimates} \label{EllipticEst}

Using the elliptic estimate (Lemma \ref{ellipticEst}),
we rewrite some terms for the time-derivatives into terms for the spatial-derivatives.

\begin{lemma} \label{lm1-E}
Under the same assumption as in Proposition \ref{apriori1} with $m=3$, it holds that
\begin{gather}
\| \nabla^{k+2} \psi \| \lesssim  \| \psi_t \|_{H^{k}} + E_{k+1,\beta}+ \delta, \quad k=0,1,
\label{eh0-E} \\
\| \nabla^{2}   \psi_t \|  \lesssim  \|\psi_{tt} \| +D_{2,\beta}.
\label{ej0-E}
\end{gather}
\end{lemma}
\begin{proof}
Let us first show \eqref{eh0-E}.
From \eqref{eq-pv}, we have the elliptic boundary value problem:
\begin{equation} \label{eh1-E}
-\mu_1 \Delta \psi -(\mu_1 + \mu_2) \nabla \div \psi = - \rho_+ \psi_t - p'(\rho_+)\nabla \vp + H, \quad 
\psi|_{\partial\Omega}=0, \quad \lim_{|x|\to\infty}\psi=0,
\end{equation}
where
\begin{equation*}
H := -(\rho - \rho_+) \psi_t - (p'(\rho) -p'(\rho_+)) \nabla \vp -\rho (u \cdot \nabla) \psi +g +G. 
\end{equation*}
Applying Lemma \ref{ellipticEst}, we obtain, for $k=0, \ 1$, 
\begin{equation*}
\| \nabla^{k+2} \psi \|  \lesssim  \|\psi_t \|_{H^{k}} + \| \nabla \vp \|_{H^k} +\|H \|_{H^k} + \|  \psi \|.  
\end{equation*}
It is straightforward to check that the term $\|H \|_{H^k}$ is bound by the right hand side of \eqref{eh0-E}. 
Therefore, we conclude \eqref{eh0-E}.

We next apply $\partial_t$ to \eqref{eh1-E} to obtain
\begin{equation*}
-\mu_1 \Delta \psi_t - (\mu_1 + \mu_2) \nabla \div \psi_t 
= - \rho_+\psi_{tt} -p'(\rho_+)\nabla \vp_t + H_t, \quad 
\psi_t |_{\partial \Omega} =0, \quad \lim_{|x|\rightarrow \infty} \psi_t =0. 
\end{equation*}
Applying Lemma \ref{ellipticEst} with $k=0$ again to this boundary value problem gives
\begin{equation*}
\| \nabla^{ 2} \psi_t \|  \lesssim  \|\psi_{tt} \|  + \| \nabla \vp_t \|  +\|H_t \| + \|  \psi_t \|.
\end{equation*}
Furthermore, $\|H_t \|$ can be estimated easily by the right hand side of \eqref{ej0-E}.
From the discussion above, we obtain \eqref{ej0-E}.
\end{proof}

\subsection{Completion of the a priori estimates} \label{ss-comp-apriori}

Now we can complete the a priori estimates. 

\begin{proof}[Proof of Proposition \ref{apriori1}]
From \eqref{ea0}, we know that
\begin{align}
& e^{\zeta t} E_{0, \beta} (t) + \int^t_0 e^{\zeta \tau} D_{0,\beta} (\tau ) \, d \tau  
\notag \\
& \lesssim \|\Phi_0\|_{\lteasp{\beta}}^2 + \| \Phi_0\|^2_{H^3} 
+ (N_\beta (T) + \delta +\zeta) \int^t_0 e^{\zeta \tau} D_{3, \beta} (\tau) \, d \tau
+ \delta  \int^t_0 e^{\zeta \tau}  \, d \tau.   
\label{en0-comp}
\end{align}
We next show that for $l=1,2,3$,
\begin{align}
& e^{\zeta t} E_{l, \beta} (t) + \int^t_0 e^{\zeta \tau} D_{l,\beta} (\tau ) \, d \tau  
\notag \\
& \lesssim_{\Omega}  \| \Phi_0\|^2_{H^3} 
+ e^{\zeta t} E_{l-1, \beta} (t) + \int^t_0 e^{\zeta \tau} D_{l-1,\beta} (\tau ) \, d \tau
\notag \\
& \qquad + (N_\beta (T) + \delta +\zeta) \int^t_0 e^{\zeta \tau} D_{3, \beta} (\tau) \, d \tau
+ \delta  \int^t_0 e^{\zeta \tau}  \, d \tau + (l-1) \delta e^{\zeta \tau} .   
\label{en1-comp}
\end{align}

Let us first treat the case $l=1$.
Multiply \eqref{em0-comp} with $p =0$ by a positive constant $\nu$ and add the result to 
\eqref{ed0} with $k=0$. Taking $\nu$ and $\lambda$ suitably small, we obtain 
\eqref{en1-comp} with $l=1$.

Next, for the case $l=2$, we recall \eqref{PhiH1}, then multiply \eqref{em0-comp} with $p =1$ and 
\eqref{eh0-E} with $k=0$ by $\nu$ and $\nu e^{\zeta t}$, respectively.
Adding up the two results and \eqref{ec0}
and then taking $\nu$ small, we have \eqref{en1-comp} with $l=2$.

Lastly, for the case $l=3$, we multiply \eqref{ej0-E} by $\nu e^{\zeta \tau}$ and integrate it over $(0,t)$.
We then multiply \eqref{eh0-E} with $k=1$ and 
\eqref{em0-comp} with $p =2$ by $\nu$ and $\nu^2$, respectively.
Add up these three results and \eqref{ed0} with $k=1$. 
Then taking $\nu$ and $\lambda$ suitably small yields \eqref{en1-comp} with $l=3$.

The estimates \eqref{en0-comp} and \eqref{en1-comp} imply that
\begin{equation*}
\begin{split}
& e^{\zeta t} E_{3, \beta} (t) + \int^t_0 e^{\zeta \tau} D_{3,\beta} (\tau ) \, d \tau  \\
& \lesssim_{\Omega}  \|\Phi_0\|_{\lteasp{\beta}}^2 + \| \Phi_0\|^2_{H^3} 
+ (N_\beta (T) + \delta +\zeta) \int^t_0 e^{\zeta \tau} D_{3, \beta} (\tau) \, d \tau
+ \delta  \int^t_0 e^{\zeta \tau}  \, d \tau + \delta e^{\zeta \tau}.   \\  
\end{split}
\end{equation*}
Letting $N_\beta (T) + \delta +\zeta$ be small enough, one can have \eqref{apes0}.
Dividing \eqref{apes0} by $e^{\zeta \tau}$ and using equations \eqref{eq-pv1} and \eqref{eq-pv2}
to obtain the estimates for the time-derivatives, we conclude \eqref{apes1}.
\end{proof}

We discuss briefly the proof of Corollary \ref{cor1}.

\begin{proof}[Proof of Corollary \ref{cor1}]
If $\|M\|_{H^{9}(\mathbb R^2)} \leq \kappa$ holds for $\kappa$ being in Lemma \ref{CattabrigaEst},
we can replace $\lesssim_{\Omega}$ by $\lesssim$ 
in inequalities \eqref{ef0-C}, \eqref{eg0-C}, and \eqref{em0-comp}.
Then following the proof of Proposition \ref{apriori1} with these improved inequalities, 
we conclude Corollary \ref{cor1}.
\end{proof}

\section{Construction of stationary solutions}\label{S5}
For the construction of stationary solutions, 
we make use of the time-global solution $\Phi$ in Theorem \ref{global1}. 
We first prove an unique result Proposition \ref{5.1} 
for the time-periodic solutions to \eqref{eq-pv1}--\eqref{pbc}. 
Then we consider $\Phi$ and its translated version $\Phi^k(t,x):=\Phi(t+kT^*,x)$ 
for any $T^* > 0$ and $k=1,2,3,\ldots$. 
We prove in Proposition \ref{5.3} that $\{\Phi^k\}$ is a Cauchy sequence 
in the Banach space $C([0,T^*];H^{m-1}(\Omega)) \cap C^1([0,T^*];H^{m-3}(\Omega))$ 
and obtain a limit $\Phi^*$ from it, 
then we show in Proposition \ref{5.3} 
that $\Phi^*$ is a time-periodic solution to problem \eqref{eq-pv1}--\eqref{pbc} with period $T^*>0$. 
In Subsection \ref{S5.2}, 
using uniqueness of time-periodic solutions, 
we prove that $\Phi^*$ is actually time-independent and 
therefore gives a stationary solution to \eqref{eq-pv1}--\eqref{pbc}.

We remark that it is also possible to show directly that
$\Phi^k_t$ converges to zero as $k\to\infty$ by differentiating equations \eqref{eq-pv}
with respect to $t$ and then applying the energy method used in Section \ref{sec3} 
to the resultant equations. 
However the computations are very long.
Therefore we adopt the method mentioned above to construct the stationary solution.

\subsection{Time-periodic solutions}\label{S5.1}
\subsubsection{Uniqueness}\label{S5.1.1}
In this subsection, we show the uniqueness of time-periodic solutions to the problem 
of equations \eqref{eq-pv1} and \eqref{eq-pv2} 
with boundary condition \eqref{pbc} in the solution space 
\begin{equation*}
\mathcal{X}^{\text{\rm e}}_{m,\beta} (0,T)
=X^{\text{\rm e}}_{m-1,\beta} (0,T)
\cap L^\infty(0,T ; H^m(\Omega)).
\end{equation*}

\begin{proposition}\label{5.1}
Let \eqref{super1} and \eqref{cd-st} hold.
For $\beta>0$ being in Theorem \ref{global1},
there exists $\epsilon>0$ depending on $\|M\|_{H^9}$ but independent of $\Omega$ such that 
if a time-periodic solution $\Phi^* \in \mathcal{X}^{\text{\rm e}}_{3,\beta} (0,T)$
with a period $T^*>0$ to problem \eqref{eq-pv1}--\eqref{pbc}
exists and satisfies the following inequality, then it is unique:
\begin{equation}\label{uniasp1}
\sup_{t\in[0,T^*]}(\|\Phi^*(t)\|_{H^3}
+\|\partial_t \vp^*(t)\|_{H^{2}}
+\|\partial_t \psi^*(t)\|_{H^{1}})
+\delta \leq \epsilon.
\end{equation}
\end{proposition}

Let $\Phi^*=(\vp^*,\psi^*)$ and $\Phi^\#=(\vp^\#,\psi^\#)$ 
be time-periodic solutions to \eqref{eq-pv1}--\eqref{pbc}.
It is straightforward to see that
$\overline{\Phi} = (\overline{\vp}, \overline{\psi} )= \Phi^* - \Phi^\#$ 
satisfies the system   
\begin{subequations} \label{5.1-eq1}
\begin{gather}
\overline{\vp}_{t} + (\tilde{u} + U + \psi^* ) \cdot \nabla \overline{\vp}  + (\tilde{\rho} + \vp^* ) \div \overline{\psi}  
=\overline{f}, \label{5.1-eq1-1}
\\
(\tilde{\rho} + \vp^*) \{ \overline{\psi}_{t} + (\tilde{u} + U + \psi^*) \cdot \nabla \overline{\psi}  \} - L \overline{\psi} + p'(\tilde{\rho} + \vp^*) \nabla \overline{\vp} 
= \overline{g}. \label{5.1-eq1-2}
\end{gather}
The boundary condition for $(\overline{\vp},\overline{\psi})$ 
follow from (\ref{pbc}) as
\begin{gather*}
\overline{\psi} (t,M(x'),x') = 0.
\end{gather*}
\end{subequations}
Here $\overline{f}$ and $\overline{g}$ are defined by
\begin{align*}
\overline{f} &:=
- \nabla\tilde{\rho} \cdot \overline{\psi}  - \ut' \overline{\vp}  - \overline{\vp} \div U - \overline{\psi} \cdot \nabla \vp^\# - \overline{\vp} \div \psi^\#,
\\
\overline{g} &:= - \left(\{(\rt\!+\!\vp^*)\overline{\psi}+\overline{\vp} \psi^\# \}\!\cdot\! \nabla \right) (\tilde{u}\!+\!U) - \overline{\varphi} ((\tilde{u}\!+\!U) \cdot \nabla) (\tilde{u}\!+\!U) -(p'(\rt\!+\!\vp^*)-p'(\rt\!+\!\vp^\#)) \nabla \tilde{\rho}
\\
&\qquad - \overline{\vp}  \psi^\#_t - \{(\rt+\vp^*)(\ut+U+\psi^*)-(\rt+\vp^\#)(\ut+U+\psi^\#) \}\cdot \nabla \psi^\# 
\\ 
&\qquad - ( p'(\tilde{\rho} + \vp^*) -p'(\tilde{\rho} + \vp^\#) ) \nabla \vp^\#.
\end{align*}
It is easy to check from \eqref{uniasp1} that 
\begin{equation*}
|\overline{f}| \lesssim \epsilon |\overline{\Phi}|, \quad |\overline{g}| \lesssim \epsilon |\overline{\Phi}| + |\overline{\vp}| |\psi^\#_t| . 
\end{equation*}

\begin{proof}[Proof of Proposition \ref{5.1}]
Denote $p'_* := p'(\tilde{\rho} + \vp^*)$ and $P_* := \frac{p'_*}{\tilde{\rho} + \vp^*}$. 
Multiplying \eqref{5.1-eq1-1} by $e^{\beta x_1} P_* \overline{\vp}$, we obtain
\begin{multline}\label{5.1-eq2}
\left(\frac{1}{2} e^{\beta x_1} P_* \overline{\vp}^2\right)_t  
+ \div\left(\frac{1}{2}  e^{\beta x_1} P_* (\tilde{u} + U + \psi^*) \overline{\vp}^2 \right)
- \frac{\beta}{2}   e^{\beta x_1} P_* (\tilde{u}_1 + U_1 + \psi^*_1 ) \overline{\vp}^2 
\\
+ e^{\beta x_1} p_*' \overline{\vp} \div\overline{\psi}
=e^{\beta x_1} P_* \overline{\vp} \overline{f}
+\frac{1}{2}  e^{\beta x_1} (P_*)_t \overline{\vp}^2 
+ \frac{1}{2}   e^{\beta x_1} \{\div (P_* (\tilde{u} + U + \psi^* ))\} \overline{\vp}^2.
\end{multline}
Multiplying \eqref{5.1-eq1-2} by $e^{\beta x_1} \overline{\psi}$ gives
\begin{align}
&\left(\frac{1}{2}  e^{\beta x_1} (\tilde{\rho} + \vp^*) |\overline{\psi}|^2\right)_t 
+ \div \left( \frac{1}{2} e^{\beta x_1} |\overline{\psi}|^2 (\tilde{\rho} + \vp^*) (\tilde{u} + U  + \psi^*) +e^{\beta x_1} p'_*\overline{\vp}\overline{\psi} \right) 
\notag \\
&\quad - \div \left(\mu_1 e^{\beta x_1} (\nabla \overline{\psi})\cdot \overline{\psi}
+(\mu_1 + \mu_2) e^{\beta x_1} (\div \overline{\psi}) \overline{\psi} \right)
\notag \\
&\quad - \frac{\beta}{2} e^{\beta x_1}(\tilde{\rho} + \vp^*) (\tilde{u}_1 + U_1 + \psi^*_1)|\overline{\psi}|^2
- \beta e^{\beta x_1} p'_*\overline{\vp}\overline{\psi}_1
\notag \\
&\quad + \mu_1 e^{\beta x_1} |\nabla \overline{\psi}|^2  
+ (\mu_1 + \mu_2)  e^{\beta x_1} |\div \overline{\psi}|^2
- e^{\beta x_1} p'_* \overline{\vp} \div\overline{\psi}
\notag \\
& = e^{\beta x_1} \overline{\psi} \cdot \overline{g}
+ \frac{1}{2} e^{\beta x_1} \vp^*_t |\overline{\psi}|^2
+ \frac{1}{2} e^{\beta x_1} |\overline{\psi}|^2 \div\{(\tilde{\rho} + \vp^*)(\tilde{u} + U + \psi^*)\}
+ e^{\beta x_1} \overline{\vp}\nabla p'_* \cdot \overline{\psi}
\notag \\
&\quad - \mu_1 \beta   e^{\beta x_1}  (\partial_1 \overline{\psi}) \cdot \overline{\psi} 
- (\mu_1 + \mu_2) \beta  e^{\beta x_1} (\div \overline{\psi}) \overline{\psi}_1.
\label{5.1-eq3}
\end{align}
Adding \eqref{5.1-eq2} and \eqref{5.1-eq3} yields that
\begin{multline}
\left(\frac{1}{2} e^{\beta x_1} P_* \overline{\vp}^2
+\frac{1}{2}  e^{\beta x_1} (\tilde{\rho} + \vp^*) |\overline{\psi}|^2\right)_t 
+\div\left(\overline{G}_1+\overline{B}_1\right)
-\beta(\overline{G}_1)_1
\\
+ \mu_1 e^{\beta x_1} |\nabla \overline{\psi}|^2  
+ (\mu_1 + \mu_2)  e^{\beta x_1} |\div \overline{\psi}|^2 
=\overline{R}_1+ \beta (\overline{B}_1)_1,
\label{5.1-eq4}
\end{multline}
\begin{align*}
\overline{G}_1:=& \frac{1}{2} e^{\beta x_1} P_* (\tilde{u} + U + \psi^*) \overline{\vp}^2
+\frac{1}{2} e^{\beta x_1} |\overline{\psi}|^2 (\tilde{\rho} + \vp^*) (\tilde{u} + U + \psi^*)
+e^{\beta x_1} p'_*\overline{\vp}\overline{\psi},
\\
\overline{B}_1:=&-\mu_1 e^{\beta x_1} (\nabla \overline{\psi})\cdot \overline{\psi}
-(\mu_1 + \mu_2) e^{\beta x_1} (\div \overline{\psi}) \overline{\psi},
\\
\overline{R}_1:=& \frac{1}{2}  e^{\beta x_1} (P_*)_t \overline{\vp}^2
+ \frac{1}{2}   e^{\beta x_1} \{\div (P_* (\tilde{u} + U + \psi^* ))\} \overline{\vp}^2 
+ \frac{1}{2} e^{\beta x_1} \vp^*_t |\overline{\psi}|^2
\notag \\
&+ \frac{1}{2} e^{\beta x_1} |\overline{\psi}|^2 \div\{(\tilde{\rho} + \vp^*)(\tilde{u} + U + \psi^*)\}
+ e^{\beta x_1} \overline{\vp}\nabla p'_* \cdot \overline{\psi}
+ e^{\beta x_1} P_* \overline{\vp} \overline{f}
+ e^{\beta x_1} \overline{\psi} \cdot \overline{g}.
\end{align*}

The second term on the left hand side of \eqref{5.1-eq4} is estimated from
below by using the divergence theorem, the fact $(u_b\cdot n) \geq c>0$, $ P_* \geq c>0$ and \eqref{pbc} as
\begin{equation}\label{5.1-eq5}
\int_\Omega \div(\overline{B}_1+\overline{G}_1) \,dx 
= \int_{\partial\Omega} \frac{1}{2} e^{\beta M(x')}P_* \overline{\vp}^2 (u_b\cdot n) d \sigma
\gtrsim \|\overline{\vp}(t,M(\cdot),\cdot) \|_{L^2(\mathbb R^2)}^2.
\end{equation}
Next we derive the lower estimate of the third term on the
left hand side of \eqref{5.1-eq4}.
We compute the term $\overline{G}_1$ as
\begin{gather*}
(\overline{G}_1)_1
=
\frac{1}{2} e^{\beta x_1} p'(\rho_+) \rho_+^{-1} u_+\overline{\vp}^2
+\frac{1}{2} e^{\beta x_1} |\overline{\psi}|^2 \rho_+u_+
+e^{\beta x_1} p'(\rho_+) \overline{\vp}\overline{\psi}_1
+ e^{\beta x_1} \overline{R}_{2},
\label{5.1-ea6}
\\
\begin{aligned}
\overline{R}_{2}
:=& \frac{1}{2} \left\{P_* (\tilde{u}_1 + U_1 + \psi^*_1 )- p'(\rho_+)\rho_+^{-1}u_+ \right\}\overline{\vp}^2
+\frac{1}{2}|\overline{\psi}|^2 \left\{(\tilde{\rho} + \vp^*) (\tilde{u}_1 + U_1 + \psi^*_1)-\rho_+u_+ \right\}
\\
&+\left\{p'_*-p'(\rho_+)\right\}\overline{\vp}\overline{\psi}_1.
\end{aligned}
\notag
\end{gather*}
Thus, by using this, the third term on the left hand side of \eqref{5.1-eq4} is rewritten as
\begin{gather*}
-\beta (\overline{G}_1)_1
=
\beta e^{\beta x_1}
\Bigl(
F(\overline{\vp},\overline{\psi}_1) + \frac{\rho_+ |u_+|}{2} |\overline{\psi}'|^2 - \overline{R}_{2}
\Bigr),
\\
F(\overline{\vp},\overline{\psi}_1)
:=
\frac{p'(\rho_+)}{2\rho_+}|u_+| \overline{\vp}^2
- p'(\rho_+) \overline{\vp} \overline{\psi}_1
+ \frac{\rho_+ |u_+|}{2} \overline{\psi}_1^2,
\nonumber
\end{gather*}
where $\overline{\psi}'$ is the second and third components of $\overline{\psi}$ defined by 
$\overline{\psi}' := (\overline{\psi}_2,\overline{\psi}_3)$.
Owing to the supersonic condition \eqref{super1}, 
the quadratic form $F(\overline{\vp},\overline{\psi}_1)$ becomes
positive definite since the discriminant of $F(\overline{\vp},\overline{\psi}_1)$ satisfies
\[
p'(\rho_+)^2 - p'(\rho_+) u_+^2
= p'(\rho_+)^2 \left(1 - \frac{u_+^2}{p'(\rho_+)}\right) < 0.
\]
On the other hand,
the remaining terms $\overline{R}_{2}$ satisfy
\begin{equation*}
|\overline{R}_{2}|
\lesssim |(\rt - \rho_+,\ut - \rho_+,U_1)| |\overline{\Phi}|^2 + |(\vp^*,\psi^*)| |\overline{\Phi}|^2
\lesssim \epsilon |\Phi|^2.
\end{equation*}
Therefore we get the lower estimate of the integration of
$-\beta (\overline{G}_1)_1$ as 
\begin{equation}
\int_{\Omega} -\beta (\overline{G}_1)_1 \, dx
\geq
\beta
(c - C \epsilon)
\| \overline{\Phi} \|^2_{L^2_{e, \beta} (\Omega)}.
\label{5.1-eq6}
\end{equation}
The right hand side of \eqref{5.1-eq4}
is estimated by using \eqref{uniasp1}, \eqref{5.1-eq4}, 
and the Schwarz inequality as
\begin{gather*}
|\overline{R}_1| \lesssim \epsilon e^{\beta x_1}|\overline{\Phi}|^2
+|\psi^\#_t|e^{\beta x_1}|\overline{\vp}||\overline{\psi}|,
\\
|\beta (B_1)_1| \lesssim \beta (\lambda e^{\beta x_1}|\overline{\psi}|^2
+\lambda^{-1}e^{\beta x_1}|\nabla\overline{\psi}|^2),
\end{gather*}
where $\lambda$ is an arbitrary positive constant.
Then the integrations are estimated by \eqref{uniasp1} and the Sobolev inequality as
\begin{gather}
\int_{\Omega}
|\overline{R}_1| \, dx
\lesssim  \epsilon \| \Phi \|^2_{L^2_{e, \beta}} 
+ \|\psi^\#_t\|_{L^4}\|e^{\beta x_1/2}\overline{\vp}\|_{L^2}
\|e^{\beta x_1/2}\overline{\psi}\|_{L^4}
\lesssim  \epsilon \| (\Phi,\nabla \overline{\psi} )\|^2_{L^2_{e,\beta}},
\label{5.1-eq7}
\\
\int_{\Omega}
|\beta (G_1)_1|  \, dx
\lesssim
\beta \big(\nu \| \overline{\psi} \|^2_{L^2_{e, \beta}} 
+ \nu^{-1}  \| \nabla \overline{\psi} \|^2_{L^2_{e, \beta}} \big).
\label{5.1-eq8}
\end{gather}

Integrate \eqref{5.1-eq4} over $\Omega$,
substitute the estimates \eqref{5.1-eq5}--\eqref{5.1-eq8} in the resultant equality
and then let $\lambda$, \footnote{Here it is enough to take the same $\beta$ as in the proof of Lemma \ref{lm1}}{$\beta$}, and $\epsilon$ suitably small to obtain
\begin{equation}\label{5.1-eq14}
\frac{1}{2} \frac{d}{dt} \int e^{\beta x_1} 
\left(P_* \overline{\vp}^2 + (\tilde{\rho} + \vp^*) |\overline{\psi}|^2 \right) dx 
+ c \beta \|\overline{\Phi}\|_{L^2_{e, \beta}}^2  
\leq 0. 
\end{equation}
Integrating \eqref{5.1-eq14} over $[0,T^*]$ and using the periodicity of solutions
lead to $\overline{\Phi}=0$.
The proof is complete.
\end{proof}

\subsubsection{Existence}\label{S5.1.2}

For the construction of time-periodic solutions, 
we use the time-global solution $\Phi$ in Theorem \ref{global1}.
Here we see from Lemma \ref{CompatibilityCond} in Appendix \ref{B}
that there exist initial data satisfying the conditions in Theorem \ref{global1}.
Now we define
\[
 \Phi^k(t,x):=\Phi(t+kT^*,x)
\quad \text{for $k=1,2,3,\ldots$.} 
\]

Let us first show the following lemma.

\begin{lemma}\label{5.2}
Let \eqref{super1} and \eqref{cd-st} hold.
For $\beta>0$ being in Theorem \ref{global1} and any $T^*>0$,
there exists $\gamma_0=\gamma_0(\Omega)>0$ and $C_0=C_0(\Omega)>0$ 
depending on $\Omega$ but independent of $k$ and $T^*$ such that
\begin{equation}\label{exiapes0}
\|(\Phi-\Phi^k)(t)\|_{\lteasp{\beta}}\leq C_0e^{-\gamma_0 t}
\quad \text{for $k=1,2,3,\ldots$.}
\end{equation}
\end{lemma}
\begin{proof}
For any $k$, $k'$, let $\overline{\Phi} = (\overline{\vp}, \overline{\psi} )= \Phi^k - \Phi^{k'}$ satisfy the system   
\begin{subequations}
\label{5.2-eq1}
\begin{gather}
\overline{\vp}_{t} + (\tilde{u} + U + \psi^k ) \cdot \nabla \overline{\vp}  + (\tilde{\rho} + \vp^k ) \div \overline{\psi}  
=f^{k,k'}, 
\label{5.2-eq1-1}
\\
(\tilde{\rho} + \vp^k) \{ \overline{\psi}_{t} + (\tilde{u} + U + \psi^k) \cdot \nabla \overline{\psi}  \} - L \overline{\psi} + p'(\tilde{\rho} + \vp^k) \nabla \overline{\vp} 
= g^{k,k'}. 
\label{5.2-eq1-2}
\end{gather}
The boundary condition for $(\vp,\psi)$ 
follow from (\ref{pbc}) as
\begin{gather*}
\overline{\psi} (t,M(x'),x') = 0.
\end{gather*}
\end{subequations}
Here $f^{k,k'}$ and $g^{k,k'}$ are defined by
\begin{align*}
f^{k,k'} \!&:=
- \!\nabla\tilde{\rho} \cdot \overline{\psi}  - \ut' \overline{\vp}  - \overline{\vp} \div U - \overline{\psi} \cdot \nabla \vp^{k'} - \overline{\vp} \div \psi^{k'},
\\
g^{k,k'} \!&:= - \left(\!\{(\rt\!+\!\vp^k)\overline{\psi}+\overline{\vp} \psi^{k'} \}\!\cdot\! \nabla \!\right)\! (\tilde{u}\!+\!U) - \overline{\varphi} ((\tilde{u}\!+\!U) \!\cdot\! \nabla) (\tilde{u}\!+\!U) -(p'(\rt\!+\!\vp^k)-p'(\rt\!+\!\vp^{k'})) \nabla \tilde{\rho}
\\
&\qquad - \overline{\vp}  \psi^{k'}_t - \{(\rt+\vp^k)(\ut+U+\psi^k)-(\rt+\vp^{k'})(\ut+U+\psi^{k'}) \}\cdot \nabla \psi^{k'} 
\\ 
&\qquad - ( p'(\tilde{\rho} + \vp^k) -p'(\tilde{\rho} + \vp^{k'}) ) \nabla \vp^{k'}.
\end{align*}

Repeat exactly the proof of Proposition \ref{5.1} 
with $k$ in place of $*$, $k'$ in place of $\#$, 
and \eqref{uniasp1} in place of \eqref{bound1} with small initial data,
we obtain
\begin{equation*}
\frac{1}{2} \frac{d}{dt} \int e^{\beta x_1} \left( \frac{p'(\tilde{\rho} + \vp^k)}{\tilde{\rho} 
+ \vp^k} \overline{\vp}^2  +  (\tilde{\rho} + \vp^k) |\overline{\psi}|^2 \right) dx 
+ c \beta \|\overline{\Phi}\|_{L^2_{e, \beta}}^2  \leq 0 ,  \\
\end{equation*}
which implies \eqref{exiapes0} once we take $k' =0$. The proof is complete.
\end{proof}

Now we can construct the time-periodic solutions:
\begin{proposition}\label{5.3}
Let \eqref{super1} and \eqref{cd-st} hold, and $m = 3,4,5$.
For $\beta>0$ being in Theorem \ref{global1} and any $T^*>0$,
there exists a constant $\epsilon>0$ independent of $T^*$ such that 
if $\delta \leq \epsilon$, then the problem of  \eqref{eq-pv1}--\eqref{pbc}
has a time-periodic solution $\Phi^*\in {\mathcal X}^m_\beta(0,T^*)$
with a period $T^*>0$. Furthermore, it satisfies
\begin{equation}\label{apes3}
\sup_{t\in[0,T^*]}(\|\Phi^*(t)\|_{H^m}
+\|\partial_t \vp^*(t)\|_{H^{m-1}}
+\|\partial_t \psi^*(t)\|_{H^{m-2}})
\leq C_0\delta,
\end{equation}
where $C_0=C_0(\Omega)>0$ is a constant depending on $\Omega$ but independent of $T^*$.
\end{proposition}

\begin{proof}
First of all, applying Theorem \ref{global1} to 
initial-boundary value problem \eqref{eq-pv},
we obtain a time-global solution $\Phi$ to \eqref{eq-pv} with
\eqref{bound1} and \eqref{exiapes0}.

Recall that $\Phi^k(t,x):=\Phi(t+kT^*,x)$ for any $T^* > 0$ and $k=1,2,3,\ldots$. Let us first prove that $\{\Phi^k\}$ is a Cauchy sequence in the Banach space 
$C([0,T^*];H^{m-1}(\Omega)) \cap C^1([0,T^*];H^{m-3}(\Omega))$.
For $k>k'$, one can see from \eqref{exiapes0} that 
for $k>k'$, there holds
\begin{align*}
\sup_{t \in [0,T^*]} \|(\Phi^k-\Phi^{k'})(t) \|_{\lteasp{\beta}}
&= \sup_{t \in [0,T^*]} \|(\Phi(t+kT^*)-\Phi(t+k'T^*)\|_{\lteasp{\beta}}
\\
&= \sup_{t \in [k'T^*,(k'+1)T^*]} \|\Phi(t)-\Phi(t+(k-k')T^*)\|_{\lteasp{\beta}}
\\
& \lesssim_{\Omega} e^{-\gamma k'T^*}.
\end{align*}
This and \eqref{bound1} together with the Gagliardo-Nirenberg inequalities \eqref{GN1} leads to 
\begin{align*}
\sup_{t \in [0,T^*]} \|(\Phi^k-\Phi^{k'})(t) \|_{H^{m-1}}
& \lesssim \sup_{t \in [0,T^*]} \|(\Phi^k-\Phi^{k'})(t)\|_{H^m}^{1-1/m}
\|(\Phi^k-\Phi^{k'})(t)\|_{\lteasp{\beta}}^{1/m}
\notag \\
& \lesssim_{\Omega} e^{-\gamma k'T^*/m}. 
\end{align*}
So, what is left is to show that $\{\Phi^k\}$ is a Cauchy sequence 
in $C^1([0,T^*];H^{m-3}(\Omega))$.

We have already known from the proof of Lemma \ref{5.2} that $\Phi^k-\Phi^{k'}$ satisfies \eqref{5.2-eq1}. From this and \eqref{bound1}, one can have
\[
|\partial_t (\Phi^k-\Phi^{k'})| \lesssim
|((\Phi^k-\Phi^{k'}),\nabla(\Phi^k-\Phi^{k'}),\nabla^2(\Phi^k-\Phi^{k'}))|
\]
which gives
\begin{equation*}
\|\partial_t (\Phi^k-\Phi^{k'})\| \lesssim_\Omega e^{-\gamma k'T^*/m}.
\end{equation*}
In the case $m=3$, this estimate is sufficient.
For the case $m=4,5$, this estimate and \eqref{bound1}
together with Gagliardo-Nirenberg inequalities leads to 
\begin{align}
{}&
\sup_{t \in [0,T^*]} \|\partial_t(\Phi^k-\Phi^{k'})(t) \|_{H^{m-3}}
\lesssim_{\Omega} e^{-\gamma k'T^*/\{m(m-3)\}}.
\notag
\end{align}
Hence, we see that  $\{\Phi^k\}$ is a Cauchy sequence 
and thus there exists a limit  $\Phi^*$ such that
\begin{equation}\label{converge1}
\Phi^k \to \Phi^* \quad \text{in} \ \
C([0,T];{\lteasp{\beta}}(\Omega)) \cap
\bigcap_{i=0}^1C^i([0,T^*];H^{m-1-2i}(\Omega)).
\end{equation}
It is straightforward to check that 
the limit $\Phi^*$ satisfies \eqref{eq-pv}.

Then we can check that $\Phi^* \in {\cal X}^m_\beta(0,T)$ as follows.
On the other hand, by a standard argument,
$\Phi^k(t)$ converges to $\Phi^*(t)$ weakly in $H^m(\Omega)$ 
for each $t \in [0,T^*]$ and also
\begin{equation}\label{exies8}
\sup_{t\in [0,T^*]}\|\Phi^*(t)\|_{H^m} 
\lesssim_\Omega \|\Phi_0\|_{\lteasp{\beta}}+\|\Phi_0\|_{H^m}+\delta
\end{equation}
follows from \eqref{bound1}. 
Hence, we conclude $\Phi^* \in L^\infty([0,T^*];H^{m}(\Omega))$.
It is also seen from system \eqref{eq-pv} that
$\partial_t \vp^* \in L^\infty([0,T^*];H^{m-1}(\Omega))$,
$\partial_t \psi^* \in L^\infty([0,T^*];H^{m-2}(\Omega))$,
\begin{gather}
\sup_{t\in [0,T^*]}\|\partial_t \vp^*(t)\|_{H^{m-1}} 
+\sup_{t\in [0,T^*]}\|\partial_t \psi^*(t)\|_{H^{m-2}} 
\lesssim_\Omega \|\Phi_0\|_{\lteasp{\beta}}+\|\Phi_0\|_{H^{m}}+\delta.
\label{exies9}
\end{gather}

Let us show that $\Phi^*$ is a time-periodic function with period $T^*>0$.
The sequences $\Phi^k(T^*,x)$ and $\Phi^{k+1}(0,x)$ 
converge to $\Phi^*(T^*,x)$ and $\Phi^*(0,x)$, respectively, 
as $k$ tends to infinity.
Notice that $\Phi^k(T^*,x)=\Phi^{k+1}(0,x)$ holds and so does $\Phi^*(T^*,x)=\Phi^*(0,x)$.
Hence, we have constructed a time-periodic solution $\Phi^*$ 
to problem \eqref{eq-pv1}--\eqref{pbc} 
in the function space ${\mathcal X}^{\text{e}}_{m,\beta}(0,T^*)$
in which the uniqueness has been shown.
What is left is to prove estimate \eqref{apes3}.
For the initial data $\Phi_0=\Phi_0^{\#}$ in Lemma \ref{CompatibilityCond},
we have another time-periodic solution by the above method.
However, Proposition \ref{5.1} together with estimates \eqref{exies8} and \eqref{exies9}
ensure that these periodic solutions are same.
Hence, \eqref{apes3} follows from plugging $\Phi_0=\Phi_0^{\#}$ into \eqref{exies8} and \eqref{exies9}.
The proof is complete.
\end{proof}

\subsection{Stationary solutions}\label{S5.2}

Now we show that the time-periodic solutions constructed in Subsection \ref{S5.1}
are time-independent, which gives us Theorem \ref{th4}.

\begin{proof}[Proof of Theorem \ref{th4}]
Proposition \ref{5.3} ensures the existence of
time-periodic solutions $\Phi^*$
of problem \eqref{eq-pv1}--\eqref{pbc} for any period $T^*$.
We remark that the smallness assumption for $\delta$ is independent of the period $T^*$.
Hence, one can have time-periodic solutions $\Phi^*$ with the period $T^*$ and
$\Phi^*_l$ with the period $T^*/2^l$ for $l \in \mathbb N$ 
under the same assumption for $\delta$.
Furthermore, $\Phi^*=\Phi^*_l$ follows from 
Proposition \ref{5.1}, since $\Phi^*$ and $\Phi^*_l$ are 
the time-periodic solutions with the period $T^*$ and satisfy \eqref{apes3}. 
Hence, we see that
\[
\Phi^*\left(0,x\right)=
\Phi^*\left(\frac{i}{2^l}T^*,x\right)
\quad \text{for $i=1,2,3,\ldots,2^l$ and $l=0,1,2,\ldots$.} 
\]
Because the set 
$\cup_{l \geq 0} \{{i}/{2^l} \ ; \ i=1,2,3,\ldots,2^l\}$
is dense in $[0,T^*]$,
we see from the continuity of $\Phi^*$
that $\Phi^*$ is independent of $t$.
Hence, $\Phi^s=\Phi^*$
is the desired solution to 
the stationary problem corresponding to problem \eqref{eq-pv}.
\end{proof}

\subsection{Stability with exponential weight functions}\label{S5.3}

We prove the stability of stationary solutions, which gives us Theorem \ref{th5}.

\begin{proof}[Proof of Theorem \ref{th5}]
Theorem \ref{global1} and Lemma \ref{5.2} ensure that
initial--boundary value problem \eqref{eq-pv} has 
a unique time-global solution satisfying \eqref{bound1} and \eqref{exiapes0}
if $\|\Phi_0\|_{\lteasp{\beta}}+\|\Phi_0\|_{H^m}$ and $\delta$ are small enough.
So, it suffices to show that this time-global solution $\Phi$ converges 
to the stationary solution solution $\Phi^s$ exponentially fast 
as $t$ tends to infinity.
Passing the limit $k\to \infty$ in \eqref{exiapes0},
we have $\|(\Phi-\Phi^s)(t)\|_{\lteasp{\beta}} \lesssim_\Omega e^{-\gamma t}$
thanks to \eqref{converge1} and $\Phi^s=\Phi^*$.
This and \eqref{bound1} together with Gagliardo-Nirenberg inequalities \eqref{GN1} 
and Sobolev's inequality \eqref{sobolev2} lead to 
\begin{equation*}
\sup_{x\in \Omega}|(\Phi-\Phi^s)(t,x)| \lesssim_\Omega e^{-\gamma t},
\end{equation*}
where $\gamma$ is a positive constant independent of $t$.
Hence, the proof is complete.
\end{proof}

\subsection{Corollary}

We discuss briefly the proof of Corollary \ref{cor2}.

\begin{proof}[Proof of Corollary \ref{cor2}]
From Corollary \ref{cor1}, we have an improved estimate \eqref{bound1} 
with constants $C_0=C_0(\beta)$ and $\zeta=\zeta(\beta)$ independent of $\Omega$. 
In the same way as in Subsections \ref{S5.1}--\ref{S5.3} with the improved estimate, 
we can conclude Corollary \ref{cor2}.
\end{proof}

\section{Stability with no weight function}\label{S6}

In this section we discuss Theorem \ref{th3}, which gives the stability of $(\rho^s, u^s)$ in $H^3$.
Here we do not assume $(\rho_0-\rho^s, u_0-u^s) \in \lteasp{\beta}$.

For $(\rt, \ut)$ in Proposition \ref{ex-st},
$U$ in \eqref{ExBdry0},
 and $\Phi^s$ in Theorem \ref{th4}, let us set
\begin{equation*}
(\rho^s, u^s)(x) : = (\rt, \ut)(\tilde{M}(x)) + (0,U)(x)+ \Phi^s(x) .
\end{equation*}
Then it is obvious that $(\rho^s, u^s)$ satisfies \eqref{snse}.
We also introduce the perturbations
\begin{gather*}
 (\vp,\psi)(t,x)
:=
(\rho, u)(t,x)
-
(\rho^s, u^s)( x),
\ \; \text{where} \ \;
\psi = (\psi_1,\psi_2,\psi_3). 
\end{gather*}
%
For notational convenience, we use norms $E_{m,0}$ and $D_{m,0}$,
which are defined in the same way as \eqref{Ekbeta-def} and \eqref{Dkbeta-def}, 
with $\Phi = ( \vp, \psi)(t,x)$ being replaced by the functions defined right above. 
Furthermore, we also define 
\[
N (T) := \sup_{0\leq t \leq T} \|\Phi (t)\|_{H^3}.  
\]
We will see that $\Phi = ( \vp, \psi)(t,x)$ satisfies the bound $N (T) \ll1$
by assuming the smallness of the initial data $( \vp, \psi)(0,x)$.

Owing to (\ref{nse}),
the perturbation $(\vp,\psi)$ satisfies the system of
equations
\begin{subequations}
\label{nowe-eq-pv}
\begin{gather}
\vp_t + u \cdot \nabla \vp + \rho \div \psi
= f,
\label{nowe-eq-pv1}
\\
\rho \{ \psi_t + (u \cdot \nabla) \psi \}
- L \psi + p'(\rho) \nabla \vp
= g .
\label{nowe-eq-pv2}
\end{gather}
The boundary and initial conditions for $(\vp,\psi)$ 
follow from (\ref{ice}), (\ref{bce}), and (\ref{stbc}) as
\begin{gather}
\psi(t,M(x'),x') = 0,
\label{nowe-pbc}
\\
(\vp,\psi)(0,x)
=
 (\vp_0, \psi_0)(x)
:= (\rho_0, u_0)(x) - (\rho^s, u^s)(x).
\label{nowe-pic}
\end{gather}
\end{subequations}
Here $L \psi$, $f$, and $g$ are defined by
\begin{equation*} 
\begin{split}
L \psi &:= \mu_1 \Delta \psi + (\mu_1 + \mu_2) \nabla \div \psi,
\\
f &:=
-\vp \div u^s  - \nabla \rho^s \cdot \psi ,
\\
g &:=
- \rho \psi \cdot \nabla u^s - \vp u^s \cdot \nabla u^s - (p'(\rho)- p'(\rho^s) ) \nabla \rho^s . \\
\end{split}
\end{equation*}


In order to prove Theorem \ref{th3}, it suffices to show Proposition \ref{nowe-apriori1} below.
Indeed, the global solvability follows from the continuation argument used in \cite{m-n83}.
Furthermore, the decay property also can be obtained 
in much the same way as in Section 5 of \cite{kg06}.

\begin{proposition}\label{nowe-apriori1}
Let \eqref{super1} and \eqref{cd-st} hold.
Suppose that $\Phi \in X_3 (0,T)$
be a solution to initial--boundary value problem \eqref{nowe-eq-pv}
for some positive constant $T$.
Then there exists a positive constant 
$\ep_0=\ep_0(\Omega)$ depending on $\Omega$ such that if $N (t) + \dels \le \ep_0$, 
the following estimate holds:
\begin{equation}\label{nowe-apes1}
E_{3,0}(t)+\int_0^t D_{3,0}(\tau)\,d\tau
\leq C_0\hs{3}{\Phi_0}^2
\quad \text{for $t \in [0,T]$},
\end{equation}
where $C_0(\Omega)$ is a positive constant 
depending on $\Omega$ but independent of $t$.
\end{proposition}

We remark that the essential difference between the problems \eqref{eq-pv} and \eqref{nowe-eq-pv} 
is whether the inhomogeneous terms $F$ and $G$ appear.
Therefore the proof of Proposition \ref{nowe-apriori1} is very similar 
to that of Proposition \ref{apriori1}.
In the remainder of this section, we sketch the proof of Proposition \ref{nowe-apriori1},
which is given by making use of $ \rho  =\rt +  \varphi^s  +  \varphi$, $u  = \ut + \psi^s  +  \psi  + U $, 
and Theorem \ref{th4}.

\subsection{$L^2$ estimate} \label{ss-nowe-L2}

The first step is to derive the 
estimate of the perturbation $(\vp,\psi)$ in $L^2 $.
To do this, we introduce an energy form $\cale$ as in Subsection \ref{ss-L2}:
\[
\cale :=
K (\rho^s)^{\gamma-1}
\omega \Bigl( \frac{\rho^s}{\rho} \Bigr)
+ \frac{1}{2} |\psi|^2,
\]
which is equivalent to the 
square of the perturbation $(\vp,\psi)$:
\begin{equation}
c (\vp^2 + |\psi|^2)
\le
\cale
\le
C  (\vp^2 + |\psi|^2).
\label{nowe-sqr}
\end{equation}
Moreover 
we have the uniform bounds of solutions as follows:
\begin{equation}
0 < c  \le \rho(t,x) \le C  ,
\quad
|u(t,x)| \le C.
\label{nowe-bdd}
\end{equation}
Here we have used $  N (T) \ll 1$. We obtain the energy inequality in $L^2$ framework, as stated in the following lemma analogous to Lemma \ref{lm1}:

\begin{lemma}
\label{nowe-lm1}
Under the same conditions as in Proposition \ref{nowe-apriori1}, 
it holds that
\begin{equation}
  \|\Phi(t)\|^2
+ \int_0^t  D_{0,0}(\tau) \, d \tau
\lesssim  \|\Phi_0\|^2
+  \dels \int_0^t   \lt{\nabla \vp(\tau)}^2 \, d \tau
\label{nowe-ea0}
\end{equation}
for $t \in [0,T]$.
\end{lemma}

\begin{proof}
By a computation similar to the derivation of \eqref{ea1}, we see that
the energy form $\cale$ satisfies
\begin{equation}
(\rho \cale)_t
- \div (G_1 + B_1)
+ \mu_1 |\nabla \psi|^2
+ (\mu_1 + \mu_2) (\div \psi)^2
=
R_{1},
\label{nowe-ea1}
\end{equation}
where 
\begin{align*}
G_1
&:=
- \rho u \cale
- (p(\rho) - p(\rho^s)) \psi,
\nonumber
\\
B_1
&:=
\mu_1 \nabla \psi \cdot \psi
+ (\mu_1 + \mu_2) \psi \div \psi,
\nonumber
\\
R_{1}
&:=
-\rho (\psi \cdot \nabla) u^s \cdot \psi - ( p(\rho) - p(\rho^s) -p'(\rho^s) \vp) \div u^s - \frac{\vp}{\rho^s} L u^s \cdot \psi .
\nonumber
\end{align*}
We integrate \eqref{nowe-ea1} over $\Omega$. The second term on the left hand side is estimated from below by
using the divergence theorem, (\ref{nowe-sqr}), (\ref{nowe-bdd}), the boundary conditions \eqref{bce} and \eqref{nowe-pbc}
as
\begin{equation}
- \int_{\Omega}  \div \bigl\{ G_1 + B_1 \bigr\} \, dx  
=  \int_{\pd \Omega}
( \rho  \cale)(u_b\cdot n)  \, d \sigma \,
  \gtrsim
 \|\vp|_{\partial \Omega} \|^2_{L^2_{x'}}. 
\label{nowe-ea3}
\end{equation}
Using $ u^s =  (\ut +U ) + \psi^s$, we decompose $R_{1}$ into two parts $R_{11}$ and $R_{12}$ as
\begin{equation*}
\begin{split}
& R_{11} :=
-\rho (\psi \cdot \nabla) (\ut +U) \cdot \psi - ( p(\rho) - p(\rho^s) -p'(\rho^s) \vp) \div (\ut +U) - \frac{\vp}{\rho^s} L (\ut +U) \cdot \psi , \\
& R_{12} :=
-\rho (\psi \cdot \nabla) \psi^s \cdot \psi - ( p(\rho) - p(\rho^s) -p'(\rho^s) \vp) \div \psi^s - \frac{\vp}{\rho^s} L \psi^s \cdot \psi . \\
\end{split}
\end{equation*}
The integral of $R_{11}$ is estimated in a very similar way as the one for the integral of $R_{11}$ in Lemma \ref{lm1}:
\begin{equation} \label{nowe-ea4}
\left|\int_\Omega R_{11} dx \right| \lesssim ( N(T)+\delta) ( \| \nabla \Phi \|^2 +  \|\vp|_{\partial \Omega} \|^2_{L^2_{x'}}) . 
\end{equation}
For the first term in $R_{12}$, we use integration by parts with the boundary condition \eqref{nowe-pbc} as
\begin{equation*}
\begin{split}
\left| \int_\Omega  \rho (\psi \cdot \nabla) \psi^s \cdot \psi  dx \right| 
& =  \Big|   - \int_\Omega (\nabla \rho \cdot \psi ) \psi^s \cdot \psi dx - \int_\Omega (\rho \div \psi ) \psi^s \cdot \psi  dx - \int_\Omega  \rho \psi^s \cdot (\nabla \psi \cdot  \psi)  dx   \Big|   \\
& \lesssim   \delta\| e^{-\alpha x_1/2}\psi \|^2 + \| \nabla \Phi \| \|\psi^s \|_{L^3} \| \psi \|_{L^6} \\
& \lesssim  ( N(T)+\delta)  \| \nabla \Phi \|^2,
\end{split}
\end{equation*}
where we have used \eqref{stdc1}, Theorem \ref{th4}, Hardy's inequality \eqref{hardy}, and Sobolev's inequality \eqref{sobolev1} in deriving the above inequalities.
Similarly one can estimate the other terms in $R_{12}$ and thus have
\begin{equation} \label{nowe-ea5}
\left|\int_\Omega R_{12} dx \right| \lesssim ( N(T)+\delta) ( \| \nabla \Phi \|^2 + \|\vp|_{\partial \Omega} \|^2_{L^2_{x'}}) . 
\end{equation}
We integrate (\ref{nowe-ea1}) over $(0,t) \times \Omega$ and substitute the estimates (\ref{nowe-ea3})--(\ref{nowe-ea5}) into the resultant equality. Then we let $\ep$ and $N (T)+\dels$ be suitably small.
Furthermore, using \eqref{stdc1}, \eqref{ExBdry0}, \eqref{nowe-eq-pv1}, Theorem \ref{th4}, \eqref{hardy}, and \eqref{sobolev1}, we obtain
\begin{equation*}
\begin{split}
\left\| \frac{d}{dt} \vp  \right\|^2 
 = \| \rho \div \psi + \vp \div u^s + \nabla \rho^s \cdot \psi \|^2  
\lesssim  \|\nabla \psi\|^2 + \delta\|\nabla \vp\|^2 + \delta \|\vp|_{\partial \Omega} \|^2_{L^2_{x'}}  .
\end{split}
\end{equation*}
These computations yield the desired inequality.
\end{proof}

\subsection{Time-derivative estimates} \label{ss-nowe-time-deriv}

In this section we derive time-derivative estimates.
To this end, by applying the differential operator $\partial_t^k$ for $k=0,1$ to 
(\ref{nowe-eq-pv1}) and (\ref{nowe-eq-pv2}), we have two equations:
\begin{gather}
\partial_t^k \vp_t
+ u \cdot \nabla \partial_t^k \vp
+ \rho \div \partial_t^k \psi
= f_{0,k},
\label{nowe-ec1}
\\
\rho \{
\partial_t^k \psi_t
+ (u \cdot \nabla) \partial_t^k \psi
\}
- L(\partial_t^k \psi)
+ p'(\rho) \nabla \partial_t^k \vp
= g_{0,k},
\label{nowe-ec2}
\end{gather}
where
\begin{align*}
f_{0,k}
&:=
\partial_t^k f
- [\partial_t^k, u] \nabla \vp
- [\partial_t^k, \rho] \div \psi,
\nonumber
\\
g_{0,k}
&:= \partial_t^k g
- [\partial_t^k, \rho] \psi_t
- [\partial_t^k, \rho u] \nabla \psi
- [\partial_t^k, p'(\rho)] \nabla \vp.
\nonumber
\end{align*}
Here $[T,u]v := T(uv) - u T v$ is a commutator. 

Firstly, the following lemma, which is parallel to Lemma \ref{lm2}, provides an estimate of $\pd_t \Phi$. 

\begin{lemma}
\label{nowe-lm2}
Under the same conditions as in Proposition \ref{nowe-apriori1}, 
it holds that
\begin{equation} \label{nowe-ec0}
 \|\pd_t \Phi(t)\|^2
+ \int_0^t  
\|\pd_t \nabla\psi(\tau)\|^2 \, d \tau
\lesssim  \|\Phi_0\|_{H^3}^2
+ (  N (T)+\dels) \int_0^t  D_{3,0}(\tau) \, d \tau   
\end{equation}
for $t \in [0,T]$.
\end{lemma}
\begin{proof}
The proof is very similar to Lemma \ref{lm2}. Indeed, by a similar process as the derivation for \eqref{ec5}, we know from \eqref{nowe-ec1} and \eqref{nowe-ec2} that
\begin{multline*}
\Bigl(
\frac{1}{2} P(\rho) | \partial_t \vp |^2
+ \frac{1}{2} \rho |\partial_t \psi|^2
\Bigr)_t
+ \div \Bigl(
\frac{1}{2} P(\rho) u
|\partial_t \vp|^2
+ B_2
\Bigr)
\mspace{150mu}
\nonumber
\\
\mspace{200mu}
{}+ \mu_1 |\nabla (\partial_t \psi)|^2
+ (\mu_1 + \mu_2) |\div (\partial_t \psi)|^2
=
R_2,
\end{multline*}
where
\begin{align*}
B_2&:=
\frac{1}{2} \rho u |\partial_t \psi|^2
- \mu_1 \nabla (\partial_t \psi)  \cdot \partial_t \psi
- (\mu_1+\mu_2) \div (\partial_t \psi) \partial_t \psi 
+ p'(\rho) \partial_t \vp \, \partial_t \psi, \nonumber \\
R_2&:=
P(\rho) f_{0,k} \, \partial_t \vp
+
\frac{3-\gamma}{2} P(\rho) \div u \, |\partial_t \vp|^2
+ (g_{0,k} + p''(\rho) \, \partial_t \vp \nabla \rho)
\cdot \partial_t \psi. 
\end{align*}
Then carrying out estimates as in the proof of Lemma \ref{lm2} with $\zeta=0$ gives \eqref{nowe-ec0}. We omit the details here.
\end{proof}

We also estimate $\pd_t^k \nabla\psi$ for $k=0,1$. The result is the following lemma parallel to Lemma \ref{lm3}.

\begin{lemma}
\label{nowe-lm3}
Under the same conditions as in Proposition \ref{nowe-apriori1}, 
it holds that
\begin{multline}
 \lt{\pd_t^k\nabla\psi(t)}^2
+ \int_0^t  \lt{\pd_t^{k+1}\psi(\tau)}^2 \, d \tau 
\\
\lesssim 
 \hs{3}{\Phi_0}^2
+ \lambda \calh_{k} (t)
+ \lambda^{-1} \calp_{k} (t)
+ ( N (t)+\dels ) \int_0^t  D_{3,0}(\tau)^2 \, d \tau
\label{nowe-ed0}
\end{multline}
for $t \in [0,T]$ and $k=0,1$.
Here, $\calh_{k}^\zeta(t)$ and $\calp_{k}^\zeta(t)$ are defined by
\begin{align*}
\calh_{k} (t)
 :=
  \lt{\pd_t^k \nabla \vp(t)}^2
+ \int_0^t   \lt{\pd_t^k \nabla \vp(\tau)}^2 \, d \tau,
\quad
\calp_{k} (t)
 :=
 \lt{\pd_t^k \psi(t)}^2
+ \int_0^t   \lt{\pd_t^k\nabla \psi(\tau)}^2 \, d \tau.
\end{align*}
\end{lemma}
\begin{proof}
By a similar computation as in the derivation of \eqref{ed1}, 
we see from \eqref{nowe-ec1} and \eqref{nowe-ec2} that
\begin{equation*}
\dt E_{3}
- \div B_3
+ \rho |\partial_t^k \psi_t|^2
= G_3 + R_3,
\label{nowe-ed1}
\end{equation*}
where $E_3$, $B_3$, $G_3$ and $R_3$ are defined by
\begin{align*}
E_3 := {}
&
\frac{\mu_1}{2} |\nabla \partial_t^k \psi|^2
+ \frac{\mu_1 + \mu_2}{2} |\div \partial_t^k \psi|^2
+ p'(\rho) \nabla \partial_t^k \vp \cdot \partial_t^k \psi,
\\
B_3 := {}
&
\mu_1 \nabla \partial_t^k \psi \cdot \partial_t^k \psi_t
+ (\mu_1 + \mu_2) \partial_t^k \psi_t \div \partial_t^k \psi
+ p'(\rho) \partial_t^k \vp_t \, \partial_t^k \psi,
\\
G_3 := {}
&
p'(\rho) \div \partial_t^k \psi \, (u \cdot \nabla \partial_t^k \vp 
  + \rho \div \partial_t^k \psi)
-\rho (u \cdot \nabla) \partial_t^k \psi \cdot \partial_t^k \psi_t,
\\
R_3 := {}
&
p''(\rho) \vp_t \nabla \partial_t^k \vp \cdot \partial_t^k \psi
+ p''(\rho) \nabla \rho \cdot \partial_t^k \psi
     (u \cdot \nabla \partial_t^k \vp  + \rho \div \partial_t^k \psi)
\\
&
- f_{0,k} \div(p'(\rho) \partial_t^k \psi)
+    g_{0,k}  \cdot \partial_t^k \psi_t.
\end{align*}
Estimate the terms as in the proof of Lemma \ref{lm3},
we arrive at \eqref{nowe-ed0}. 
\end{proof}

\subsection{Spatial-derivative estimates}  \label{ss-nowe-spatial-deriv}

We do the same change of variables as in Subsection \ref{Sptial-deriv}
and also define
\begin{equation*}
\hat{\vp} (t,y) := \vp(t,\Gamma(y)), \quad \hat{\psi} (t,y) : = \psi (t,\Gamma (y)), \quad \hat{\rho} (t,y) := \rho(t,\Gamma(y)), \quad \hat{u} (t,y) := u(t,\Gamma(y))
\end{equation*}
for $y \in \mathbb{R}_3^+ : = \{ (y_1, y_2, y_3) \in \mathbb{R}^3 : y_1 >0 \}$,
where $\Gamma$ is defined in \eqref{CV1} and $\Gamma(\mathbb R^3_+)=\Omega$. 
Furthermore, $\hat{\nabla}$, $\hat{\div}$, $\hat{\Delta}$, 
$\hat{\frac{d}{dt}}$ denote the same differential operators defined in 
\eqref{CV3}--\eqref{CV6}. 

Now we obtain the equation for $(\hat{\vp}, \hat{\psi})$:
\begin{subequations}
\label{nowe-eq-pv-hat}
\begin{gather}
 \hat{\vp}_t + \hat{u} \cdot \hat{\nabla} \hat{\vp} + \hat{\rho} \hat{\div} \hat{\psi}
= \hat{f},
\label{nowe-eq-pv1-hat}
\\
\hat{\rho} \{ \hat{\psi}_t  + (\hat{u} \cdot \hat{\nabla}) \hat{\psi} \}
- \hat{L} \hat{\psi} + p'(\hat{\rho}) \hat{\nabla} \hat{\vp}
= \hat{g}.
\label{nowe-eq-pv2-hat}
\end{gather}
The initial and boundary conditions for $(\hat{\vp},\hat{\psi})$ 
follow from (\ref{nowe-pic}), (\ref{nowe-pbc}) as
\begin{gather}
(\hat{\vp},\hat{\psi})(0,y)
=
 (\hat{\vp}_0, \hat{\psi}_0)(y)
= (\rho_0, u_0)(\Gamma(y)) - (\rho^s, u^s)(\Gamma(y)),
\label{nowe-pic-hat}
\\
\hat{\psi}(t,0,y') = 0.
\label{nowe-pbc-hat}
\end{gather}
\end{subequations}
Here $\hat{L} \hat{\psi}$, $\hat{f}$, and $\hat{g}$ are defined by
\begin{align*}
\hat{L} \hat{\psi}(t,y) := \mu_1 \hat{\Delta} \hat{\psi}(t,y) + (\mu_1 + \mu_2) \hat{\nabla} \hat{\div} \hat{\psi}(t,y), \quad
\hat{f}(t,y) :=f(t,\Gamma(y)), \quad 
\hat{g}(t,y) := g(t,\Gamma(y)).
\end{align*}

We derive the estimate on the tangential spatial derivatives, 
which is parallel to Lemma \ref{lm2-hat}.

\begin{lemma}
\label{nowe-lm2-hat}
Under the same conditions as in Proposition \ref{nowe-apriori1}, 
it holds,
\begin{align}
{}& \|\nabla^l_{y'} \hat{\Phi}(t)\|_{L^2(\mathbb R^3_+)}^2
+ \int_0^t 
\left(\|\nabla\nabla^l_{y'}  \hat{\psi}(\tau)\|_{L^2(\mathbb R^3_+)}^2
+ \left\| \nabla^l_{y'}  \hat{\frac{d}{dt}} \hat{\vp}  (\tau) \right\|_{L^2(\mathbb R^3_+)}^2 
 \right) \, d \tau 
\notag\\
&\lesssim  \|\Phi_0\|_{H^3}^2
+ \epsilon \int^t_0  (\| \nabla \vp(\tau) \|^2_{H^{l-1}} + \| \nabla \psi(\tau) \|^2_{H^{l}} ) d \tau 
+ \epsilon^{-1} \int^t_0  \| \nabla \psi \|^2_{H^{l-1}} d \tau 
\notag\\
& \quad + (N (T)+\dels ) \int_0^t  D_{3,0}(\tau) \, d \tau
\label{nowe-ec0-hat}
\end{align}
for $t \in [0,T]$, $\epsilon \in (0,1)$,  and $l = 1, 2, 3$.
\end{lemma}
\begin{proof}
By a similar computation as in the derivation of \eqref{ec5-hat}, 
we see from \eqref{nowe-eq-pv1-hat} and \eqref{nowe-eq-pv2-hat} that 
\begin{multline*}
\Bigl(
\frac{1}{2} P(\hat{\rho}) | \nabla^l_{y'} \hat{\vp} |^2
+ \frac{1}{2} \hat{\rho} |\nabla^l_{y'} \hat{\psi}|^2
\Bigr)_t
+ \hat{\div} \Bigl(
\frac{1}{2} P(\hat{\rho}) \hat{u}
|\nabla^l_{y'} \hat{\vp}|^2
+ \hat{B}_2
\Bigr)
\\
{}+ \mu_1 |\hat{\nabla} (\nabla^l_{y'} \hat{\psi})|^2
+ (\mu_1 + \mu_2) |\hat{\div} (\nabla^l_{y'} \hat{\psi})|^2
=
\hat{R}_2,
\end{multline*}
where
\begin{align}
\hat{B}_2
:=& 
\frac{1}{2} \hat{\rho} \hat{u} |\nabla^l_{y'} \hat{\psi}|^2
- \mu_1 \hat{\nabla} (\nabla^l_{y'} \hat{\psi})  \cdot \nabla^l_{y'} \hat{\psi}
- (\mu_1+\mu_2) \hat{\div} (\nabla^l_{y'} \hat{\psi}) \nabla^l_{y'} \hat{\psi} 
+ p'(\hat{\rho}) \nabla^l_{y'} \hat{\vp} \, \nabla^l_{y'} \hat{\psi} ,
\nonumber
\\
\hat{R}_2
:=&
P(\hat{\rho}) \hat{f}_{l,0} \, \nabla^l_{y'} \hat{\vp}
+
\frac{3-\gamma}{2} P(\hat{\rho}) \hat{\div} \hat{u} \, |\nabla^l_{y'} \hat{\vp}|^2
+ (\hat{g}_{l,0} + p''(\hat{\rho}) \, \nabla^l_{y'} \hat{\vp} \hat{\nabla} \hat{\rho})
   \cdot \nabla^l_{y'} \hat{\psi}.
\nonumber
\end{align}
Then by a similar way of estimating the terms as in the proof of Lemma \ref{lm2-hat} with $\zeta=0$, 
we arrive at \eqref{nowe-ec0-hat}. We omit the details.
\end{proof}

We also derive the estimates on the normal spatial derivative,
which is parallel to Lemma \ref{lm4-hat}. 

\begin{lemma}
\label{nowe-lm4-hat}
Suppose that the same conditions as in Proposition \ref{nowe-apriori1} hold. Define the index $\bm{a} = (a_1, a_2, a_3)$ with $a_1, a_2, a_3 \geq 0$ and $|\bm{a}| := a_1 + a_2 + a_3$. Let $\partial^{\bm{a}} := \partial^{a_1}_{y_1} \partial^{a_2}_{y_2} \partial^{a_3}_{y_3}$. Then it holds that 
\begin{align}
&  \| \partial^{\bm{a}} \partial_{y_1} \hat{\vp} (t) \|_{L^2(\mathbb R^3_+)}^2 
+ \int^t_0 \left( \|\partial^{\bm{a}} {\cal D} \hat{\vp} (\tau) \|_{L^2(\mathbb R^3_+)}^2 
+ \left\| \partial^{\bm{a}} \partial_{y_1} \hat{\frac{d}{dt}} \hat{\vp}(\tau) \right\|_{L^2(\mathbb R^3_+)}^2  \right) \, d\tau 
\notag\\
&  \lesssim \|\vp_0\|_{H^3}^2 
+ \int_0^t  \left(
\left\|\partial^{\bm{a}} \nabla_{y'} \hat{\frac{d}{dt}} \hat{\vp}(\tau) \right\|_{L^2(\mathbb R^3_+)}^2
+ \|\partial^{\bm{a}} \nabla_y \nabla_{y'} \hat{\psi}(\tau) \|_{L^2(\mathbb R^3_+)}^2 \right)\, d \tau
\notag\\
&\quad + \int_0^t  \left(|{\bm{a}}|\|\nabla_y {\vp} (\tau)\|^2_{H^{|{\bm{a}}| -1}}
+|{\bm{a}}| \left\|\nabla {\frac{d}{dt}} {\vp}(\tau) \right\|_{H^{|{\bm{a}}|-1}}^2
+\|{\psi}_t (\tau)\|^2_{H^{|{\bm{a}}|}} 
+\|\nabla_y {\psi} (\tau)\|^2_{H^{|{\bm{a}}|}} \right)\, d \tau
\notag\\
& \quad 
+ (N (T)+\dels ) \int_0^t  D_{3, 0}(\tau) \, d \tau  
\label{nowe-ee0-hat}
\end{align}
for $t \in [0, T]$ and $0\leq |\bm{a}| \leq 2$.
\end{lemma}

\begin{proof}
By a similar computation as in the derivation of \eqref{ee7-hat}, we obtain 
\begin{multline*}
 \left(\frac{1}{2} \mu |\partial^{\bm{a}} \mathcal{D} \hat{\vp}|^2 + \frac{1}{2} \tilde{\mathcal{A}}_1 p'(\hat{\rho}) \hat{\rho} |\partial^{\bm{a}} \mathcal{D} \hat{\vp}|^2  \right)_t + \mu \left|\partial^{\bm{a}} \mathcal{D} \hat{\frac{d}{dt}} \hat{\vp}\right|^2 + \tilde{\mathcal{A}}_1 p'(\hat{\rho}) \hat{\rho} |\partial^{\bm{a}} \mathcal{D} \hat{\vp}|^2 
\\
+ \hat{\div}\left(\frac{1}{2} \mu |\partial^{\bm{a}} \mathcal{D} \hat{\vp}|^2 \hat{u} + \frac{1}{2} \tilde{\mathcal{A}}_1 p'(\hat{\rho}) \hat{\rho} |\partial^{\bm{a}} \mathcal{D} \hat{\vp}|^2\hat{u}  \right)
= \hat{R}_3.
\end{multline*}
Here $\tilde{\cal A}_1$ and ${\cal D}$ are defined in \eqref{Aj-def}, and 
\begin{equation*}
\begin{split}
\hat{R}_3&:=\big\{[\hat{u} \cdot \hat{\nabla} \partial^{\bm{a}} \mathcal{D} \hat{\vp} ] - \partial^{\bm{a}} \mathcal{D}   [\hat{u} \cdot \hat{\nabla} \hat{\vp}] \big\} (\mu \partial^{\bm{a}} \mathcal{D} \hat{\vp} + \tilde{\mathcal{A}}_1 p'(\hat{\rho}) \hat{\rho}\partial^{\bm{a}} \mathcal{D} \hat{\vp})
\\
& \quad + \frac{1}{2}\mu |\partial^{\bm{a}} \mathcal{D} \hat{\vp}|^2 \hat{\div} \hat{u}  
+ \frac{1}{2} \tilde{\mathcal{A}}_1 (p'(\hat{\rho}) \hat{\rho} )_t|\partial^{\bm{a}} \mathcal{D} \hat{\vp}|^2 
+ \frac{1}{2} |\partial^{\bm{a}} \mathcal{D} \hat{\vp}|^2 \hat{\div} \left(\tilde{\mathcal{A}}_1 p'(\hat{\rho}) \hat{\rho} \hat{u}\right) \\
& \quad + \hat{I}_1 \left(\partial^{\bm{a}} \mathcal{D} \hat{\vp} + \partial^{\bm{a}} \mathcal{D}  \hat{\frac{d}{dt}} \hat{\vp} \right),
\end{split}
\end{equation*}
where $\hat{I}_1$ is defined for the coefficients ${\cal A}_j$ in \eqref{Aj-def} and $a_{\bm{b}}$ in \eqref{ee4-hat} as
\begin{equation*}
\begin{split}
 \hat{I}_1 := 
& \Big\{ -\partial^{\bm{a}} [\tilde{\mathcal{A}}_1 p'(\hat{\rho}) \hat{\rho}  \mathcal{D} \hat{\vp}]
+ \tilde{\mathcal{A}}_1 p'(\hat{\rho}) \hat{\rho} \partial^{\bm{a}} \mathcal{D} \hat{\vp} 
\\
& + \partial^{\bm{a}} \Big\{\mu(\mathcal{A}_2\partial_{y_2}+\mathcal{A}_3\partial_{y_3}) \hat{\frac{d}{dt}} \hat{\vp} - \mu \tilde{\mathcal{A}}_1 \partial_{y_1} \hat{\rho} \hat{\div} \hat{\psi} - \tilde{\mathcal{A}}_1 \hat{\rho}^2 ( \sum_{j=1}^3 \mathcal{A}_j \hat{\psi}_{jt} + \sum_{j=1}^3 \mathcal{A}_j \hat{u} \cdot \hat{\nabla} \hat{\psi}_j ) 
\\
& \qquad \qquad  + \tilde{\mathcal{A}}_1 \hat{\rho}\sum_{1 \leq |\bm{b}|\leq 2, \ b_1\neq2, \ j=1,2,3} a_{\bm{b}} \partial^{\bm{b}}_y\hat{\psi}_j + \mu \tilde{\mathcal{A}}_1 \partial_1 \hat{f} + \tilde{\mathcal{A}}_1 \hat{\rho} \sum_{j=1}^3 \mathcal{A}_j \hat{g}_j \Big\} . \\
\end{split}
\end{equation*}
Then using similar methods as in the proof of Lemma \ref{lm4-hat}, we arrive at \eqref{nowe-ee0-hat}.
\end{proof}

\subsection{Cattabriga estimates}  \label{ss-nowe-CattabrigaEst}

As in Subsection \ref{ss-CattabrigaEst}, we control the good contribution of spatial-derivatives 
using the Cattabriga estimate in Lemma \ref{CattabrigaEst}, which has crucial dependence on $\Omega$. Firstly we have the following lemma, which is parallel to Lemma \ref{lm1-C}.

\begin{lemma} \label{nowe-lm1-C}
Under the same assumption as in Proposition \ref{nowe-apriori1},
it holds, for $k=0, 1, 2$, 
\begin{equation}
\| \nabla^{k+2} \psi \|^2 + \| \nabla^{k+1} \vp \|^2  \lesssim_\Omega  \| \psi_t \|^2_{H^{k}} 
+ \| \nabla \psi \|^2_{H^{k}} 
+ \left\| \frac{d}{dt} \vp \right\|^2_{H^{k+1}} 
 + (N (T) + \delta ) D_{3, 0} . 
\label{nowe-ef0-C}
\end{equation}
\end{lemma}
\begin{proof}
The proof is almost the same as in that of Lemma \ref{lm1-C}.
From \eqref{nowe-eq-pv}, we have a boundary value problem of the Stokes equation:
\begin{equation}\label{nowe-ef7-C}
 \rho_+\div\psi = V , \quad
 - \mu_1 \Delta \psi + p'(\rho_+)\nabla \vp = W, 
\quad
 \psi|_{\partial \Omega} =0 , \quad \lim_{|x| \rightarrow \infty} |\psi| =0,
\end{equation}
where
\begin{align*}
V : =& f - \frac{d}{dt}\vp - (\rho-\rho_+)\div \psi , \\
W :
=& - \rho \{ \psi_t + ( u \cdot \nabla ) \psi \} + (\mu_1 + \mu_2) \rho_+^{-1} \nabla V + g  - (p'(\rho)-p'(\rho_+)) \nabla \vp.
\end{align*}
The application of the Cattabriga estimate \eqref{Cattabriga} 
to this boundary value problem gives \eqref{nowe-ef0-C}.
\end{proof}

As in Subsection \ref{ss-CattabrigaEst}, 
we show similar estimates for $(\hat{\vp},\hat{\psi})(t,y)$, where $y \in \mathbb R^3_+$.
We use the same differential operator $\check{\partial}_{y_j}$ defined in \eqref{checkOp}. 
Furthermore, $\check{\nabla}^l_{y'}$ means the totality of all $l$-times tangential derivatives 
$\check{\partial}_{y_j}$ only for $j=2,3$.
Applying $\check{\nabla}^l_{y'}$ to \eqref{nowe-ef7-C}, 
we obtain a boundary value problem of the Stokes equation: 
\begin{equation}\label{nowe-checkEq}
\rho_+\div \check{\nabla}^l_{y'} \psi = \check{V}, \ \
- \mu_1 \Delta \check{\nabla}^l_{y'} \psi + p'(\rho_+)\nabla \check{\nabla}^l_{y'} \vp = \check{W}, \ \  
\check{\nabla}^l_{y'} \psi|_{\partial \Omega} =0 , \ \ \lim_{|x| \rightarrow \infty} | \check{\nabla}^l_{y'} \psi| =0,
\end{equation}
where
\begin{align*}
\check{V} : = 
& \check{\nabla}^l_{y'} f - \check{\nabla}^l_{y'} \frac{d}{dt}\vp - \check{\nabla}^l_{y'} [(\rho-\rho_+)\div \psi ] + [\rho_+ \div \check{\nabla}^l_{y'} \psi -  \rho_+ \check{\nabla}^l_{y'} \div \psi] , \\
\check{W} : = 
&- \check{\nabla}^l_{y'} [ \rho \{ \psi_t + ( u \cdot \nabla ) \psi \} ] + (\mu_1 + \mu_2) \check{\nabla}^l_{y'}  \nabla \div \psi + \check{\nabla}^l_{y'} g \\
& - \check{\nabla}^l_{y'} [ (p'(\rho)-p'(\rho_+)) \nabla \vp ]  + \mu_1 (\check{\nabla}^l_{y'} \Delta \psi - \Delta \check{\nabla}^l_{y'} \psi) 
+ p'(\rho_+) ( \nabla \check{\nabla}^l_{y'} \vp - \check{\nabla}^l_{y'} \nabla \vp ).  
\end{align*}
Then we have the following estimate parallel to Lemma \ref{lm2-C}.    
    
\begin{lemma} \label{nowe-lm2-C}
Under the same assumption as in Proposition \ref{nowe-apriori1},
it holds, for $k=0, 1$, $k+l =  1,  2$, 
\begin{align}
&\| \nabla^{k+2}  \nabla^l_{y'} \hat{\psi} \|_{L^2(\mathbb R^3_+)}^2 + \|\nabla^{k+1} \nabla^l_{y'} \hat{\vp} \|_{L^2(\mathbb R^3_+)}^2 
\notag\\
& \lesssim_\Omega 
\left\|{\nabla}^l_{y'} \hat{\frac{d}{dt}} \hat{\vp} \right\|^2_{H^{k+1} (\mathbb R^3_+)} 
+ \| \psi_t \|^2_{H^{k+l}} 
+ \| \nabla \psi \|^2_{H^{k+l}} 
+ \| \nabla \vp \|^2_{H^{k+l-1}}
 + (N (T) + \delta ) D_{3,0} . 
\label{nowe-eg0-C}
\end{align}
\end{lemma}
\begin{proof}

Applying the Cattabriga estimate \eqref{Cattabriga} 
to problem \eqref{nowe-checkEq}, one can have \eqref{nowe-eg0-C}.
The proof is essentially the same as the one for Lemma \ref{lm2-C} so we omit it.
\end{proof}

Using Lemmas \ref{nowe-lm2-hat}--\ref{nowe-lm2-C},
we complete the derivation of the dissipative terms for the spatial derivatives.

\begin{lemma} \label{nowe-lm3-comp}
Under the same assumption as in Proposition \ref{nowe-apriori1},
it holds, for $p=0,  1,  2$, 
\begin{align}
&   \| \nabla^{p+1}  \vp (t) \|^2 + \int^t_0    \|( \nabla^{p+1} \vp, \nabla^{p+2} \psi )(\tau)\|^2   \, d \tau  + \int^t_0  \left\| \nabla^{p+1} \frac{d}{dt} \vp (\tau) \right\|^2 \, d \tau   
\notag\\
& \lesssim_\Omega \| \Phi_0\|^2_{H^3} +   \| \vp (t) \|_{H^p}^2 + \int^t_0  \|\psi_t(\tau) \|^2_{H^p} \, d \tau  
+ \int^t_0   D_{p, 0} (\tau) \, d \tau   
\notag\\
& \qquad + (N (T) + \delta ) \int^t_0  D_{3, 0} (\tau) \, d \tau . 
\label{nowe-em0-comp}
\end{align}
\end{lemma}
\begin{proof}
It is straightforward to show this 
by following the proofs of Lemmas \ref{lm1-comp} and \ref{lm3-comp}
with the aid of Lemmas \ref{nowe-lm2-hat}--\ref{nowe-lm2-C}.
\end{proof}

\subsection{Elliptic estimates} \label{ss-nowe-EllipticEst}
Similarly as in Subsection \ref{EllipticEst}, we use the elliptic estimate (Lemma \ref{ellipticEst}) to rewrite some terms for the time-derivatives into terms for the spatial-derivatives.
\begin{lemma} \label{nowe-lm1-E}
Under the same assumption as in Proposition \ref{nowe-apriori1},
it holds that  
\begin{gather}
\| \nabla^{k+2}   \psi \|
  \lesssim  \| \psi_t \|_{H^{k}} + E_{k+1,0}, \quad k=0,1,
\label{nowe-eh0-E}
\\
\| \nabla^{2}   \psi_t \|
  \lesssim  \|\psi_{tt} \| +D_{2,0}.
\label{nowe-ej0-E}
\end{gather}
\end{lemma} 
\begin{proof}
The proof is essentially the same as the one for Lemma \ref{lm1-E}. 
From \eqref{nowe-eq-pv}, we have the elliptic boundary value problem:
\begin{equation} \label{nowe-eh1-E}
-\mu_1 \Delta \psi -(\mu_1 + \mu_2) \nabla \div \psi = - \rho_+ \psi_t - p'(\rho_+)\nabla \vp + H, \quad 
\psi|_{\partial\Omega}=0, \quad \lim_{|x|\to\infty}\psi=0,
\end{equation}
where
\begin{equation*}
H := -(\rho - \rho_+) \psi_t - (p'(\rho) -p'(\rho_+)) \nabla \vp -\rho (u \cdot \nabla) \psi +g. 
\end{equation*}
Applying Lemma \ref{ellipticEst} to this problem, we conclude \eqref{nowe-eh0-E}. 
We next apply $\partial_t$ to \eqref{nowe-eh1-E} to obtain
\begin{equation*}
-\mu_1 \Delta \psi_t - (\mu_1 + \mu_2) \nabla \div \psi_t 
= - \rho_+\psi_{tt} -p'(\rho_+)\nabla \vp_t + H_t, \quad 
\psi_t |_{\partial \Omega} =0, \quad \lim_{|x|\rightarrow \infty} \psi_t =0. 
\end{equation*}
Then applying Lemma \ref{ellipticEst} again to this boundary value problem gives \eqref{nowe-ej0-E}.
\end{proof}

\subsection{Completion of the a priori estimate} 

Now we can complete the a priori estimate.

\begin{proof}[Proof of Proposition \ref{nowe-apriori1}]
The a priori estimate can be shown in the same way as in the proof of Proposition \ref{apriori1}
in Subsection \ref{ss-comp-apriori}.
Indeed the proof is complete just by 
replacing \eqref{ea0}, \eqref{ec0}, \eqref{ed0} and \eqref{em0-comp}--\eqref{ej0-E}
by \eqref{nowe-ea0}, \eqref{nowe-ec0}, \eqref{nowe-ed0} and \eqref{nowe-em0-comp}--\eqref{nowe-ej0-E}, respectively.
\end{proof}

We discuss briefly the proof of Corollary \ref{cor}.

\begin{proof}[Proof of Corollary \ref{cor}]
If $\|M\|_{H^{9}(\mathbb R^2)} \leq \kappa$ holds for $\kappa$ being in Lemma \ref{CattabrigaEst},
we can replace $\lesssim_{\Omega}$ by $\lesssim$ 
in inequalities \eqref{nowe-ef0-C}, \eqref{nowe-eg0-C}, and \eqref{nowe-em0-comp}.
Then following the proof of Proposition \ref{nowe-apriori1} with these improved inequalities, 
we conclude Corollary \ref{cor}.
\end{proof}

\begin{appendix}
\section{General inequalities}\label{Appenx1}

We discuss some basic inequalities and estimates that are frequently used throughout the paper. 
The following lemmas cover the case $M\equiv 0$, that is, $\Omega=\mathbb R^3_+$.

\begin{lemma}\textbf{(Hardy's Inequality)}
Let $\alpha>0$. For $f \in H^1(\Omega)$, it holds that
\begin{equation}
\int_{\Omega} e^{-\alpha x_1}|f(t)|^2 \, d x
\lesssim
\lt{\nabla f(t)}^2
+ \|{f(t,M(\cdot),\cdot)}\|_{L^2(\mathbb R^2)}^2.
\label{hardy}
\end{equation}
\end{lemma}
\begin{proof}
This can be proved in the same way as in \cite{knz03}.
\end{proof}

\begin{lemma}\textbf{(Sobolev's Inequalities)}
For $f \in H^1(\Omega)$ and $g \in H^2(\Omega)$, it holds that
\begin{align}
\|f\|_{L^p(\Omega)} &\lesssim \|f\|_{H^1(\Omega)}, \quad 2\leq p <6,
\label{sobolev0}\\
\|f\|_{L^6(\Omega)} &\lesssim \|\nabla f\|_{L^2(\Omega)},
\label{sobolev1}\\
\|g\|_{L^\infty(\Omega)} &\lesssim \|g\|_{H^2(\Omega)}.
\label{sobolev2}
\end{align}
\end{lemma}
\begin{proof}
It is straightforward to show \eqref{sobolev0} and \eqref{sobolev2}.
We show only \eqref{sobolev1}.
Let us introduce a standard extension operator $E$ 
from $H^1(\Omega)$ to $H^1(\mathbb R^3)$ with
\begin{gather*}
Eh(x)=h(x) \quad \text{for $x\in \Omega$},
\quad
\|Eh\|_{H^1(\mathbb R^3)} \lesssim \|h\|_{H^1(\Omega)},
\quad
\|\nabla Eh\|_{L^{2}(\mathbb R^3)} \lesssim \|\nabla h\|_{L^{2}(\Omega)}.
\end{gather*}
Furthermore, we know that for $\tilde{f} \in H^{1}(\mathbb R^3)$,
\[
\|\tilde{f}\|_{L^6(\mathbb{R}^3)} \lesssim \|\nabla \tilde{f}\|_{L^2(\mathbb{R}^3)}. 
\]
Then putting $\tilde{f}=Ef$ gives
\[
\|Ef\|_{L^6(\mathbb{R}^3)} \lesssim \|\nabla Ef\|_{L^2(\mathbb{R}^3)} \lesssim \|\nabla f\|_{L^2(\Omega)}, 
\]
where we have used the properties of the extension operator in deriving the last inequality. 
This together with $\|f\|_{L^6 (\Omega)} \leq \|Ef\|_{L^6(\mathbb{R}^3)}$ gives \eqref{sobolev1}. 
\end{proof}

\begin{lemma}
\textbf{(Gagliardo-Nirenberg Inequality)}
Let $k=2,3,4 \cdots$.
For $f \in H^{k} (\Omega)$, there holds that
\begin{equation} \label{GN1}
\|f\|_{H^{k-1}(\Omega)} \lesssim \|f\|_{H^{k}(\Omega)}^{1-1/k}\|f\|_{L^2(\Omega)}^{1/k}.
\end{equation}
\end{lemma}
\begin{proof}
This can be shown in much the same way as in Sobolev's inequality.
\end{proof}

\begin{lemma} \label{CommEst}
\textbf{(Commutator Estimate)}
Let $k=0,1,2,\cdots$.
For $f, g, \nabla f \in H^{k} (\Omega) \cap L^\infty (\Omega)$, we have
\begin{equation} \label{CommEst-1}
\|[\nabla^{k+1},f]g\|_{L^2(\Omega)}\lesssim  \|\nabla f\|_{L^\infty(\Omega)}\|g\|_{H^{k}(\Omega)}+\|\nabla f\|_{H^{k}(\Omega)}\|g\|_{L^\infty(\Omega)}
\end{equation}
\begin{equation} \label{CommEst-2}
\|[\nabla^{k+1},\nabla M]g\|_{L^2(\Omega)} \lesssim \|g\|_{H^{k}(\Omega)}.
\end{equation} 
\end{lemma}
\begin{proof}
Lemma 4.9 in \cite{Rac} claims that for $\tilde{f}, \tilde{g}, \nabla \tilde{f} \in H^{k}(\mathbb R^3) \cap L^\infty (\mathbb R^3)$,
\[
\|[\nabla^{k+1},\tilde{f}]\tilde{g} \|_{L^2(\mathbb R^3)} \lesssim 
\|\nabla \tilde{f}\|_{L^\infty(\mathbb R^3)}\|\tilde{g}\|_{H^{k}(\mathbb R^3)}+\|\nabla \tilde{f}\|_{H^{k}(\mathbb R^3)}\|\tilde{g}\|_{L^\infty(\mathbb R^3)}.
\]
Then one can show \eqref{CommEst-1} similarly to the proof of Sobolev's inequality.
Furthermore, \eqref{CommEst-2} can be shown by direct expansion and Sobolev's inequality.
\end{proof}

\begin{lemma} \label{CattabrigaEst}
\textbf{(Cattabriga Estimate)}
Consider the following Stokes system 
\begin{equation*} 
\bar{\rho} \div u = h , \quad
- \hat{\mu} \Delta u + \hat{p} \nabla p = g , \quad
u|_{\partial \Omega} =0, \quad
\ \lim_{|x| \rightarrow \infty}u =0
\end{equation*}
with $\bar{\rho}$, $\hat{\mu}$, $\hat{p}$ being constants. For $k= 0, 1, \cdots, 4$ and 
$(h,g)\in H^{k+1}(\Omega)\times H^{k}(\Omega)$, 
if $(u,p) \in H^{k+2}(\Omega)\times H^{k+1}(\Omega)$ is a solution to the Stokes system, 
then it holds that
\begin{equation}\label{Cattabriga}
\| \nabla^{k+2} u \|_{L^2(\Omega)}^2 + \|\nabla^{k+1} p \|_{L^2(\Omega)}^2 
\leq C_0 (\|h \|^2_{H^{k+1}(\Omega)} + \| g \|^2_{H^k(\Omega)} + \|\nabla u \|_{L^2(\Omega)}^2),
\end{equation}
where $C_0=C_0(\Omega)$ is a positive constant depending on $\Omega$.
Furthermore, there exists a positive constant $\kappa$ such that 
if $\|M\|_{H^{9}(\mathbb R^2)} \leq \kappa$, 
then \eqref{Cattabriga} holds with $C_0$ independent of $\Omega$.
\end{lemma}
\begin{proof}
We may suppose that $\bar{\rho}=\hat{\mu}=\hat{p}=1$.
Indeed, suitable change of variables enables us to have this. 
Let us also set 
\[
\Omega_{R'}:=\Omega \cap B(0,R') \quad \text{for any $R'>1$}
\]
and then take a bounded domain $\tilde{\Omega}_{R'}$ 
whose boundary is $C^2$ such that 
\[
\Omega_{R'} \subset \tilde{\Omega}_{R'} \subset \Omega. 
\]
For any $\phi \in C_0^\infty(\tilde{\Omega}_{R'})$, define $\phi_{\tilde{\Omega}_{R'}}$ by
\[
\phi_{\tilde{\Omega}_{R'}}(x) 
:= \phi(x)-|\tilde{\Omega}_{R'}|^{-1}\int_{\tilde{\Omega}_{R'}} \phi(x) \,dx
\]
and $\psi$ by solving the following problem:
\begin{gather*}
\Delta \psi = \phi_{\Omega_{R'}} \quad \text{in $\tilde{\Omega}_{R'}$}, \quad
\nabla \psi \cdot {n}|_{\tilde{\Omega}_{R'}}=0, \quad
\int_{\tilde{\Omega}_{R'}} \psi \,dx=0,
\end{gather*}
where $n$ is the unit outer normal vector on $\tilde{\Omega}_{R'}$.
The paper \cite[Section 15]{ADN59} ensures that $\psi$ is well-defined, and that the following estimate holds:
\begin{equation}\label{PhiEs1}
\|\psi\|_{H^2(\tilde{\Omega}_{R'})} 
\lesssim_{\tilde{\Omega}_{R'}} \|\phi\|_{L^2(\tilde{\Omega}_{R'})}.
\end{equation}

From now on we show the Cattabriga estimate for $k=0$. 
Let us first show that
\begin{equation}\label{pEs1}
\|p_{\tilde{\Omega}_{R'}}\|_{L^2(\tilde{\Omega}_{R'})}
\lesssim_{{R'}} (\epsilon \|\nabla^2 u\|_{L^2(\Omega)}+C(\epsilon) \|\nabla u\|_{L^2(\Omega)}+\|g\|_{L^2(\Omega)}), \quad R'>1, \quad \epsilon\in(0,1),
\end{equation}
where $p_{\tilde{\Omega}_{R'}}:=p-|\tilde{\Omega}_{R'}|^{-1}\int_{\tilde{\Omega}_{R'}} p\,dx$.
We observe from $\int_{\tilde{\Omega}_{R'}} p_{\tilde{\Omega}_{R'}}\,dx=0$ 
and the definition of $\psi$ that
\begin{align*}
&\int_{\tilde{\Omega}_{R'}} p_{\tilde{\Omega}_{R'}} \phi \,dx
\\
&=\int_{\tilde{\Omega}_{R'}} p_{\tilde{\Omega}_{R'}} \phi_{\tilde{\Omega}_{R'}} \,dx
=\int_{\tilde{\Omega}_{R'}} p_{\tilde{\Omega}_{R'}} \Delta \psi \,dx 
=-\int_{\tilde{\Omega}_{R'}} (\nabla p) \cdot (\nabla \psi) \,dx 
=-\int_{\tilde{\Omega}_{R'}} (\Delta u +g) \cdot (\nabla \psi) \,dx
\\
&=-\sum_{j=1}^3\left\{\int_{\partial\tilde{\Omega}_{R'}}  (\partial_j \psi)\{(\nabla u_j)\cdot{n}\} \,dS 
-\int_{\tilde{\Omega}_{R'}} (\nabla u_j) \cdot (\nabla \partial_j \psi) \,dx \right\}
-\int_{\tilde{\Omega}_{R'}} g \cdot (\nabla \psi) \,dx.
\end{align*}
Then estimating the right hand side by \eqref{PhiEs1} leads to
\begin{align*}
&\int_{\tilde{\Omega}_{R'}} p_{\tilde{\Omega}_{R'}} \phi \,dx
\\
& \lesssim (\epsilon\|\nabla^2 u\|_{L^2(\Omega)} + C(\epsilon)\|\nabla u\|_{L^2(\Omega)})\|\psi\|_{H^2(\tilde{\Omega}_{R'})}
+\|\nabla u\|_{L^2(\Omega)}\|\psi\|_{H^2(\tilde{\Omega}_{R'})}+\|g\|_{L^2(\Omega)}\|\psi\|_{H^2(\tilde{\Omega}_{R'})}
\\
& \lesssim_{\tilde{\Omega}_{R'}} (\epsilon\|\nabla^2 u\|_{L^2(\Omega)} + C(\epsilon)\|\nabla u\|_{L^2(\Omega)}+\|g\|_{L^2(\Omega)})\|\phi\|_{L^2(\tilde{\Omega}_{R'})}.
\end{align*}
From the arbitrariness of $\phi \in C_0^\infty(\tilde{\Omega}_{R'})$,
we conclude \eqref{pEs1}.

We next show that 
\begin{equation}\label{LocalEs1}
\|\nabla^2 u\|_{L^2(\Omega_R)}+\|\nabla p\|_{L^2(\Omega_R)}
\lesssim_{\Omega_R} \|h\|_{H^1(\Omega)}+\|g\|_{L^2(\Omega)}+ \epsilon\|\nabla^2 u\|_{L^2(\Omega)} + C(\epsilon)\|\nabla u\|_{L^2(\Omega)}, \quad R>1,
\end{equation}
Noting that $\nabla p = \nabla p_{\tilde{\Omega}_{4R}}$, we have
\begin{equation*} 
\div u = h , \quad
-\Delta u + \nabla( p_{\tilde{\Omega}_{4R}} )=  g , \quad
u|_{\partial \Omega} =0 , \quad \lim_{|x| \rightarrow \infty} u =0.
\end{equation*}
Applying Theorem~IV.5.1 (see also Exercise~IV.5.2) in \cite{galdi} to the above problem,
we have 
\begin{align*}
\|\nabla^2 u\|_{L^2(\Omega_R)}+\|\nabla p\|_{L^2(\Omega_R)}
&\lesssim_{\Omega_R} \|h\|_{H^1(\Omega)} + \|g\|_{L^2(\Omega)} + \|u\|_{H^1(\Omega_{2R})} 
+ \|p_{\tilde{\Omega}_{4R}}\|_{L^2(\Omega_{2R})}
\\
&\lesssim_{\Omega_R} \|h\|_{H^1(\Omega)} + \|g\|_{L^2(\Omega)} + \epsilon\|\nabla^2 u\|_{L^2(\Omega)} + C_\epsilon\|\nabla u\|_{L^2(\Omega)},
\end{align*}
where we have also used the H\"older inequality, \eqref{sobolev1}, and \eqref{pEs1}
in deriving the last inequality. Hence, we conclude \eqref{LocalEs1}.

We complete the Cattabriga estimate for $k=0$ by deriving
an estimate over the domain $\Omega\backslash\Omega_R$.
To do this, we use the cut-off function $\chi_R(\cdot)=\chi(|\cdot|/R) \in C^\infty_0$, where $R>0$ and
\begin{equation}\label{cut-off1}
 \chi(s):= \left\{\begin{array}{ll}
  1 & \text{if $s\leq 1$},
  \\ 
  0 & \text{if $s\geq 2$}.
\end{array}\right. 
\end{equation}
Let us set
\[
M_{R/8}(x'):=(1-\chi_{R/8}(x'))M(x').
\]
For any $\delta \in (0,1)$, there exists $R_0=R_0(\delta)>1$ such that if $R/8>R_0$, then
\begin{equation}\label{smallness1}
 \|M_{R/8}\|_{W^{1,\infty}(\mathbb R^2)} < \delta, 
\quad
 M_{R/8}(x')=M(x') \ \ \text{for $(x,x') \in \Omega \backslash\Omega_{R/3}$}.
\end{equation}
We show that
\begin{multline}\label{farawayEs1}
\|\nabla^2 u\|_{L^2(\Omega\backslash\Omega_R)}+\|\nabla p\|_{L^2(\Omega\backslash\Omega_R)}
\\
\lesssim_{\Omega_R}\|h\|_{H^1(\Omega)}+\|g\|_{L^2(\Omega)}+ \epsilon\|\nabla^2 u\|_{L^2(\Omega)} + C_\epsilon\|\nabla u\|_{L^2(\Omega)}, \quad R> R_0(\delta_0),
\end{multline}
where $\delta_0$ is a constant to be determined later.
Multiplying the Stokes equation by the cut-off function $(1-\chi_{R/2}(x))$ 
and using the zero extension of  $(1-\chi_{R/2}) u$ and $(1-\chi_{R/2}) p_{\tilde{\Omega}_{4R}}$ 
on $\Omega'_{R/8}:=\{x_1>M_{R/8}(x')\}$, we see that
\begin{gather*} 
\div ((1-\chi_{R/2}) u) = h^*, \quad
-\Delta ((1-\chi_{R/2}) u) + \nabla( (1-\chi_{R/2}) p_{\tilde{\Omega}_{4R}} ) = g^* \quad 
\text{in $\Omega'_{R/8}$}, 
\\
 (1-\chi_{R/2})u|_{\partial \Omega'_{R/8}} =0, \quad 
\lim_{|x| \rightarrow \infty} (1-\chi_{R/2}) u =0, 
\end{gather*} 
where 
\begin{align*}
h^* (x)& := (1-\chi_{R/2}) h + (\nabla \chi_{R/2})\cdot u, \\
g^* (x)& := (1-\chi_{R/2}) g  -(\Delta \chi_{R/2}) u-\nabla u \nabla \chi_{R/2} + p_{\tilde{\Omega}_{4R}} \nabla \chi_{R/2}.
\end{align*}
Furthermore, using the change of variables
\[
 x_1 = y_1 + M_{R/8}(y_2, y_3), \quad
 x_2 = y_2, \quad
 x_3 = y_3,
 \]
we have the problem
\begin{gather*} 
 {\div}_y ((1-\chi_{R/2}) u) = h^* + \{\partial_{y_1} ((1-\chi_{R/2}) u)\}\cdot \nabla_y M_{R/8}, \\
\begin{aligned}
& -\Delta_y ((1-\chi_{R/2}) u) + \nabla_y( (1-\chi_{R/2}) p_{\tilde{\Omega}_{4R}} ) \\
& \quad =g^* \!+ \sum_{j=2}^3\!\left[((\partial_{y_j}M_{R/8})\partial_{y_1}-\partial_{y_j})\{(\partial_{y_j}M_{R/8})\partial_{y_1} ((1-\chi_{R/2}) u)\}
\!-\!(\partial_{y_j}M_{R/8})\{\partial_{y_1y_j} ((1-\chi_{R/2}) u)\} \!\right] \\
& \quad \qquad  + \{\partial_{y_1} ((1-\chi_{R/2}) p_{\tilde{\Omega}_{4R}})\}\nabla_y M_{R/8}  \quad \text{in $\mathbb R_+^3$},
\end{aligned}
\\
(1-\chi_{R/2})u|_{\partial \mathbb R_+^3} =0, \quad
\lim_{|y| \rightarrow \infty} (1-\chi_{R/2}) u =0.
\end{gather*} 
Applying Theorem~IV.3.2 in \cite{galdi} with \eqref{smallness1} to the above problem, we have
\begin{align*}
&\|\nabla^2 ((1-\chi_{R/2}) u)\|_{L^2(\mathbb R_+^3)}
+\|\nabla ((1-\chi_{R/2}) p_{\tilde{\Omega}_{4R}})\|_{L^2(\mathbb R_+^3)}
\\
& \lesssim 
\delta\|\nabla^2 ((1-\chi_{R/2}) u)\|_{L^2(\mathbb R_+^3)}+\delta\|\nabla ((1-\chi_{R/2}) p_{\tilde{\Omega}_{4R}})\|_{L^2(\mathbb R_+^3)}
\\
&\qquad + \|h^*\|_{H^1({\rm supp}(1-\chi_{R/2}))} + \|g^*\|_{L^2({\rm supp}(1-\chi_{R/2}))}
+\|u\|_{L^2({\rm supp}\nabla\chi_{R/2})}+\|\nabla u\|_{L^2({\rm supp}(1-\chi_{R/2}))}.
\end{align*}
Let us now take $\delta_0$ so small that
\begin{align*}
&\|\nabla^2 ((1-\chi_{R/2}) u)\|_{L^2(\mathbb R_+^3)}
+\|\nabla ((1-\chi_{R/2}) p_{\tilde{\Omega}_{4R}})\|_{L^2(\mathbb R_+^3)}
\\
& \lesssim \|h^*\|_{H^1({\rm supp}(1-\chi_{R/2}))} + \|g^*\|_{L^2({\rm supp}(1-\chi_{R/2}))}
+\|u\|_{L^2({\rm supp}\nabla\chi_{R/2})}+\|\nabla u\|_{L^2({\rm supp}(1-\chi_{R/2}))}.
\end{align*}
Then changing the coordinate $y \in \mathbb R_+^3$ to the coordinate $x \in \Omega_{R/8}'$ and noting that $ {\rm supp}(1-\chi_{R/2}) \subset \Omega$ and $(1-\chi_{R/2})(x)=1$ hold for 
$x \in \Omega\backslash\Omega_{R} \subset \Omega_{R/8}'$, we have 
\begin{align*}
&\|\nabla^2 u\|_{L^2(\Omega\backslash\Omega_R)}+\|\nabla p\|_{L^2(\Omega\backslash\Omega_R)}
\lesssim_{\Omega}  
\|h^*\|_{H^1(\Omega)} + \|g^*\|_{L^2(\Omega)}+
\|u\|_{L^2({\rm supp}\nabla\chi_{R/2})}+\|\nabla u\|_{L^2(\Omega)}.
\end{align*}
Then estimating the right hand side by \eqref{sobolev1}, \eqref{pEs1}, and Poincar\'e inequality,
we obtain 
\begin{align*}
\|\nabla^2 u\|_{L^2(\Omega\backslash\Omega_R)}+\|\nabla p\|_{L^2(\Omega\backslash\Omega_R)}
&\lesssim_{\Omega}  \|h\|_{H^1(\Omega)} + \|g\|_{L^2(\Omega)}+
\|(u,p_{\tilde{\Omega}_{4R}})\|_{L^2({\rm supp}\nabla\chi_{R/2})}+\|\nabla u\|_{L^2(\Omega)}
\\
&\lesssim_{\Omega} \|h\|_{H^1(\Omega)} + \|g\|_{L^2(\Omega)} + \epsilon\|\nabla^2 u\|_{L^2(\Omega)} + C_\epsilon\|\nabla u\|_{L^2(\Omega)}.
\end{align*}
Hence, we conclude \eqref{farawayEs1}.

From \eqref{LocalEs1} and \eqref{farawayEs1}, 
we have \eqref{Cattabriga} with $k=0$ by taking $\epsilon$ small enough. 
Furthermore, one can show inductively for the case $k=1, 2$ with aid of 
Theorem~IV.3.2 and Theorem~IV.5.1 in \cite{galdi} which discusses the estimate of higher order derivatives.

\medskip
 
We next discuss the case $\|M\|_{H^{9}(\mathbb R^2)}\ll1$.
Using \eqref{CV1}, we have the following problem:
\begin{gather*} 
{\div}_y u = h  + \partial_{y_1}u\cdot \nabla_y M, \\
\begin{aligned}
-\Delta_y u + \nabla_y p  
&=g + \sum_{j=2}^3\left[\{(\partial_{y_j}M)\partial_{y_1}-\partial_{y_j}\}\{(\partial_{y_j}M)\partial_{y_1}u\} -(\partial_{y_j}M)\partial_{y_1y_j}  u \right] 
+ \partial_{y_1}  p \nabla_y M \quad \text{in $\mathbb R_+^3$},
\end{aligned}
\\
u|_{\partial \mathbb R_+^3} =0 , \quad
\lim_{|y| \rightarrow \infty} u =0 . 
\end{gather*} 
Applying Theorem~IV.3.2 in \cite{galdi} with $\|M\|_{H^{9}(\mathbb R^2)} \leq \kappa$ to the above problem, we have
\begin{align*}
&\|\nabla^2 u\|_{L^2(\mathbb R_+^3)}
+\|\nabla p \|_{L^2(\mathbb R_+^3)}
\\
& \lesssim 
\kappa\|\nabla^2 u\|_{L^2(\mathbb R_+^3)}+\kappa\|\nabla p \|_{L^2(\mathbb R_+^3)}
+ \|h\|_{H^1(\mathbb R_+^3)} + \|g\|_{L^2(\mathbb R_+^3)}
+\|\nabla u\|_{L^2(\mathbb R_+^3)}. 
\end{align*}
Let us take $\kappa$ so small that
\begin{align*}
&\|\nabla^2 u\|_{L^2(\mathbb R_+^3)}
+\|\nabla p \|_{L^2(\mathbb R_+^3)}
\lesssim \|h\|_{H^1(\mathbb R_+^3)} + \|g\|_{L^2(\mathbb R_+^3)}
+\|\nabla u\|_{L^2(\mathbb R_+^3)}.
\end{align*}
Then changing the coordinate $y \in \mathbb R_+^3$ to the coordinate $x \in \Omega$ 
we conclude that 
\begin{align*}
\|\nabla^2 u\|_{L^2(\Omega)}+\|\nabla p\|_{L^2(\Omega)}
\lesssim \|h\|_{H^1(\Omega)} + \|g\|_{L^2(\Omega)} + \|\nabla u\|_{L^2(\Omega)}.
\end{align*}
Furthermore, one can show inductively for the case $k=1, 2$ with aid of 
Theorem~IV.5.1 in \cite{galdi} which discusses the estimate of higher order derivatives.
\end{proof}

\begin{lemma} \label{ellipticEst}
\textbf{(Elliptic Estimate)}
Consider the following elliptic system 
\begin{equation*} 
\begin{split}
 - \hat{\mu} \Delta u - \hat{\nu} \nabla \div u = f, \quad  
 u|_{\partial \Omega} =0 , \quad \lim_{|x| \rightarrow \infty} u =0 .  
\end{split}
\end{equation*}
with $\hat{\mu}$ and $\hat{\nu}$ being positive constants. For $k= 0,1,2$ and $f \in H^{k}(\Omega)$,
if $u \in H^{k+2}(\Omega)$ is a solution to the elliptic system, then it holds that
\begin{equation}\label{Elliptic}
\| \nabla^{k+2} u \|_{L^2(\Omega)}  \lesssim \|f \|_{H^{k}(\Omega)} + \|  u \|_{L^2(\Omega)}  .  
\end{equation}
\end{lemma}
\begin{proof}
This can be shown in much the same way as Theorems 4 and 5 in Section 6.3 in \cite{Ev}.
\end{proof}

\section{Initial data}\label{B}

We find a certain initial data $\Phi_0^{\#}$ 
which satisfies the compatibility conditions \eqref{cmpa0}.

\begin{lemma}\label{CompatibilityCond}
There exists $\psi_0^{\#}\in H^5(\Omega)$ such that 
$\Phi_0^{\#}=(0,\psi_0^{\#})$ satisfies \eqref{cmpa0} and
$\|\Phi_0^{\#}\|_{\lteasp{\beta}}+\|\Phi_0^{\#}\|_{H^5} \lesssim \delta$.
\end{lemma}
\begin{proof}
Note that problem \eqref{eq-pv} over $\Omega$ is equivalent
to problem \eqref{eq-pv-hat} over ${\mathbb R^3_+}$.
To complete the proof, let us consider problem \eqref{eq-pv-hat}.
It suffices to find
the data $\hat{\Phi}_0^{\#}(y)=(0,\chi(y_1)\hat{\psi}_0(y))\in H^5({\mathbb R^3_+})$
of which $\hat{\psi}_0$ satisfies 
\begin{subequations} \label{cmpa4}
\begin{gather}
\hat{\psi}_0|_{y_1=0} = 0,
\\
\left\{
\hat{\rho}_0 (\hat{u}_0 \cdot \hat{\nabla}) \hat{\psi}_0
- \hat{L} \hat{\psi}_0
- (\hat{g}+\hat{G})|_{t=0}
\right\}|_{y_1 = 0}
= 0,
\label{cmpa2nd}\\
\left[\partial_t\left\{
\hat{\rho} (\hat{u} \cdot \hat{\nabla}) \hat{\psi}
- \hat{L} \hat{\psi}
+ p'(\hat{\rho}) \hat{\nabla} \hat{\vp}
- \hat{g}
\right\}|_{t=0}\right]_{y_1 = 0}
= 0,
\label{cmpa3rd}\\
\|\hat{\psi}_0\|_{H^5} \lesssim \delta,
\end{gather}
\end{subequations}
where the cut-off function $\chi$ is defined in \eqref{cut-off1}.
Indeed, we see from the first three conditions that $\hat{\Phi}_0^{\#}$ satisfies \eqref{cmpa3}.
The last condition implies $\|\Phi_0^{\#}\|_{\lteasp{\beta}}+\|\Phi_0^{\#}\|_{H^5} \lesssim \delta$.

We will apply an extension theorem \cite[Theorem~2.5.7]{He}.
To do so, let us first fix the boundary values of the zeroth, first, and third derivatives 
with respect to $y_1$ of $\hat{\psi}_0$ as
\begin{equation}\label{0thCond}
\hat{\psi}_0|_{y_1=0} 
= (\partial_{y_1}\hat{\psi}_0)|_{y_1=0}
= (\partial_{y_1}^3\hat{\psi}_0)|_{y_1=0}
= 0. 
\end{equation}
Next we determine the boundary value of $\partial_{y_1}^2\hat{\psi}_0$
from the compatibility condition of order 2.
Using \eqref{0thCond}, we simplify \eqref{cmpa2nd} as
\begin{gather}
\left\{
- {\cal A}\partial_{y_1}^2\hat{\psi}_0
- \hat{G}
\right\}|_{y_1 = 0}
= 0, 
\label{cmpa6}\\
{\cal A}(y_2,y_3):=\mu_1(1+|\nabla M|^2)I+
(\mu_1+\mu_2){\cal B}, \quad
{\cal B}(y_2,y_3):=
\begin{bmatrix}
1 & M_{y_2} & M_{y_3}
\\
M_{y_2} & (M_{y_2})^2 & M_{y_2}M_{y_3}
\\
M_{y_3} & M_{y_2}M_{y_3} & (M_{y_3})^2
\end{bmatrix}.
\notag 
\end{gather}
Since $\cal A$ is nonsingular, we see from \eqref{cmpa6} that the boundary value of
$\partial_{y_1}^2\hat{\psi}_0$ must be 
\begin{equation}\label{2ndCond}
(\partial_{y_1}^2\hat{\psi}_0)|_{y_1 = 0}
= ({\cal A}^{-1} {\hat{G}}) |_{y_1 = 0}.
\end{equation}

Now let us determine the boundary value of $\partial_{y_1}^4\hat{\psi}_0$.
Using $\hat{\vp}_t|_{t=0, \, y_1 = 0}=0$ which comes from
\eqref{eq-pv1-hat} and \eqref{0thCond}, we simplify  \eqref{cmpa3rd} as
\begin{gather}
\left[\left\{
\hat{\rho}_0 \{u_b\cdot \nabla(x_1-M) \}\partial_{y_1} \hat{\psi}_{t}
- {\cal A} \partial_{y_1}^2 \hat{\psi}_t
+ p'(\hat{\rho}_0) \hat{\nabla} \hat{\vp}_t
\right\}|_{t=0}\right]_{y_1 = 0}
= 0.
\label{cmpa5}
\end{gather}
We compute necessary conditions for 
$(\hat{\nabla}\hat{\vp}_t)|_{t=0,\,y_1=0}$,
$(\partial_{y_1}\hat{\psi}_t)|_{t=0,\,y_1=0}$,
and $(\partial_{y_1}^2\hat{\psi}_t)|_{t=0,\,y_1=0}$.
Applying $\hat{\nabla}$ to \eqref{eq-pv1-hat} gives
\[
(\hat{\nabla}\hat{\vp}_t)|_{t=0,\,y_1=0}
=-(\hat{\rho}_0  {\cal B} \partial_{y_1}^2 \hat{\psi}_0) |_{y_1 = 0}
=-(\hat{\rho}_0  {\cal B} {\cal A}^{-1} {\hat{G}}) |_{y_1 = 0}.
\]
Furthermore, applying $\partial_{y_1}$ to \eqref{eq-pv2-hat} leads to
\begin{align*}
(\partial_{y_1}\hat{\psi}_t)|_{t=0,\,y_1=0}
&=\left[ -\{u_b\cdot \nabla(x_1-M) \} \partial_{y_1}^2 \hat{\psi}_0 
+\hat{\rho}_0^{-1} \hat{L}\partial_{y_1}\hat{\psi}_0
+\hat{\rho}_0^{-1} \partial_{y_1}\hat{G}\right]|_{y_1 = 0}.
\end{align*}
Apply $\partial_{y_1}^2$ to \eqref{eq-pv2-hat} and 
use \eqref{0thCond} and $(\tilde{\rho} \tilde{u}_1)_{y_1}=0$ to obtain
\begin{align*}
&(\partial_{y_1}^2\hat{\psi}_t)|_{t=0,\,y_1=0}
-(\hat{\rho}_0^{-1}{\cal A}\partial_{y_1}^4\hat{\psi})|_{y_1 = 0}
\\
&=\biggl\{-\sum_{i=2}^3 u_{bi}\partial_{y_i}\partial_{y_1}^2\hat{\psi}_0
-2\hat{\rho}_0^{-1}\{ \partial_{y_1} (\rho_0 U) \cdot \nabla(x_1-M) \} \partial_{y_1}^2 \hat{\psi}_0
\\
&\qquad  +\hat{\rho}_0^{-1} (\hat{L}-{\cal A}\partial_{y_1}^2)\partial_{y_1}^2\hat{\psi}_0
+\hat{\rho}_0^{-1} \partial_{y_1}^2(\hat{g}+\hat{G})\biggr\}\biggr|_{y_1 = 0}.
\end{align*}
Note that $\hat{L}-{\cal A}\partial_{y_1}^2$ does not have the second derivative operator $\partial_{y_1}^2$.
Plugging these into \eqref{cmpa5}, we see that the boundary value of $\partial_{y_1}^4\hat{\psi}_0$ must be 
\begin{align}\label{4thCond}
&(\partial_{y_1}^4\hat{\psi}_0)|_{y_1 = 0}
\notag\\
&=({\cal A}^{-1})^2\biggl[
\{-\hat{\rho}_0^2p'(\hat{\rho}_0){\cal B} {\cal A}^{-1} {\hat{G}}\}|_{y_1 = 0}
\notag\\
& \qquad \qquad \quad \left.+\hat{\rho}_0 \{u_b\cdot \nabla(x_1-M) \}(-\hat{\rho}_0\{u_b\cdot \nabla(x_1-M) \}\partial_{y_1}^2 \hat{\psi}_0 
+\hat{L}\partial_{y_1}\hat{\psi}_0
+\partial_{y_1}\hat{G})\right\}|_{y_1 = 0} 
\notag \\
& \qquad \qquad \quad 
+{\cal A}\biggl\{\sum_{i=2}^3 \hat{\rho}_0u_{bi}\partial_{y_i}\partial_{y_1}^2\hat{\psi}_0
+2\{ \partial_{y_1} (\rho_0 U) \cdot \nabla(x_1-M) \} \partial_{y_1}^2 \hat{\psi}_0
\notag \\
& \qquad \qquad \qquad \qquad 
-(\hat{L}-{\cal A}\partial_{y_1}^2)\partial_{y_1}^2\hat{\psi}_0
-\partial_{y_1}^2(\hat{g}+\hat{G})\biggr\}\biggr|_{y_1=0} \biggr].
\end{align}
We notice that the right hand side  
can be expressed by a linear combination of $\hat{G}$ and its derivatives
with some coefficients given by the smooth functions 
$\hat{\rho}_0(=\rt)$, $u_b$, ${\cal A}$, ${\cal B}$, ${\cal A}^{-1}$, $\nabla M$, $\nabla^2 M$, $\partial_{y_1}(\hat{\rho}_0U)$
if we write explicitly $\partial_{y_1} \hat{\psi}_0|_{y_1 = 0}$, 
$\partial_{y_1}^2 \hat{\psi}_0|_{y_1 = 0}$,
and $\partial_{y_1}^3 \hat{\psi}_0|_{y_1 = 0}$ by using \eqref{0thCond} and \eqref{2ndCond}.

Using an extension theorem \cite[Theorem~2.5.7]{He} with \eqref{0thCond}, \eqref{2ndCond}, and \eqref{4thCond}, we have a function $\hat{\psi}_0$ satisfies  \eqref{cmpa4}.
Indeed, the first three lines in \eqref{cmpa4} obviously follow from the above computations
of the compatibility conditions.
The last line in \eqref{cmpa4} can be also obtained by using the fact that
all derivatives with respect to $y_1$ of $\hat{\psi}_0$ are linear combinations of $\hat{G}$ 
and its derivatives whose Sobolev norms are estimated by $C \delta$.
The proof is complete.
\end{proof}

\end{appendix}





\begin{thebibliography}{10}

\bibitem{ADN59}
{S.~Agmon, A.~Douglis and L.~Nirenberg}, \newblock
{Estimates near the boundary for solutions of
elliptic partial differential equations 
satisfying general boundary conditions. I},  \newblock
\emph{Comm.\ Pure Appl.\ Math.} \textbf{12} (1959) 623--727.

\bibitem{Ev}
{L.C.~Evans}, \newblock
\emph{Partial differential equations, Second edition}, \newblock 
{Graduate Studies in Mathematics} \textbf{19}, 
American Mathematical Society, Providence, 2010.

\bibitem{galdi}
G.~P. Galdi, \emph{An introduction to the mathematical theory of the
  {N}avier--{S}tokes equations}, vol.~1, {S}pringer--{V}erlag {N}ew {Y}ork,
  1994.

\bibitem{He}
{L.~H\"ormander}, \emph{Linear partial differential operators}, 
Die Grundlehren der Mathematischen Wissenschaften, Bd. \textbf{116}
Springer-Verlag, Berlin-G\"ottingen-Heidelberg, 1963


\bibitem{kg06-loc}
Y.~Kagei and S.~Kawashima, \emph{Local solvability of an initial boundary value
  problem for a quasilinear hyperbolic-parabolic system}, J. Hyperbolic Differ.
  Equ. \textbf{3} (2006), no.~2, 195--232. \MR{MR2229854 (2007c:35120)}

\bibitem{kg06}
Y.~Kagei and S.~Kawashima, \emph{Stability of planar stationary solutions to
  the compressible {N}avier-{S}tokes equation on the half space}, Comm. Math.
  Phys. \textbf{266} (2006), no.~2, 401--430. \MR{MR2238883 (2007g:35186)}

\bibitem{kagei05}
Y.~Kagei and T.~Kobayashi, \emph{Asymptotic behavior of solutions of the
  compressible {N}avier-{S}tokes equations on the half space}, Arch. Ration.
  Mech. Anal. \textbf{177} (2005), no.~2, 231--330. \MR{MR2188049
  (2006i:76088)}


\bibitem{knz03}
S.~Kawashima, S.~Nishibata, and P.~Zhu, \emph{Asymptotic stability of the
  stationary solution to the compressible {N}avier-{S}tokes equations in the
  half space}, Comm. Math. Phys. \textbf{240} (2003), no.~3, 483--500.
  \MR{MR2005853 (2004g:76103)}


\bibitem{matsu81}
A.~Matsumura, \emph{An energy method for the equations of motion of
  compressible viscous and heat-conductive fluids}, University of
  Wisconsin-Madison, MRC Technical Summary Report \#2194 (1981), 1--16.

\bibitem{matu-01}
A.~Matsumura, \emph{Inflow and outflow problems in the half space for a
  one-dimensional isentropic model system of compressible viscous gas}, Methods
  Appl. Anal. \textbf{8} (2001), no.~4, 645--666, IMS Conference on
  Differential Equations from Mechanics (Hong Kong, 1999). \MR{MR1944189
  (2004d:76087)}

\bibitem{m-n83}
A.~Matsumura and T.~Nishida, \emph{Initial-boundary value problems for the
  equations of motion of compressible viscous and heat-conductive fluids},
  Comm. Math. Phys. \textbf{89} (1983), no.~4, 445--464. \MR{MR713680
  (84h:35137)}


\bibitem{nny07}
T.~Nakamura, S.~Nishibata, and T.~Yuge, \emph{Convergence rate of solutions
  toward stationary solutions to the compressible {N}avier-{S}tokes equation in
  a half line}, J. Differential Equations \textbf{241} (2007), no.~1, 94--111.
  \MR{MR2356211}

\bibitem{nn09}
T.~Nakamura and S.~Nishibata, \emph{Convergence rate toward planar stationary waves for compressible viscous fluid in multidimensional half space}, SIAM J. Math. Anal. \textbf{41} (2009), no.~5, 1757--1791. 



\bibitem{Rac}
{R.~Racke}, \newblock  
{\it Lectures on nonlinear evolution equations. Initial value problems}, \newblock
Aspects of Mathematics, E19. Friedr. Vieweg \& Sohn, Braunschweig 1992.


\bibitem{Va83}
A.~Valli, \emph{Periodic and stationary solutions for compressible Navier-Stokes equations via a stability method}, Ann. Scuola Norm. Sup. Pisa Cl. Sci. (4) \textbf{10} (1983), no.~4, 607--647. 


\end{thebibliography}

\providecommand{\bysame}{\leavevmode\hbox to3em{\hrulefill}\thinspace}
\providecommand{\MR}{\relax\ifhmode\unskip\space\fi MR }

\end{document}